\documentclass{article}
\usepackage[utf8]{inputenc}

\usepackage{amsmath}
\usepackage{comment}
\usepackage{leftidx}
\usepackage{multicol}
\usepackage{amssymb}
\usepackage{stmaryrd}
\usepackage{accents}

\usepackage{xfrac}
\usepackage{amsthm}
\usepackage{xcolor}
\usepackage{sectsty}
\usepackage{relsize}

\usepackage{indentfirst}
\usepackage{tikz}
\tikzset{shorten <>/.style={shorten >=#1,shorten <=#1}}

\usepackage{tikz-cd}
\newcounter{nodemaker}
\tikzset{%
    symbol/.style={%
        draw=none,
        every to/.append style={%
            edge node={node [sloped, allow upside down, auto=false]{$#1$}}}
    }
}
 \usepackage{quiver}
\makeatletter
\newcommand{\bigdoublevee}{\big@doubleop{\bigvee}}
\newcommand{\bigdoublewedge}{\big@doubleop{\bigwedge}}
\newcommand{\big@doubleop}[1]{%
  \DOTSB\mathop{\mathpalette\big@doubleop@aux{#1}}\slimits@
}

\tikzset{Rightarrow/.style={double equal sign distance,>={Implies},->},
triple/.style={-,preaction={draw,Rightarrow}},
quadruple/.style={preaction={draw,Rightarrow,shorten >=0pt},shorten >=1pt,-,double,double
distance=0.2pt}}

\newcommand\big@doubleop@aux[2]{%
  \sbox\z@{$\m@th#1#2$}%
  \makebox[1.35\wd\z@][s]{$\m@th#1#2\hss#2$}%
}
\makeatother

\makeatletter
\newcommand*{\doublerightarrow}[2]{\mathrel{
  \settowidth{\@tempdima}{$\scriptstyle#1$}
  \settowidth{\@tempdimb}{$\scriptstyle#2$}
  \ifdim\@tempdimb>\@tempdima \@tempdima=\@tempdimb\fi
  \mathop{\vcenter{
    \offinterlineskip\ialign{\hbox to\dimexpr\@tempdima+1em{##}\cr
    \rightarrowfill\cr\noalign{\kern.5ex}
    \rightarrowfill\cr}}}\limits^{\!#1}_{\!#2}}}
\newcommand*{\triplerightarrow}[1]{\mathrel{
  \settowidth{\@tempdima}{$\scriptstyle#1$}
  \mathop{\vcenter{
    \offinterlineskip\ialign{\hbox to\dimexpr\@tempdima+1em{##}\cr
    \rightarrowfill\cr\noalign{\kern.5ex}
    \rightarrowfill\cr\noalign{\kern.5ex}
    \rightarrowfill\cr}}}\limits^{\!#1}}}
\makeatother

\DeclareFontFamily{U}{min}{}
\DeclareFontShape{U}{min}{m}{n}{<-> udmj30}{}

\DeclareUnicodeCharacter{3088}{\yo}

\usepackage{geometry}
\newtheorem{theorem}{Theorem}[section]
\theoremstyle{proposition}
\newtheorem{proposition}[theorem]{Proposition}
\newtheorem{corollary}[theorem]{Corollary}
\newtheorem{corollary'}[theorem]{Corollary}
\newtheorem{lemma}[theorem]{Lemma}
\theoremstyle{definition}
\newtheorem{definition}[theorem]{Definition}

\theoremstyle{remark}
\newtheorem*{remark}{Remark}

\sectionfont{\large}
\subsectionfont{\normalsize}

\usepackage[backend=biber,
sorting=nty
]{biblatex}
\nocite{*} 

\newcommand{\Spec}% spectrum
 {{\bf Spec}}
 
 \newcommand{\Alg}% subobject lattice
 {{\rm Alg}}
 
 \newcommand{\Lex}% Lex
 {{\bf Lex}}
 
 \newcommand{\LFP}% LFP
 {{\bf LFP}}
 
 \newcommand{\Et}% etale map
 {{\bf Et}}
 
\newcommand{\Loc}% local map
 {{\bf Loc}}
 
  \newcommand{\Diag}% diagonally universal morphisms
 {{\bf Diag}}
 
 \newcommand{\GTop}% Topos
 {{\bf GTop}}
 
 \newcommand{\Geom}% Topos
 {{\bf Geom}}
 
 \newcommand{\Cat}% cat
 {{\bf Cat}}

\newcommand{\Fib}% fibration
 {{\bf Fib}}

\newcommand{\DFib}% discrete fibration
 {{\bf DFib}}

\newcommand{\DopFib}% discrete opfibration
 {{\bf DopFib}}
 
 \newcommand{\opFib}% discrete opfibration
 {{\bf opFib}}

\newcommand{\cod}% codomain
 {{\rm cod}}

\newcommand{\comp}% composition
 {\circ}

\newcommand{\Cont}% category of continuous G-sets
 {{\bf Cont}}

\newcommand{\dom}% domain
 {{\rm dom}}
 
\newcommand{\fp}% fp
{{\rm fp }}

\newcommand{\li}% limit
{{\textup{lim } }}

\newcommand{\oplaxlim}% limit
{{\textup{oplaxlim } }}

\newcommand{\oplaxcolim}% limit
{{\textup{oplaxcolim } }}

\newcommand{\lan}% Lan
{{\textup{lan } }}

\newcommand{\ran}% limit
{{\textup{ran } }}

\newcommand{\colim}% colimit
{{\textup{colim}}}

\newcommand{\coeq}% coeq
{{\textup{coeq}}}

\newcommand{\eq}% eq
{{\textup{eq}}}

\newcommand{\comma}[2]% comma object
{\mbox{$(#1\!\downarrow\!#2)$}}

\newcommand{\empstg}% empty string
 {[\,]}

\newcommand{\epi}% epimorphism
 {\twoheadrightarrow}

\newcommand{\hy}% hyphen (in math mode)
 {\mbox{-}}

\newcommand{\im}% image
 {{\rm im}}

\newcommand{\imp}% implication
 {\!\Rightarrow\!}

\newcommand{\Ind}% ind
 {{\rm Ind}}
 
 \newcommand{\Pro}% ind
 {{\rm Pro}}

\newcommand{\mono}% monomorphism 
 {\rightarrowtail}

\newcommand{\ob}% class of objects
 {{\rm ob}}
 
 \newcommand{\Hom}% class of objects
 {{\rm Hom}}

\newcommand{\op}% opposite category
 {^{\rm op}}
 
 \newcommand{\pt}% points
 {{\bf pt}}

\newcommand{\Set}% category of sets
 {{\bf Set }}

\newcommand{\Sh}% category of sheaves
 {{\bf Sh}}

\newcommand{\sh}% category of sheaves
 {{\bf sh}}

\newcommand{\Sub}% subobject lattice
 {{\rm Sub}}
 
\usepackage{hyperref}
\hypersetup{
    colorlinks=true,
    linkcolor=blue,
    filecolor=magenta,      
    urlcolor=cyan,
}

\usepackage{cleveref}

\title{On coslices and commas of locally finitely presentable categories}
\author{Axel Osmond}
\date{March 2021}

\begin{document}

\maketitle

\begin{abstract}
    We give an explicit description of the generator of finitely presented objects of the coslice of a locally finitely presentable category under a given object, as consisting of all pushouts of finitely presented maps under this object. Then we prove that the comma category under the direct image part of a morphism of locally finitely presentable category is still locally finitely presentable, and we give again an explicit description of its generator of finitely presented objects. We finally deduce that 2-category $\LFP$ has comma objects computed in $\Cat$. 
\end{abstract}

\section*{Introduction}

In this note we investigate in detail the comma and coslice construction in the context of locally finitely presentable categories.\\

It is well known since \cite{adamek1994locally} that coslices of locally finitely presentable categories are again locally finitely presentable categories; however the proof involves an abstract characterization from which one cannot retrieve an explicit generator of finitely presented objects. However, in the context of the \emph{spectral construction} arose a similar problem of characterising finitely presentable left maps under a given object - which are used to construct the \emph{spectral site} associated to this object - and this led to ask in particular what could be a good choice of generator in a coslice. After discussing this topics at several occasion with members of the category theory community, it appeared that we were still lacking a reference containing the explicit result, though elements of answer were already present in two previous works, \cite{Anel} and \cite{Coste}, and this kind of result might have been folklore for a few connoisseurs. We though it could be useful to provide an answer as detailed and concrete as possible to fix this lack of reference, which is the purpose of first half of this work. \\

The second half of this paper is devoted to the construction of comma categories in the 2-category $\LFP$ of locally finitely presentable categories. Concretely, we prove that if $ F : \mathcal{A} \rightarrow \mathcal{B}$ is a morphism of locally finitely presentable categories, then the comma $ F_*\downarrow \mathcal{B}$ is locally finitely presentable and define together with its canonical projections a comma object in $\LFP$. It seems that until now the only elements of answer on this topic was \cite{makkai1988strong}[Proposition 6.1.1] and a subsequent adaptation in \cite{adamek1994locally}, which however was a weaker statement only proving accessibility of comma object without any control of the rank of accessibility. We improve this result by proving that in the situation above the comma construction needs not raising the rank of accessibility, thanks to a characterization of finitely presented objects of the comma. Moreover, we prove that this comma object calculated in $\Cat$ actually inherits the universal property of the comma object in $\LFP$. \\

%Moreover this prove that $\LFP$ inherits the so called \emph{comprehensive factorization} of Street and Walter. This question arose in the context of an ongoing work on the construction of a \emph{2-dimensional geometry} for Gabriel-Ulmer duality, were the induced comprehensive factorization will provide a dual factorization system on $\Lex$ and the subsequent factorization data of a 2-categorical geometry. \\

Generalities about coslices and few technicalities are gathered in the first part of this work. 

\section*{Acknowledgements}

The author would like to express special thanks to Mathieu Anel for his help on this topic and several important discussions, and also useful remarks on a previous draft of this work; he should also emphasize the key influence of \cite{Anel} on this work. He is also specially grateful to Ivan Di Liberti for several helpful and stimulating discussions on this problem. He also thanks Paul-André Melliès, Clemens Berger and Jiri Adamek for past discussions on this topic.   

\section{Generalities about coslices and comma categories}

In this section we first recall very basic facts about coslice categories.

\begin{definition}
Let be $\mathcal{A}$ a category; the \emph{coslice} of $\mathcal{A}$ at an object $A$ is the category $ A\downarrow \mathcal{A}$ whose\begin{itemize}
    \item objects are arrows $ f : A \rightarrow B$ in $\mathcal{A}$
    \item morphisms $ f_1 \rightarrow f_2$ are all arrows $g$ inscribed in a triangle as below 
    % https://q.uiver.app/?q=WzAsMyxbMSwwLCJBIl0sWzAsMSwiQl8xIl0sWzIsMSwiQl8yIl0sWzAsMSwiZl8xIiwyXSxbMSwyLCJnIiwyXSxbMCwyLCJmXzIiXV0=
\[\begin{tikzcd}
	& A \\
	{B_1} && {B_2}
	\arrow["{f_1}"', from=1-2, to=2-1]
	\arrow["g"', from=2-1, to=2-3]
	\arrow["{f_2}", from=1-2, to=2-3]
\end{tikzcd}\]
\end{itemize}
We denote as $ \cod : A \downarrow \mathcal{A} \rightarrow \mathcal{A}$ the \emph{codomain functor} sending an arrow $ f : A \rightarrow B$ on its codomain $B$ and a morphism $ g : f_1 \rightarrow f_2$ on the underlying arrow $ g : B_1 \rightarrow B_2$. 
\end{definition}

Coslices are related to the following notion:

\begin{definition}
For a category $\mathcal{A}$, the \emph{arrow category} $ \mathcal{A}^2$ is the category whose\begin{itemize}
    \item objects are arrow $ f : A \rightarrow B$
    \item morphisms $f_1 \rightarrow f_2$ are pairs $(g,g')$ inscribed in a square as below
    % https://q.uiver.app/?q=WzAsNCxbMCwwLCJBXzEiXSxbMCwxLCJCXzEiXSxbMSwxLCJCXzIiXSxbMSwwLCJBXzIiXSxbMCwxLCJmXzEiLDJdLFsxLDIsImcnIiwyXSxbMywyLCJmXzIiXSxbMCwzLCJnIl1d
\[\begin{tikzcd}
	{A_1} & {A_2} \\
	{B_1} & {B_2}
	\arrow["{f_1}"', from=1-1, to=2-1]
	\arrow["{g'}"', from=2-1, to=2-2]
	\arrow["{f_2}", from=1-2, to=2-2]
	\arrow["g", from=1-1, to=1-2]
\end{tikzcd}\]
\end{itemize}
\end{definition}

Similarly, the arrow category is equipped with a codomain functor $ \cod : \mathcal{A}^2 \rightarrow \mathcal{A}$; it is moreover equipped with a further \emph{domain functor} $ \dom : \mathcal{A}^2\rightarrow \mathcal{A} $ sending $ f : A \rightarrow B$ on $A$ and $(g,g') : f_1 \rightarrow f_2$ on the underlying arrow between domains $ g : A_1 \rightarrow A_2$.\\ 

Arrows categories inherit generally lot of the structure of the underlying category. In particular, $\mathcal{A}^2$ inherits limits and colimits existing in $\mathcal{A}$, and moreover they are \emph{pointwise}, that is, for $ F : I \rightarrow \mathcal{A}^2$, the limit cone (resp. the colimiting cocone) in $\mathcal{B}^2$ are computed as 
% https://q.uiver.app/?q=WzAsNCxbMCwwLCJcXHVuZGVyc2V0e2kgXFxpbiBJfXtcXGxpbX0gXFwsXFxkb20gXFwsIEYoaSkiXSxbMCwxLCJcXGRvbVxcLCBGKGkpIl0sWzEsMSwiXFxjb2QgXFwsIEYoaSkiXSxbMSwwLCJcXHVuZGVyc2V0e2kgXFxpbiBJfXtcXGxpbX0gXFwsXFxjb2QgXFwsIEYoaSkiXSxbMCwxLCJmXzEiLDJdLFsxLDIsIkYoaSkiLDJdLFszLDIsImZfMiJdLFswLDMsIlxcdW5kZXJzZXR7aSBcXGluIEl9e1xcbGltfSBcXCwgRihpKSJdXQ==
\[\begin{tikzcd}
	{\underset{i \in I}{\lim} \,\dom \, F(i)} & {\underset{i \in I}{\lim} \,\cod \, F(i)} \\
	{\dom\, F(i)} & {\cod \, F(i)}
	\arrow["{p_i}"', from=1-1, to=2-1]
	\arrow["{F(i)}"', from=2-1, to=2-2]
	\arrow["{p_i'}", from=1-2, to=2-2]
	\arrow["{\underset{i \in I}{\lim} \, F(i)}", from=1-1, to=1-2]
\end{tikzcd} \hskip 1cm \textrm{resp.} \hskip 1cm \begin{tikzcd}
	{\dom\, F(i)} & {\cod \, F(i)} \\
	{\underset{i \in I}{\colim} \,\dom \, F(i)} & {\underset{i \in I}{\colim} \,\cod \, F(i)}
	\arrow["{q_i}"', from=1-1, to=2-1]
	\arrow["{F(i)}", from=1-1, to=1-2]
	\arrow["{q_i'}", from=1-2, to=2-2]
	\arrow["{\underset{i \in I}{\colim} \, F(i)}"', from=2-1, to=2-2]
\end{tikzcd} \]
where the projections $ p_i, p_i'$ (resp. the inclusion $ q_i, q_i'$) form the limiting cone (resp. the colimiting cocone) of the domains and codomains respectively, while the middle arrow is induced by the universal property of the limit (resp. the colimit) from the induced cone $ ( F(i) p_i : \lim_{i \in I} \dom \, F(i) \rightarrow \cod \, F(i))_{i \in I}$ (resp. from the induced cocone $ (q'_i F(i) : \dom \, F(i) \rightarrow \colim_{i \in I} \cod \, F(i))_{i \in I}$). As a consequence, both the functors $ \dom $ and $ \cod$ preserve either limits and colimits. \\

We also have a common section $ 1 : \mathcal{A} \rightarrow \mathcal{A}^2$, sending an object $A$ on its identity arrow $ 1_A : A \rightarrow A$, endowed with a string of adjunctions
% https://q.uiver.app/?q=WzAsMixbMCwwLCJcXG1hdGhjYWx7QX1eMiJdLFsyLDAsIlxcbWF0aGNhbHtBfSJdLFsxLDAsIjEiLDFdLFswLDEsIlxcY29kIiwwLHsiY3VydmUiOi0zfV0sWzAsMSwiXFxkb20iLDIseyJjdXJ2ZSI6M31dLFszLDIsIiIsMSx7ImxldmVsIjoxLCJzdHlsZSI6eyJuYW1lIjoiYWRqdW5jdGlvbiJ9fV0sWzIsNCwiIiwxLHsibGV2ZWwiOjEsInN0eWxlIjp7Im5hbWUiOiJhZGp1bmN0aW9uIn19XV0=
\[\begin{tikzcd}
	{\mathcal{A}^2} && {\mathcal{A}}
	\arrow[""{name=0, anchor=center, inner sep=0}, "1"{description}, from=1-3, to=1-1]
	\arrow[""{name=1, anchor=center, inner sep=0}, "\cod", curve={height=-18pt}, from=1-1, to=1-3]
	\arrow[""{name=2, anchor=center, inner sep=0}, "\dom"', curve={height=18pt}, from=1-1, to=1-3]
	\arrow["\dashv"{anchor=center, rotate=-90}, draw=none, from=1, to=0]
	\arrow["\dashv"{anchor=center, rotate=-90}, draw=none, from=0, to=2]
\end{tikzcd}\]
It is easy to see that $ 1$ also preserves both limits and colimits. \\

We recall that limits in the coslice $ A \downarrow \mathcal{A}$ are computed from the universal property of the limit: for a diagram $ F : I \rightarrow A\downarrow  \mathcal{A}$ the limit of $F$ is computed as 
% https://q.uiver.app/?q=WzAsMyxbMCwxLCJBIl0sWzEsMSwiXFxjb2QgIFxcLCBGKGkpIl0sWzEsMCwiXFxsaW0gXFxjb2RcXCwgRiJdLFswLDEsIkZfaSIsMl0sWzIsMSwicF9pIl0sWzAsMiwiKEYoaSkpX3tpIFxcaW4gSX0iXV0=
\[\begin{tikzcd}
	& {\lim \cod\, F} \\
	A & {\cod  \, F(i)}
	\arrow["{F_i}"', from=2-1, to=2-2]
	\arrow["{p_i}", from=1-2, to=2-2]
	\arrow["{(F(i))_{i \in I}}", from=2-1, to=1-2]
\end{tikzcd}\]
In particular the codomain functor preserves limits. \\

Let us now turn our attention to the colimits in the coslices. For any object $ A$ in a cocomplete category $ \mathcal{A}$, $A \downarrow \mathcal{A}$ has colimits and they are computed as follows. Let $ I \rightarrow A \downarrow \mathcal{A}$ a functor; from the category $ \overline{ I}$ by adding to $ I$ a cone made of an object $ i_0 $ and for each $ i \in I$ an arrow $ \overline{f}_i : i_0  \rightarrow i$ such that for each $ d : i \rightarrow j$ in $I$ one has $ d \overline{ f}_i = \overline{ f}_j$, and denote $ \iota : I \hookrightarrow \overline{ I}$ the full inclusion. Then one can extend canonically $\cod \,F$ into $ \overline{F}$ % = lan_\iota Cod \, Fthe left Kan extension

\[ % https://tikzcd.yichuanshen.de/#N4Igdg9gJgpgziAXAbVABwnAlgFyxMJZABgBpiBdUkANwEMAbAVxiRAEkQBfU9TXfIRRkAjFVqMWbADrSINGACcGWMDGDsu3XiAzY8BIiPLj6zVohABjAASyoEAO5g6ixU7vSAtnRwALK0ZgAGEtHj59QSNSMWozKUtZH39AhhCw8RgoAHN4IlAAM3cvJAAmahwIJABmOMkLEGDobULipDIQSrK68zYAMRaQIogSxA6uxGNOuiwGNj8ICABrEB6EkFl8HDpVkAY6ACMYBgAFfgMhEEUsbL8cQeHRqYnavdUGqDo4Pyzd+Ib9mAAPqbCDbGxNKCeUg2AbUfZHU7nKKWa63e7hIZtMYVKqIV7-GTSbBeGAAR12COOZ0ihksqmwsEp7zYkDUux+dCgrIIrAqMzmljZrC4FC4QA
\begin{tikzcd}
I \arrow[r, "F"] \arrow[d, "\iota"', hook] \arrow[rd, "=", phantom] & A \downarrow \mathcal{A} \arrow[d, "\cod"] \\
\overline{I} \arrow[r, "{\overline{F}}"', dashed]                  & \mathcal{A}                              
\end{tikzcd} \]
by defining $ \overline{ F}(i) = \cod (F(i))$, $ \overline{ F}(i_0) = A$ and $ \overline{F}(\overline{f}_i) = F(i)$. Then $\overline{F}$ form a cone in $\mathcal{A}$ with vertex $A$; now one has
%where $ \overline{F} \circ \iota = lan_\iota Cod \, F \circ \iota \simeq Cod \, F$ because $\iota $ is full and faithful. Then one has 
\[  \cod ( \colim \, F ) \simeq \colim \, \overline{F} \]
and the colimit in $ A \downarrow \mathcal{A} $ is \[ \colim \, F = \overline{ q}_{i_0} : A \rightarrow \colim \, \overline{F} \]
which coincides also with any composite $ \overline{ q}_i F(i)$. So that actually $ \cod$ does not exactly preserve colimits, though any colimit in $ A \downarrow \mathcal{A}$ is sent to a colimit over a diagram canonically extending it: in some sense, we could say that $\cod$ ``corrects colimits". In particular we have a canonical arrow 
\[ \colim \; \cod \, F \rightarrow \cod(\colim \, F) \]
which may not be an isomorphism. 

\begin{remark}
Let us see concretely why the codomain functor $ \cod : A\downarrow \mathcal{A}$ does not preserve arbitrary colimits, failing in particular to preserve binary coproducts. Suppose that $ f_1 : A \rightarrow \cod(f_1), f_2 : A \rightarrow \cod(f_2)$ have coproduct $ f_1 + f_2 : A \rightarrow \cod(f_1 + f_2)$; then it may happen that, in $\mathcal{A}$, the inclusions of the coproduct $ \iota_1 : \cod(f_1) \rightarrow \cod(f_1) + \cod(f_2), \, \iota_2 : \cod(f_2) \rightarrow \cod(f_1) + \cod(f_2) $ are such that the composite $ \iota_1 f_1 $ and $ \iota_2 f_2$ are not equal, while the inclusion of the coproduct $ q_1 : f_1 \rightarrow f_1 + f_2, \, q_2 : f_2 \rightarrow f_1 + f_2$ are such that $ q_1 f_1 = f_1+f_2 = q_2f_2$ in $\mathcal{A}$: that is we have a diagram 
\[ % https://tikzcd.yichuanshen.de/#N4Igdg9gJgpgziAXAbVABwnAlgFyxMJZABgBoBGAXVJADcBDAGwFcYkQBBEAX1PU1z5CKcqWLU6TVuwDC0ABQAzAPrkAlDz4gM2PASKiATBIYs2iEHKhLlhjb366hRQxRNTzlhSvUACANS+VjZ2mo6C+igAzG40ptIWwT7+KqHcEjBQAObwRKCKAE4QALZIZCA4EEiikmbsPmEghSXVNJVIMbUJIFiqjc2liJ3tiAAsDk1Fg+Ujrl2eqSA0jPQARjCMAAoCesIgBVhZABY4-VNIcyOd8Z69hksgK+tbO84WB8enEwOtFVVjNHWYCgHXKN3YAEc+t9zohLv9RoCYMCkABaKJgjyQ2wPJ4bbZOSL7Q4nM4tRAzf7XLEWAA6tLAMAhuKwjPYkDZNCOMHoIIsHLYbXoWEY7IIgsea3xryJrOwsB4lG4QA
\begin{tikzcd}
                                                                   & \cod(f_1) \arrow[rd, "\iota_1"] \arrow[rrd, "q_1", bend left]    &                               &              \\
A \arrow[ru, "f_1"] \arrow[rd, "f_2"'] \arrow[rr, "\neq", phantom] &                                                             & \cod(f_1) + \cod(f_2) \arrow[]{r}{\langle q_1, q_2 \rangle } & \cod(f_1+f_2) \\
                                                                   & \cod(f_2) \arrow[ru, "\iota_2"'] \arrow[rru, "q_2"', bend right] &                               &             
\end{tikzcd} \]
where the outer square commutes as well as the upper and lower triangles, but not the inner square. Then $\cod(f_1 +f_2)$ cannot be the coproduct $ \cod(f_1) + \cod(f_2)$: in fact, it is the coequalizer $\cod(f_1 + f_2) = \coeq( \iota_1 f_1, \iota_2f_2)$ in $\mathcal{A}$. More generally, for an arbitrary diagram $ f_{(-)} : I \rightarrow A\downarrow\mathcal{A}$, we have that $ \cod( \colim_{i \in I}  f_i )$ is the wide coequalizer of the set of parallel arrows $( \iota_i f_i : A \rightarrow \colim_{i \in I} \cod(f_i)$. 
\end{remark}

However for certain shapes of colimits, we have actually preservation by the codomain functor. Before that, recall first that a category $ I$ is said to be \emph{connected} if for any $ i, i'$ there exists a finite zigzag or arrows in $I$
% https://q.uiver.app/?q=WzAsNSxbMCwxLCJpIl0sWzEsMCwiaV8xIl0sWzIsMSwiLi4uIl0sWzMsMCwiaV9uIl0sWzQsMSwiaSciXSxbMSwwXSxbMSwyXSxbMywyXSxbMyw0XV0=
\[\begin{tikzcd}[sep=small]
	& {i_1} && {i_n} \\
	i && {...} && {i'}
	\arrow[from=1-2, to=2-1]
	\arrow[from=1-2, to=2-3]
	\arrow[from=1-4, to=2-3]
	\arrow[from=1-4, to=2-5]
\end{tikzcd}\]

Now a functor $ G : I \rightarrow J$ is said to be \emph{final} if for any $ j $ in $J$ the comma category $ i\downarrow G$ is non empty and connected. In particular, whenever $I$ is filtered, this amounts to saying that for any $ j $ in $J$ there is some $ d : j \rightarrow F(i)$ and that any span $ d_1 : j \rightarrow F(i_1)$, $ d_2 : j \rightarrow F(i_2)$ can be completed by a commutative square
% https://q.uiver.app/?q=WzAsNCxbMCwwLCJqIl0sWzAsMSwiRihpXzEpIl0sWzEsMCwiRihpXzIpIl0sWzEsMSwiRihpXzMpIl0sWzAsMSwiZF8xIiwyXSxbMCwyLCJkXzIiXSxbMiwzLCJGKGVfMikiXSxbMSwzLCJGKGVfMSkiLDJdXQ==
\[\begin{tikzcd}
	j & {F(i_2)} \\
	{F(i_1)} & {F(i_3)}
	\arrow["{d_1}"', from=1-1, to=2-1]
	\arrow["{d_2}", from=1-1, to=1-2]
	\arrow["{F(e_2)}", from=1-2, to=2-2]
	\arrow["{F(e_1)}"', from=2-1, to=2-2]
\end{tikzcd}\]
It is well known that a functor $ G : I \rightarrow J$ is final if and only if precomposing with it does not modify colimits: that is, if for any $ F : J \rightarrow \mathcal{A}$, we have an isomorphism
\[ \underset{J}{\colim} \, F \simeq  \underset{I}{\colim}\, F G \]

Equipped with those notions, we can say more on the codomain functor. 
\begin{lemma}
The codomain functor preserves connected colimits.
\end{lemma}

\begin{proof}
This is because in this case the inclusion $ \iota : I \hookrightarrow \overline{I}$ is cofinal: indeed, for any $ i,j$, the unique span $ \overline{f}_i : i_0 \rightarrow i$, $ \overline{f}_j : i_0 \rightarrow j$ can be completed by a zigzag because $ I$ is connected, and we imposed that the $\overline{f}_i$ form a cone, so that triangles in this zigzag commute. Hence we have
\begin{align*}
     \cod (\colim \, F) &= \colim \, \overline{ F} \\
     &= \colim \, \overline{F} \, \iota  \\
     &= \colim \, \cod \, F
\end{align*} \end{proof}

\begin{corollary}
If $\mathcal{A}$ has filtered colimits, then so has $ A\downarrow \mathcal{A}$ and the codomain functor preserves them. 
\end{corollary}

\begin{corollary}
If $\mathcal{A}$ has coequalizers, then so has $ A\downarrow \mathcal{A}$ and the codomain functor preserves them.
\end{corollary}

\begin{remark}
Thoses facts are well known (see for instance \cite{borceux1994handbook}[Proposition 2.16.3]). For coequalizers, observe that in the diagram below with the coequalizer computed in $\mathcal{A}$
% https://q.uiver.app/?q=WzAsNCxbMCwxLCJBIl0sWzAsMCwiQl8xIl0sWzEsMCwiQl8yIl0sWzIsMCwiXFxjb2VxKGcsZycpIl0sWzIsMywicV97ZyxnJ30iXSxbMSwyLCJnIiwwLHsib2Zmc2V0IjotMX1dLFsxLDIsImcnIiwyLHsib2Zmc2V0IjoxfV0sWzAsMSwiZl8xIl0sWzAsMiwiZl8yIiwyXV0=
\[\begin{tikzcd}
	{B_1} & {B_2} & {\coeq(g,g')} \\
	A
	\arrow["{q_{g,g'}}", from=1-2, to=1-3]
	\arrow["g", shift left=1, from=1-1, to=1-2]
	\arrow["{g'}"', shift right=1, from=1-1, to=1-2]
	\arrow["{f_1}", from=2-1, to=1-1]
	\arrow["{f_2}"', from=2-1, to=1-2]
\end{tikzcd}\]
the composite $ q_{g,g'} f_2 : A \rightarrow \coeq(g,g') $ together with $ q_{g,g'}$ is a coequalizer in $ A \downarrow \mathcal{B}$. 
\end{remark}

To conclude this section, we should give a lemma that will be of use in the following section:

\begin{lemma}\label{esofull cofinal}
Let be $ F : \mathcal{I} \rightarrow \mathcal{J}$ a functor such that \begin{itemize}
    \item $F$ is essentially surjective and full
    \item $I$ is filtered
\end{itemize}
Then $ F$ is cofinal and moreover $ J$ is also filtered. 
\end{lemma}

\begin{proof}
For any $j $ in $J$, the comma $ j \downarrow F$ is non empty, as it contains an isomorphism $ j \simeq F(i)$ given by essential surjectivity. Now, for any span
% https://q.uiver.app/?q=WzAsMyxbMCwwLCJqIl0sWzAsMSwiRihpXzEpIl0sWzEsMCwiRihpXzIpIl0sWzAsMSwiZl8xIiwyXSxbMCwyLCJmXzIiXV0=
\[\begin{tikzcd}
	j & {F(i_2)} \\
	{F(i_1)}
	\arrow["{f_1}"', from=1-1, to=2-1]
	\arrow["{f_2}", from=1-1, to=1-2]
\end{tikzcd}\]
then filteredness of $ I$ ensures the existence of a cospan $ d_1 : i_1 \rightarrow i_3 $, $ d_2 : i_2 \rightarrow i_3$; however, $ F(d_1)f_1$ and $ F(i_2)f_2$ may not be equal: however, essential surjectivity gives some $ u : j \simeq F(i)$ and fullness gives some $ e_1 : i \rightarrow i_1$, $ e_2 : i \rightarrow i_2$ such that $ f_1 = F(e_1)w$ and $ f_2= F(e_2)w $, and $ d_1e_1$ and $ d_2e_2$ are equalized in $ I$ by some $ d_3 : i_3 \rightarrow i_4$ so that the following diagram commutes
% https://q.uiver.app/?q=WzAsNSxbMSwxLCJGKGkpIl0sWzEsMiwiRihpXzEpIl0sWzIsMSwiRihpXzIpIl0sWzIsMiwiRihpXzQpIl0sWzAsMCwiaiJdLFsxLDMsIkYoZF8zZF8xKSIsMl0sWzIsMywiRihkXzNkXzIpIl0sWzQsMiwiZl8yIiwwLHsiY3VydmUiOi0yfV0sWzQsMSwiZl8xIiwyLHsiY3VydmUiOjJ9XSxbNCwwLCJ3IFxcYXRvcCBcXHNpbWVxIiwxXSxbMCwyLCJGKGVfMikiLDFdLFswLDEsIkYoZV8xKSIsMV1d
\[\begin{tikzcd}[column sep=large]
	j \\
	& {F(i)} & {F(i_2)} \\
	& {F(i_1)} & {F(i_4)}
	\arrow["{F(d_3d_1)}"', from=3-2, to=3-3]
	\arrow["{F(d_3d_2)}", from=2-3, to=3-3]
	\arrow["{f_2}", curve={height=-12pt}, from=1-1, to=2-3]
	\arrow["{f_1}"', curve={height=12pt}, from=1-1, to=3-2]
	\arrow["{w \atop \simeq}"{description}, from=1-1, to=2-2]
	\arrow["{F(e_2)}"{description}, from=2-2, to=2-3]
	\arrow["{F(e_1)}"{description}, from=2-2, to=3-2]
\end{tikzcd}\]
Hence $i \downarrow F$ is non empty and connected: $F $ is cofinal. Filteredness of $ J$ is inherited from $I$. 
\end{proof}

\section{Coslices of Locally finitely presentable categories}

In this section we turn to coslices of locally finitely presentable categories. It is already well known that coslices, as well as slices, of locally finitely presentable categories are still locally finitely presentable (and also more generally, (co)slices of accessible categories are accessible), as it was for instance proven in \cite{adamek1994locally}[Proposition 1.57]. However this instance seems to be the sole reference about it in the literature, though it processes by a specific characterization of locally finitely presentable categories that dispenses to care about explicit description of finitely presented object. It seems that no other reference in the literature treated this topic nor offered concrete proof, nor any explicit description of the finitely presented objects. However this problem is related to a specific version of the small object argument, as done in \cite{Anel} and \cite{Coste}, which examine a generator of finitely presented left map in the context of a factorization system. Inspired from those sources, we provide here a characterization of the generator of finitely presented objects in the coslices of a locally finitely presented category. \\

Before going any further, we recall the following generalities about locally finitely presentable categories. We recall that a finitely presented object in a category $\mathcal{B}$ is some $K$ such that for any $F : I \rightarrow \mathcal{B}$ with $ I$ filtered we have an isomorphism 
\[ \mathcal{B}[K, \underset{i \in I}{\colim} \, F(i)] \simeq \underset{i \in I}{\colim} \,\mathcal{B}[K,  F(i)] \]
Concretely this means that: \begin{itemize}
    \item for any arrow $ a :K \rightarrow \colim_{i \in I}\, F(i)$, there is some $ i \in I$ and some factorization, called a \emph{lift} 
    % https://q.uiver.app/?q=WzAsMyxbMCwxLCJLIl0sWzEsMSwiXFx1bmRlcnNldHtpIFxcaW4gSX17XFxjb2xpbX0gXFwsIEYoaSkiXSxbMSwwLCJGKGkpIl0sWzIsMSwicV9pIl0sWzAsMSwiYSIsMl0sWzAsMiwiXFxvdmVybGluZXthfSJdXQ==
\[\begin{tikzcd}
	& {F(i)} \\
	K & {\underset{i \in I}{\colim} \, F(i)}
	\arrow["{q_i}", from=1-2, to=2-2]
	\arrow["a"', from=2-1, to=2-2]
	\arrow["{b}", from=2-1, to=1-2]
\end{tikzcd}\]
(we shall often denote such a lift as the pair $(i,b)$);
    \item and for any two such lifts $(b_1,i_1)$ and $ (b_2, i_2)$ of a same arrow $a$, we shall refer to as \emph{parallel lifts}, there exists a cospan $ d_1 : i_1 \rightarrow i$ and $ d_2 : i_2 \rightarrow i$ such that both lifts are equalized jointly into a third lift, we shall often call a \emph{further refinement} 
    % https://q.uiver.app/?q=WzAsNSxbMSwzLCJcXHVuZGVyc2V0e2kgXFxpbiBJfXtcXGNvbGltfSBcXCwgRihpKSJdLFswLDEsIkYoaV8xKSJdLFsyLDEsIkYoaV8yKSJdLFsxLDIsIkYoaSkiXSxbMSwwLCJLIl0sWzEsMywiRihkXzEpIiwxLHsibGFiZWxfcG9zaXRpb24iOjQwfV0sWzIsMywiRihkXzIpIiwxXSxbMywwLCJxX2kiLDFdLFsyLDAsInFfe2lfMn0iLDEseyJjdXJ2ZSI6LTJ9XSxbMSwwLCJxX3tpXzF9IiwxLHsiY3VydmUiOjJ9XSxbNCwxLCJiXzEiLDJdLFs0LDIsImJfMiIsMCx7ImxhYmVsX3Bvc2l0aW9uIjo3MH1dXQ==
\[\begin{tikzcd}
	& K \\
	{F(i_1)} && {F(i_2)} \\
	& {F(i)} \\
	& {\underset{i \in I}{\colim} \, F(i)}
	\arrow["{F(d_1)}"{description, pos=0.4}, from=2-1, to=3-2]
	\arrow["{F(d_2)}"{description}, from=2-3, to=3-2]
	\arrow["{q_i}"{description}, from=3-2, to=4-2]
	\arrow["{q_{i_2}}"{description}, curve={height=-12pt}, from=2-3, to=4-2]
	\arrow["{q_{i_1}}"{description}, curve={height=12pt}, from=2-1, to=4-2]
	\arrow["{b_1}"', from=1-2, to=2-1]
	\arrow["{b_2}"{pos=0.7}, from=1-2, to=2-3]
\end{tikzcd}\]
\end{itemize}

First the following lemma from \cite{adamek1994locally}[Exercice 1.o] says that finitely presented objects can be jointly used to test if a cocone is colimitting:

\begin{lemma}\label{testing filtered colimits}
Let be $ D : I \rightarrow \mathcal{B}$ with $ I$ finitary and $\mathcal{B}$ locally finitely presentable. Then a cocone $ (f_i : D_i \rightarrow C)_{i \in I}$ in $\mathcal{B}$ exhibits $ C$ as the colimit $ \colim\, D$ if and only if for any finitely presented object $K$ one has an isomorphism natural in $K$ 
\[ \mathcal{B}(K, C) \simeq \underset{i \in I}{\colim}\; \mathcal{B}(K, D_i) \]
\end{lemma}

\begin{proof}
This is because finitely presented objects from altogether a separating family. If $C$ is such that we have in each $ K$ the natural isomorphism above then one has also a natural isomorphism 
\[ \mathcal{B}(K, C) \simeq  \mathcal{B}(K,\underset{i \in I}{\colim}\; D_i) \]
Hence by naturallity we have an equivalence $ \mathcal{B}_\omega \downarrow C \simeq \mathcal{B}_\omega \downarrow \colim_{i \in I} D_i $, and hence an isomorphism $ C \simeq \colim_{i \in I} D_i $.
\end{proof}

We also will need the following, when encountering finitary functors, to control the behaviour of their eventual left adjoints:

\begin{lemma}\label{Ladj of finitary functor}
Let $F \dashv G$ be adjoint functors between locally finitely presentable categories. Then $F$ maps finitely presented objects on finitely presented objects if and only if $G $ is finitary. 
\end{lemma}

\begin{proof}
If $F : \mathcal{B} \rightarrow \mathcal{A}$ preserves finitely presented objects and $D : I \rightarrow \mathcal{A} $ a filtered diagram, then for any finitely presented object $ K$ in $\mathcal{B}$:
\begin{align*}
    \begin{split}
        \mathcal{B} (K , G(\underset{i \in I}{\colim} \; D_i)) &\simeq \mathcal{A}(F(K), \underset{i \in I}{\colim} \; D_i) \\
        &\simeq \underset{i \in I}{\colim} \; \mathcal{A}(F(K), D_i) \\
        &\simeq \underset{i \in I}{\colim} \; \mathcal{B}(K, G(D_i))
    \end{split}
\end{align*}
But from \cref{testing filtered colimits} this means that $G(\colim_{i \in I}D_i) \simeq \colim_{i \in I}G(D_i) $. Then converse is immediate: if $ G$ is finitary, $K$ is finitely presented and $ D : I \rightarrow \mathcal{A}$ is filtered, then permuting the identities above ensures that $ F(K)$ still is finitely presented. 
\end{proof}

Then we should give some precision on the arrow category of a locally finitely presentable category, and more generally, about finitely presented object 

\begin{proposition}
If $ \mathcal{B}$ is locally finitely presentable, then for any finite diagram $D$, the functor category $ \mathcal{B}^D$ is locally finitely presentable and we have 
\[ (\mathcal{B}^D)_\omega \simeq (\mathcal{B}_\omega)^D \]
\end{proposition}

A version of this concerning accessible categories is provided in \cite{makkai1988strong}[Lemma 5.1]. The corresponding statement for locally finitely presentable categories is an automatic consequence of the existence of colimits in functors categories $ \mathcal{B}^D$ where they are pointwise. \\

In particular, for a locally finitely presentable category, the arrow category $\mathcal{B}^2$ is locally finitely presentable and we have 
\[ (\mathcal{B}^2)_\omega \simeq (\mathcal{B}_\omega)^2 \]

Now we can turn to the explicit description of the following generator of finitely presented objects in the coslices. We first introduce our candidate - whose choice is of course guessed from a similar technique in \cite{Anel} and \cite{Coste}. Then we prove it to enjoy some lifting properties, stating that any finite diagram (and in particular any triangle) in it can be obtained as a pushout of a diagram of finitely presented objects of the same shape. From this we deduce it to be closed under finite limits - in particular under retracts. Then we prove the coslice to be locally finitely presentable and our candidate to be its generator of finitely presented objects.

\begin{definition}
For any object $B$ in $\mathcal{B}$, define the \emph{coslice generator at $B$} as full subcategory $ \mathcal{G}_B$ of $ B \downarrow \mathcal{B}$ consisting of morphisms $ l: B \rightarrow C$ such that there exists some $ k : K \rightarrow K'$ in $\mathcal{B}^2_\omega$ and $ a : K \rightarrow B$ exhibiting $ n$ as the pushout     % https://q.uiver.app/?q=WzAsNCxbMCwwLCJLIl0sWzEsMCwiSyciXSxbMCwxLCJCIl0sWzEsMSwiQyJdLFswLDIsImEiLDJdLFswLDEsImwiXSxbMSwzXSxbMiwzLCJuIiwyXSxbMywwLCIiLDIseyJzdHlsZSI6eyJuYW1lIjoiY29ybmVyIn19XV0=
\[\begin{tikzcd}
	{K} & {K'} \\
	{B} & {C}
	\arrow["{a}"', from=1-1, to=2-1]
	\arrow["{k}", from=1-1, to=1-2]
	\arrow[from=1-2, to=2-2]
	\arrow["{l}"', from=2-1, to=2-2]
	\arrow["\lrcorner"{very near start, rotate=180}, from=2-2, to=1-1, phantom]
\end{tikzcd}\]
\end{definition}

%Observe that right cancellation of saturated maps imposes that for 
The first thing which is clear about this category is the following:

\begin{lemma}
Objects of $\mathcal{G}_B$ are finitely presented in $B \downarrow \mathcal{B}$
\end{lemma}

\begin{proof}
Let be $ F : I \rightarrow B \downarrow \mathcal{B}$ a filtered diagram; then form what was said about filtered colimits in the coslice, we have $ \cod ( \colim \, F) \simeq \colim \, \cod \, F$. Then for any situation as below 
% https://q.uiver.app/?q=WzAsNSxbMCwwLCJLIl0sWzEsMCwiSyciXSxbMCwxLCJCIl0sWzEsMSwiYl8qSyciXSxbMSwyLCJcXGNvbGltXFwsIFxcY29kIFxcLCAgRiJdLFswLDIsImIiLDJdLFsyLDMsImJfKmsiLDFdLFswLDEsImsiXSxbMSwzLCJiJyJdLFszLDAsIiIsMSx7InN0eWxlIjp7Im5hbWUiOiJjb3JuZXIifX1dLFsyLDQsIlxcY29saW1cXCwgICBGIiwyXSxbMyw0LCJhIl1d
\[\begin{tikzcd}
	K & {K'} \\
	B & {b_*K'} \\
	& {\colim\, \cod \,  F}
	\arrow["b"', from=1-1, to=2-1]
	\arrow["{b_*k}"{description}, from=2-1, to=2-2]
	\arrow["k", from=1-1, to=1-2]
	\arrow["{b'}", from=1-2, to=2-2]
	\arrow["\lrcorner"{anchor=center, pos=0.125, rotate=180}, draw=none, from=2-2, to=1-1]
	\arrow["{\colim\,   F}"', from=2-1, to=3-2]
	\arrow["a", from=2-2, to=3-2]
\end{tikzcd}\]
the composite arrow $ a k_*b : K' \rightarrow \colim \, \cod \, F$ lifts through some $ q_i : \cod \, F(i) \rightarrow \colim \, \cod \, F(i)$ as 
% https://q.uiver.app/?q=WzAsNCxbMSwwLCJLJyJdLFswLDEsImJfKksnIl0sWzEsMiwiXFxjb2xpbVxcLCBcXGNvZCBcXCwgIEYiXSxbMiwxLCJcXGNvZFxcLCBGKGknKSJdLFswLDEsImInIiwxXSxbMywyLCJxX2kiXSxbMCwzLCJcXG92ZXJsaW5le2F9IiwwLHsic3R5bGUiOnsiYm9keSI6eyJuYW1lIjoiZGFzaGVkIn19fV0sWzEsMiwiYSIsMV1d
\[\begin{tikzcd}[sep=small]
	& {K'} \\
	{b_*K'} && {\cod\, F(i')} \\
	& {\colim\, \cod \,  F}
	\arrow["{b'}"{description}, from=1-2, to=2-1]
	\arrow["{q_i}", from=2-3, to=3-2]
	\arrow["{\overline{a}}", dashed, from=1-2, to=2-3]
	\arrow["a"{description}, from=2-1, to=3-2]
\end{tikzcd}\]
However, it is still not clear that the induced parallel lifts $ F(i)b : K \rightarrow F(i) $ and $ \overline{a}k : K \rightarrow F(i)$ commute: but from $K$ is finitely presented and $ I$ is filtered, we know they are equalized by some $ F(d)$ for some morphism $d : i \rightarrow i'$ in $I$, and moreover,  $F(i') = F(d)F(i)$; then the universal property of the pushout induces a universal map 
% https://q.uiver.app/?q=WzAsNixbMCwxLCJLIl0sWzEsMCwiSyciXSxbMCwyLCJCIl0sWzEsMSwiYl8qSyciXSxbMSwzLCJcXGNvbGltXFwsIFxcY29kIFxcLCAgRiJdLFszLDEsIlxcY29kXFwsIEYoaScpIl0sWzAsMiwiYiIsMl0sWzIsMywiYl8qayIsMV0sWzAsMSwiayJdLFsxLDMsImInIiwxXSxbMiw0LCJcXGNvbGltXFwsICAgRiIsMl0sWzUsNCwicV9pIl0sWzEsNSwiRihkKVxcb3ZlcmxpbmV7YX0iLDAseyJjdXJ2ZSI6LTJ9XSxbMyw0LCJhIiwxXSxbMiw1LCJGKGknKSIsMSx7ImxhYmVsX3Bvc2l0aW9uIjo3MH1dLFszLDUsIlxcbGFuZ2xlIEYoaScpLCBGKGQpIFxcb3ZlcmxpbmV7YX0gXFxyYW5nbGUiLDEseyJzdHlsZSI6eyJib2R5Ijp7Im5hbWUiOiJkYXNoZWQifX19XSxbMyw4LCIiLDEseyJsZXZlbCI6MSwic3R5bGUiOnsibmFtZSI6ImNvcm5lciJ9fV1d
\[\begin{tikzcd}[sep=large]
	& {K'} \\
	K & {b_*K'} && {\cod\, F(i')} \\
	B \\
	& {\colim\, \cod \,  F}
	\arrow["b"', from=2-1, to=3-1]
	\arrow["{b_*k}"{description}, from=3-1, to=2-2]
	\arrow[""{name=0, anchor=center, inner sep=0}, "k", from=2-1, to=1-2]
	\arrow["{b'}"{description}, from=1-2, to=2-2]
	\arrow["{\colim\,   F}"', from=3-1, to=4-2]
	\arrow["{q_i}", from=2-4, to=4-2]
	\arrow["{F(d)\overline{a}}", curve={height=-12pt}, from=1-2, to=2-4]
	\arrow["a"{description}, from=2-2, to=4-2]
	\arrow["{F(i')}"{description, pos=0.7}, from=3-1, to=2-4, crossing over]
	\arrow["{\langle F(i'), F(d) \overline{a} \rangle}"{description}, dashed, from=2-2, to=2-4]
	\arrow["\lrcorner"{anchor=center, pos=0.125, rotate=180}, draw=none, from=2-2, to=0]
\end{tikzcd}\]
Similarly, we use finitely presentedness of $K'$ to prove that any two lifts of $a$ have to be equalized by a further refinement. This proves the pushout map $ b_*k$ to be finitely presented in $B \downarrow \mathcal{B}$
\end{proof}

Before being able to justify the appellation of generator, we must first establish a series of technical properties allowing us to lift arrow and finite diagrams from $B \downarrow\mathcal{B}$ to $ \mathcal{B}_\omega$. Those results will depend on the following useful technical lemma, which seems to have first appeared in \cite{Anel}[sub-lemma 12]:

\begin{lemma}\label{Anel lemma}
Let be a diagram as below
% https://q.uiver.app/?q=WzAsNSxbMCwwLCJLXzAiXSxbMCwxLCJCIl0sWzEsMCwiS18wJyJdLFsxLDEsIkQiXSxbMiwwLCJLIl0sWzAsMiwiayJdLFsxLDMsIm4iLDJdLFs0LDMsImEiXSxbMCwxLCJhXzAiLDJdLFsyLDNdLFszLDAsIiIsMix7InN0eWxlIjp7Im5hbWUiOiJjb3JuZXIifX1dXQ==
\[\begin{tikzcd}
	{K_0} & {K_0'} & K \\
	B & C
	\arrow["k", from=1-1, to=1-2]
	\arrow["n"', from=2-1, to=2-2]
	\arrow["a", from=1-3, to=2-2]
	\arrow["{a_0}"', from=1-1, to=2-1]
	\arrow[from=1-2, to=2-2]
	\arrow["\lrcorner"{anchor=center, pos=0.125, rotate=180}, draw=none, from=2-2, to=1-1]
\end{tikzcd}\]
with $ K,\,K_0, \, K_0' $ finitely presented and $ k $ in $\mathcal{B}^2_\omega$: then there exists a factorization 
% https://q.uiver.app/?q=WzAsMyxbMCwwLCJLXzAiXSxbMSwxLCJCIl0sWzEsMCwiS18xIl0sWzAsMiwiYV8yIl0sWzIsMSwiYV8xIl0sWzAsMSwiYV8wIiwyXV0=
\[\begin{tikzcd}
	{K_0} & {K_1} \\
	& B
	\arrow["{a_2}", from=1-1, to=1-2]
	\arrow["{a_1}", from=1-2, to=2-2]
	\arrow["{a_0}"', from=1-1, to=2-2]
\end{tikzcd}\]
such that $a$ factorizes through the following pushout
% https://q.uiver.app/?q=WzAsNyxbMCwwLCJLXzAiXSxbMCwxLCJLXzEiXSxbMCwyLCJCIl0sWzEsMCwiS18wJyJdLFsxLDEsInthXzJ9XypLXzAnIl0sWzEsMiwiRCJdLFsyLDEsIksiXSxbMCwzLCJrIl0sWzAsMSwiYV8yIiwyXSxbMSwyLCJhXzEiLDJdLFszLDRdLFsxLDQsInthXzJ9XyprIl0sWzIsNSwibiIsMl0sWzQsNV0sWzQsMCwiIiwwLHsic3R5bGUiOnsibmFtZSI6ImNvcm5lciJ9fV0sWzUsMSwiIiwwLHsic3R5bGUiOnsibmFtZSI6ImNvcm5lciJ9fV0sWzYsNSwiYSJdLFs2LDQsIiIsMCx7InN0eWxlIjp7ImJvZHkiOnsibmFtZSI6ImRhc2hlZCJ9fX1dXQ==
\[\begin{tikzcd}
	{K_0} & {K_0'} \\
	{K_1} & {{a_2}_*K_0'} & K \\
	B & C
	\arrow["k", from=1-1, to=1-2]
	\arrow["{b_0}"', from=1-1, to=2-1]
	\arrow["{a_1}"', from=2-1, to=3-1]
	\arrow[from=1-2, to=2-2]
	\arrow["{{a_2}_*k}", from=2-1, to=2-2]
	\arrow["n"', from=3-1, to=3-2]
	\arrow[from=2-2, to=3-2]
	\arrow["\lrcorner"{anchor=center, pos=0.125, rotate=180}, draw=none, from=2-2, to=1-1]
	\arrow["\lrcorner"{anchor=center, pos=0.125, rotate=180}, draw=none, from=3-2, to=2-1]
	\arrow["a", from=2-3, to=3-2]
	\arrow[dashed, from=2-3, to=2-2]
\end{tikzcd}\]
\end{lemma}

\begin{remark}
Observe that from the properties of finitely presented objects, we can say moreover that for any two parallel lifts of the same $a$
% https://q.uiver.app/?q=WzAsNSxbMCwwLCJLXzEiXSxbMCwxLCJCIl0sWzEsMCwie2FfMn1fKktfMCciXSxbMSwxLCJEIl0sWzIsMCwiSyJdLFswLDEsImFfMSIsMl0sWzAsMiwie2FfMn1fKmsiXSxbMSwzLCJuIiwyXSxbMiwzXSxbMywwLCIiLDAseyJzdHlsZSI6eyJuYW1lIjoiY29ybmVyIn19XSxbNCwzLCJhIl0sWzQsMiwibSciLDAseyJvZmZzZXQiOi0xfV0sWzQsMiwibSIsMix7Im9mZnNldCI6MX1dXQ==
\[\begin{tikzcd}
	{K_1} & {{a_2}_*K_0'} & K \\
	B & D
	\arrow["{a_1}"', from=1-1, to=2-1]
	\arrow["{{a_2}_*k}", from=1-1, to=1-2]
	\arrow["n"', from=2-1, to=2-2]
	\arrow[from=1-2, to=2-2]
	\arrow["\lrcorner"{anchor=center, pos=0.125, rotate=180}, draw=none, from=2-2, to=1-1]
	\arrow["a", from=1-3, to=2-2]
	\arrow["{m'}", shift left=1, from=1-3, to=1-2]
	\arrow["m"', shift right=1, from=1-3, to=1-2]
\end{tikzcd}\]
there exists a further factorization $ a_1=a_2b_1$ with $ a_2 : K_2 \rightarrow B$ such that $m$ and $m'$ are equalized by the intermediate arrow 
\[ (a_{2*}k)b_1 m = (a_{2*}k)b_1 m'  \]
\end{remark}

This lemma is also crucial to the following fullness-like property of the $ \mathcal{G}_B$ in the coslice, which allows to exhibit any arrow between objects of the etale generator as a pushout square of finitely presented etale map: 

\begin{lemma}\label{lifting of diagrams}
For any triangle in $ B \downarrow \mathcal{B}$
% https://q.uiver.app/?q=WzAsMyxbMCwwLCJCIl0sWzEsMCwiQ18xIl0sWzIsMSwiQ18yIl0sWzAsMSwibl8xIl0sWzAsMiwibl8yIiwyXSxbMSwyLCJuIl1d
\[\begin{tikzcd}[sep=large]
	B & {C_1} \\
	&& {C_2}
	\arrow["{n_1}", from=1-1, to=1-2]
	\arrow["{n_2}"', from=1-1, to=2-3]
	\arrow["n", from=1-2, to=2-3]
\end{tikzcd}\]
with $ n_1$, $n_2$ in $\mathcal{G}_B$, there exists a triangle % https://q.uiver.app/?q=WzAsMyxbMCwwLCJLIl0sWzEsMCwiS18xIl0sWzIsMSwiS18yIl0sWzAsMSwibV8xIl0sWzAsMiwibV8yIiwyXSxbMSwyLCJtIl1d
\[\begin{tikzcd}[sep=large]
	K & {K_1} \\
	&& {K_2}
	\arrow["{m_1}", from=1-1, to=1-2]
	\arrow["{m_2}"', from=1-1, to=2-3]
	\arrow["m", from=1-2, to=2-3]
\end{tikzcd}\]
in $\mathcal{B}^2_\omega$ and some $a :K \rightarrow B$ such that all squares below are pushouts
% https://q.uiver.app/?q=WzAsNixbMCwwLCJLIl0sWzEsMCwiS18xIl0sWzIsMSwiS18yIl0sWzAsMSwiQiJdLFsxLDEsIkNfMSJdLFsyLDIsIkNfMiJdLFswLDEsIm1fMSJdLFswLDIsIm1fMiIsMSx7ImxhYmVsX3Bvc2l0aW9uIjozMH1dLFsxLDIsIm0iXSxbMCwzLCJhIiwyXSxbMSw0LCJhXzEiLDEseyJsYWJlbF9wb3NpdGlvbiI6MzB9XSxbMiw1LCJhXzIiXSxbMyw0LCJuXzEiLDFdLFs0LDUsIm4iLDFdLFszLDUsIm5fMiIsMl0sWzQsMCwiIiwyLHsic3R5bGUiOnsibmFtZSI6ImNvcm5lciJ9fV0sWzUsMSwiIiwxLHsic3R5bGUiOnsibmFtZSI6ImNvcm5lciJ9fV1d
\[\begin{tikzcd}[sep=large]
	K & {K_1} \\
	B & {C_1} & {K_2} \\
	&& {C_2}
	\arrow["{m_1}", from=1-1, to=1-2]
	\arrow["m", from=1-2, to=2-3]
	\arrow["a"', from=1-1, to=2-1]
	\arrow["{a_1}"{description, pos=0.2}, from=1-2, to=2-2]
	\arrow["{a_2}", from=2-3, to=3-3]
	\arrow["{n_1}"{description}, from=2-1, to=2-2]
	\arrow["n"{description}, from=2-2, to=3-3]
	\arrow["{n_2}"', from=2-1, to=3-3]
	\arrow["\lrcorner"{anchor=center, pos=0.125, rotate=180}, draw=none, from=2-2, to=1-1]
	\arrow["\lrcorner"{anchor=center, pos=0.125, rotate=180}, draw=none, from=3-3, to=1-2]	\arrow["{m_2}"{description, pos=0.3}, from=1-1, to=2-3, crossing over]
\end{tikzcd}\]
(so that in particular $ n$ is in $\mathcal{G}_{C_1}$). 
\end{lemma}

\begin{proof}
As $ n_1$ and $ n_2$ are supposed in $\mathcal{G}_B$, they are induced from pushouts as below
% https://q.uiver.app/?q=WzAsNyxbMiwwLCJLXzEiXSxbNCwwLCJLXzEnIl0sWzIsMSwiQiJdLFs0LDEsIkNfMSJdLFs0LDIsIkNfMiJdLFswLDEsIksiXSxbMiwyLCJLXzInIl0sWzAsMSwibV8xIl0sWzAsMiwiYV8xIiwyXSxbMiwzLCJuXzEiLDFdLFszLDQsIm4iLDFdLFsyLDQsIm5fMiIsMV0sWzMsMCwiIiwyLHsic3R5bGUiOnsibmFtZSI6ImNvcm5lciJ9fV0sWzUsMiwiYV8yIl0sWzYsNF0sWzUsNiwibV8yIiwyXSxbMSwzXSxbMTQsMiwiIiwxLHsib2Zmc2V0Ijo0LCJsZXZlbCI6MSwic3R5bGUiOnsibmFtZSI6ImNvcm5lciJ9fV1d
\[\begin{tikzcd}
	&& {K_1} && {K_1'} \\
	K && B && {C_1} \\
	&& {K_2'} && {C_2}
	\arrow["{m_1}", from=1-3, to=1-5]
	\arrow["{a_1}"', from=1-3, to=2-3]
	\arrow["{n_1}"{description}, from=2-3, to=2-5]
	\arrow["n", from=2-5, to=3-5]
	\arrow["{n_2}"{description}, from=2-3, to=3-5]
	\arrow["\lrcorner"{anchor=center, pos=0.125, rotate=180}, draw=none, from=2-5, to=1-3]
	\arrow["{a_2}", from=2-1, to=2-3]
	\arrow[""{name=0, anchor=center, inner sep=0}, from=3-3, to=3-5]
	\arrow["{m_2}"', from=2-1, to=3-3]
	\arrow[from=1-5, to=2-5]
	\arrow["\lrcorner"{anchor=center, pos=0.125, rotate=180}, shift right=4, draw=none, from=0, to=2-3]
\end{tikzcd}\]
Now by filteredness of $ \mathcal{B}_\omega\downarrow B$ there exists $ a_3 : K_3 \rightarrow B $ and a factorization
% https://q.uiver.app/?q=WzAsNCxbMCwwLCJLXzEiXSxbMCwyLCJLXzIiXSxbMCwxLCJLXzMiXSxbMSwxLCJCIl0sWzAsMiwiYl8xIiwyXSxbMiwzLCJhXzMiLDFdLFsxLDIsImJfMiJdLFsxLDMsImFfMiIsMl0sWzAsMywiYV8xIl1d
\[\begin{tikzcd}
	{K_1} \\
	{K_3} & B \\
	{K_2}
	\arrow["{b_1}"', from=1-1, to=2-1]
	\arrow["{a_3}"{description}, from=2-1, to=2-2]
	\arrow["{b_2}", from=3-1, to=2-1]
	\arrow["{a_2}"', from=3-1, to=2-2]
	\arrow["{a_1}", from=1-1, to=2-2]
\end{tikzcd}\]
and the pushouts themselves factorize
% https://q.uiver.app/?q=WzAsMTAsWzEsMCwiS18xIl0sWzIsMSwiS18zIl0sWzIsMywiQiJdLFszLDIsImJfezIqfUtfMiciXSxbMywwLCJLXzEnIl0sWzQsMSwiYl97MSp9S18xJyJdLFszLDQsIkNfMiJdLFs0LDMsIkNfMSJdLFsxLDIsIksnXzIiXSxbMCwxLCJLXzIiXSxbMSwyLCJhXzMiLDEseyJsYWJlbF9wb3NpdGlvbiI6MzB9XSxbMSwzLCJiX3syKn1tXzIiLDFdLFswLDEsImJfMSIsMV0sWzAsNCwibV8xIl0sWzQsNV0sWzEsNSwiYl97MSp9bV8xIiwxXSxbMyw2XSxbMiw2LCJuXzIiLDJdLFsyLDcsIm5fMSIsMSx7ImxhYmVsX3Bvc2l0aW9uIjozMH1dLFs1LDddLFs4LDNdLFs5LDEsImJfMiIsMV0sWzksOCwibV8yIiwyXSxbNSw0LCIiLDAseyJvZmZzZXQiOi0yLCJzdHlsZSI6eyJuYW1lIjoiY29ybmVyIn19XSxbNywxLCIiLDAseyJzdHlsZSI6eyJuYW1lIjoiY29ybmVyIn19XSxbNyw2LCJuIl0sWzMsMjEsIiIsMix7ImxldmVsIjoxLCJzdHlsZSI6eyJuYW1lIjoiY29ybmVyIn19XSxbNiwxMCwiIiwyLHsibGV2ZWwiOjEsInN0eWxlIjp7Im5hbWUiOiJjb3JuZXIifX1dXQ==
\[\begin{tikzcd}
	& {K_1} && {K_1'} \\
	{K_2} && {K_3} && {b_{1*}K_1'} \\
	& {K'_2} && {b_{2*}K_2'} \\
	&& B && {C_1} \\
	&&& {C_2}
	\arrow[""{name=0, anchor=center, inner sep=0}, "{a_3}"{description, pos=0.3}, from=2-3, to=4-3]
	\arrow["{b_{2*}m_2}"{description}, from=2-3, to=3-4]
	\arrow["{b_1}"{description}, from=1-2, to=2-3]
	\arrow["{m_1}", from=1-2, to=1-4]
	\arrow[from=1-4, to=2-5]
	\arrow["{b_{1*}m_1}"{description}, from=2-3, to=2-5]
	\arrow["{n_2}"', from=4-3, to=5-4]
	\arrow["{n_1}"{description, pos=0.3}, from=4-3, to=4-5]
	\arrow[from=2-5, to=4-5, " (b_{1*}m_1)_*a_3 "]
	\arrow[""{name=1, anchor=center, inner sep=0}, "{b_2}"{description}, from=2-1, to=2-3]
	\arrow["{m_2}"', from=2-1, to=3-2]
	\arrow["\lrcorner"{anchor=center, pos=0.125, rotate=180}, shift left=2, draw=none, from=2-5, to=1-4]
	\arrow["\lrcorner"{anchor=center, pos=0.125, rotate=180}, draw=none, from=4-5, to=2-3]
	\arrow["n", from=4-5, to=5-4]
	\arrow["\lrcorner"{anchor=center, pos=0.125, rotate=180}, draw=none, from=3-4, to=1]
	\arrow["\lrcorner"{anchor=center, pos=0.125, rotate=180}, draw=none, from=5-4, to=0]
	\arrow[from=3-2, to=3-4, crossing over]
	\arrow[from=3-4, to=5-4, crossing over]
\end{tikzcd}\]
But now we can apply \cref{Anel lemma} to the following situation (where $ b_{1*}$ still is finitely presented as $K_3$ is)
% https://q.uiver.app/?q=WzAsNixbMCwxLCJLXzMiXSxbMCwyLCJCIl0sWzEsMSwiYl97Mip9S18yJyJdLFsyLDAsImJfezEqfUtfMSciXSxbMSwyLCJDXzIiXSxbMiwxLCJDXzEiXSxbMCwxLCJhXzMiLDJdLFswLDIsImJfezIqfW1fMiJdLFsyLDRdLFsxLDQsIm5fMiIsMl0sWzMsNSwiKGJfezEqfW1fMSlfKmFfMyAiXSxbNSw0LCJuIl0sWzQsMCwiIiwyLHsic3R5bGUiOnsibmFtZSI6ImNvcm5lciJ9fV1d
\[\begin{tikzcd}
	&& {b_{1*}K_1'} \\
	{K_3} & {b_{2*}K_2'} & {C_1} \\
	B & {C_2}
	\arrow["{a_3}"', from=2-1, to=3-1]
	\arrow["{b_{2*}m_2}", from=2-1, to=2-2]
	\arrow[from=2-2, to=3-2]
	\arrow["{n_2}"', from=3-1, to=3-2]
	\arrow["{(b_{1*}m_1)_*a_3 }", from=1-3, to=2-3]
	\arrow["n", from=2-3, to=3-2]
	\arrow["\lrcorner"{anchor=center, pos=0.125, rotate=180}, draw=none, from=3-2, to=2-1]
\end{tikzcd}\]
to exhibit a further factorization 
% https://q.uiver.app/?q=WzAsOCxbMCwwLCJLXzMiXSxbMCwyLCJCIl0sWzIsMCwiYl97Mip9S18yJyJdLFszLDAsImJfezEqfUtfMSciXSxbMiwyLCJDXzIiXSxbMywxLCJDXzEiXSxbMCwxLCJLXzQiXSxbMiwxLCJiX3szKn1iX3syKn1LXzInIl0sWzAsMiwiYl97Mip9bV8yIl0sWzEsNCwibl8yIiwyXSxbMyw1LCIoYl97MSp9bV8xKV8qYV8zICJdLFs1LDQsIm4iXSxbNCwwLCIiLDIseyJzdHlsZSI6eyJuYW1lIjoiY29ybmVyIn19XSxbNiwxLCJhXzMiLDJdLFswLDYsImJfMyIsMl0sWzYsNywiYl97Myp9KGJfezIqfW1fMikiLDFdLFsyLDddLFs3LDRdLFszLDcsIm0iLDEseyJzdHlsZSI6eyJib2R5Ijp7Im5hbWUiOiJkYXNoZWQifX19XSxbNywwLCIiLDEseyJzdHlsZSI6eyJuYW1lIjoiY29ybmVyIn19XV0=
\[\begin{tikzcd}
	{K_3} && {b_{2*}K_2'} & {b_{1*}K_1'} \\
	{K_4} && {b_{3*}b_{2*}K_2'} & {C_1} \\
	B && {C_2}
	\arrow["{b_{2*}m_2}", from=1-1, to=1-3]
	\arrow["{n_2}"', from=3-1, to=3-3]
	\arrow["{(b_{1*}m_1)_*a_3 }", from=1-4, to=2-4]
	\arrow["n", from=2-4, to=3-3]
	\arrow["\lrcorner"{anchor=center, pos=0.125, rotate=180}, draw=none, from=3-3, to=1-1]
	\arrow["{a_4}"', from=2-1, to=3-1]
	\arrow["{b_3}"', from=1-1, to=2-1]
	\arrow["{b_{3*}b_{2*}m_2}"{description}, from=2-1, to=2-3]
	\arrow[from=1-3, to=2-3]
	\arrow[from=2-3, to=3-3]
	\arrow["m"{description}, dashed, from=1-4, to=2-3]
	\arrow["\lrcorner"{anchor=center, pos=0.125, rotate=180}, draw=none, from=2-3, to=1-1]
\end{tikzcd}\]
However, we cannot infer at this step that $m$ commutes with the other part of the diagram. Indeed, one cannot infer that $ m b_{1*}m_1$ and $ (b_{2*}m_2)_*b_{3*} b_{2*}m_2$ commute. However they are equalized by $ n_{2*}a_4$, which provides two parallel lifts of the same situation 
% https://q.uiver.app/?q=WzAsOSxbMCwyLCJLXzMiXSxbMSw0LCJCIl0sWzIsMiwiYl97Mip9S18yJyJdLFs0LDEsImJfezEqfUtfMSciXSxbMyw0LCJDXzIiXSxbNCwzLCJDXzEiXSxbMSwzLCJLXzQiXSxbMywzLCJiX3szKn1iX3syKn1LXzInIl0sWzQsMCwiS18zIl0sWzAsMiwiYl97Mip9bV8yIiwxXSxbMSw0LCJuXzIiLDJdLFszLDUsIihiX3sxKn1tXzEpXyphXzMgIl0sWzUsNCwibiJdLFs0LDAsIiIsMix7InN0eWxlIjp7Im5hbWUiOiJjb3JuZXIifX1dLFs2LDEsImFfNCIsMV0sWzYsNywiYl97Myp9KGJfezIqfW1fMikiLDFdLFs3LDRdLFszLDcsIm0iLDFdLFs4LDMsImJfezEqfW1fMSJdLFswLDEsImFfMyIsMl0sWzIsNCwiKGJfezIqfW1fMilfKmFfMyIsMSx7ImxhYmVsX3Bvc2l0aW9uIjozMH1dLFs4LDIsImJfezIqfW1fMiIsMV0sWzAsOCwiIiwxLHsiY3VydmUiOi0zLCJsZXZlbCI6Miwic3R5bGUiOnsiaGVhZCI6eyJuYW1lIjoibm9uZSJ9fX1dXQ==
\[\begin{tikzcd}
	&&&& {K_3} \\
	&&&& {b_{1*}K_1'} \\
	{K_3} && {b_{2*}K_2'} \\
	& {K_4} && {b_{3*}b_{2*}K_2'} & {C_1} \\
	& B && {C_2}
	\arrow["{b_{2*}m_2}"{description}, from=3-1, to=3-3]
	\arrow["{n_2}"', from=5-2, to=5-4]
	\arrow["{(b_{1*}m_1)_*a_3 }", from=2-5, to=4-5]
	\arrow["n", from=4-5, to=5-4]
	\arrow["\lrcorner"{anchor=center, pos=0.125, rotate=180}, draw=none, from=5-4, to=3-1]
	\arrow["{a_4}"{description}, from=4-2, to=5-2]
	\arrow[from=4-4, to=5-4]
	\arrow["m"{description}, from=2-5, to=4-4]
	\arrow["{b_{1*}m_1}", from=1-5, to=2-5]
	\arrow["{a_3}"', from=3-1, to=5-2]
	\arrow["{(b_{2*}m_2)_*a_3}"{description, pos=0.3}, from=3-3, to=5-4]
	\arrow["{b_{2*}m_2}"{description}, from=1-5, to=3-3]	\arrow["{b_{3*}b_{2*}m_2}"{description}, from=4-2, to=4-4, crossing over]
	\arrow[curve={height=-18pt}, Rightarrow, no head, from=3-1, to=1-5]
\end{tikzcd}\]
so there exists a further factorization $ a_4 = a_5 b_4$ with $ a_5 : K_5 \rightarrow B$ such that 
\begin{align*}
    (b_{4*}b_{3*}b_{2*}m_2)_*b_4 m b_{1*}m_1 &=  (b_{4*}b_{3*}b_{2*}m_2)_*b_4 (b_{2*}m_2)_*b_{3*}b_{2*}m_2 \\
    &=b_{4*}b_{3*}b_{2*}m_2 b_4b_3
\end{align*}
And now by the universal property of the pushout we have an arrow $ m'$ as in the diagram below
% https://q.uiver.app/?q=WzAsOSxbMCw0LCJCIl0sWzIsNSwiQ18yIl0sWzMsNCwiQ18xIl0sWzAsMiwiS181Il0sWzIsMywiYl97Myp9Yl97Mip9S18yJyJdLFszLDIsImJfezQqfWJfezMqfWJfezEqfUtfMSciXSxbMCwwLCJLXzMiXSxbMiwxLCJiX3syKn1LXzInIl0sWzMsMCwiYl97MSp9S18xJyJdLFswLDEsIm5fMiIsMl0sWzIsMSwibiJdLFszLDAsImFfNSIsMl0sWzMsNCwiYl97NCp9Yl97Myp9KGJfezIqfW1fMikiLDFdLFs0LDFdLFswLDIsIm5fMSIsMV0sWzMsNSwiYl97NCp9Yl97Myp9KGJfezEqfW1fMSkiLDEseyJsYWJlbF9wb3NpdGlvbiI6NDB9XSxbNSwyXSxbMiwzLCIiLDEseyJzdHlsZSI6eyJuYW1lIjoiY29ybmVyIn19XSxbNiwzLCJiXzRiXzMiLDJdLFs2LDcsImJfezIqfW1fMiIsMV0sWzcsNF0sWzYsOCwiYl97MSp9bV8xIl0sWzgsNV0sWzQsNiwiIiwxLHsic3R5bGUiOnsibmFtZSI6ImNvcm5lciJ9fV0sWzUsNiwiIiwxLHsic3R5bGUiOnsibmFtZSI6ImNvcm5lciJ9fV0sWzEsMywiIiwxLHsic3R5bGUiOnsibmFtZSI6ImNvcm5lciJ9fV0sWzgsNCwiKGJfezQqfWJfezMqfWJfezIqfW1fMilfKmJfNCBtIiwxXSxbNSw0LCJtJyIsMSx7InN0eWxlIjp7ImJvZHkiOnsibmFtZSI6ImRhc2hlZCJ9fX1dXQ==
\[\begin{tikzcd}[column sep=large]
	{K_3} &&& {b_{1*}K_1'} \\
	&& {b_{2*}K_2'} \\
	{K_5} &&& {b_{4*}b_{3*}b_{1*}K_1'} \\
	&& {b_{3*}b_{2*}K_2'} \\
	B &&& {C_1} \\
	&& {C_2}
	\arrow["{n_2}"', from=5-1, to=6-3]
	\arrow["n", from=5-4, to=6-3]
	\arrow["{a_5}"', from=3-1, to=5-1]
	\arrow["{b_{4*}b_{3*}b_{2*}m_2}"{description}, from=3-1, to=4-3]
	\arrow["{n_1}"{description, pos=0.3}, from=5-1, to=5-4]
	\arrow["{b_{4*}b_{3*}b_{1*}m_1}"{description, pos=0.3}, from=3-1, to=3-4]
	\arrow[from=3-4, to=5-4]
	\arrow["\lrcorner"{anchor=center, pos=0.125, rotate=180}, draw=none, from=5-4, to=3-1]
	\arrow["{b_4b_3}"', from=1-1, to=3-1]
	\arrow["{b_{2*}m_2}"{description}, from=1-1, to=2-3]
	\arrow["{b_{1*}m_1}", from=1-1, to=1-4]
	\arrow[from=1-4, to=3-4]
	\arrow["\lrcorner"{anchor=center, pos=0.125, rotate=180}, draw=none, from=4-3, to=1-1]
	\arrow["\lrcorner"{anchor=center, pos=0.125, rotate=180}, draw=none, from=3-4, to=1-1]
	\arrow["\lrcorner"{anchor=center, pos=0.125, rotate=180}, draw=none, from=6-3, to=3-1]	\arrow[from=2-3, to=4-3, crossing over]
	\arrow[from=4-3, to=6-3, crossing over]
	\arrow["{(b_{4*}b_{3*}b_{2*}m_2)_*b_4 m}"{description}, from=1-4, to=4-3, crossing over]
	\arrow["{m'}"{description}, dashed, from=3-4, to=4-3]
\end{tikzcd}\]
Combining stability under pushouts and right cancellation of maps in $ \mathcal{B}^2_\omega$, we know $ m'$ to be in $\mathcal{B}^2_\omega$; moreover, by right cancellation of pushout squares, the right, bottom square is also a pushout, so that $ n$ is exhibited as a pushout of $m'$ along the canonical inclusion $ (b_{3*}(b_{1*}m_1))_*a_3$.
\end{proof}

In particular, the following says we can lift any finite diagram in the etale generator of $B$ into a diagram of the same shape made of finitely presented etale arrows, from which it can be induced by pushout:

\begin{lemma}\label{lift of finite diagrams}
For any finite diagram $ F :  I \rightarrow \mathcal{G}_B$, there is some $ a : K \rightarrow B$ and some lifts 
% https://q.uiver.app/?q=WzAsMyxbMCwxLCJJIl0sWzEsMSwiXFxtYXRoY2Fse1Z9X0IiXSxbMSwwLCJcXG1hdGhjYWx7Vn1fSyJdLFsyLDEsImFfKiJdLFswLDIsIlxcb3ZlcmxpbmV7Rn0iLDAseyJzdHlsZSI6eyJib2R5Ijp7Im5hbWUiOiJkYXNoZWQifX19XSxbMCwxLCJGIiwyXSxbNCwxLCJcXHNpbWVxIiwxLHsic2hvcnRlbiI6eyJzb3VyY2UiOjIwfSwic3R5bGUiOnsiYm9keSI6eyJuYW1lIjoibm9uZSJ9LCJoZWFkIjp7Im5hbWUiOiJub25lIn19fV1d
\[\begin{tikzcd}
	& {\mathcal{G}_K} \\
	I & {\mathcal{G}_B}
	\arrow["{a_*}", from=1-2, to=2-2]
	\arrow[""{name=0, anchor=center, inner sep=0}, "{\overline{F}}", dashed, from=2-1, to=1-2]
	\arrow["F"', from=2-1, to=2-2]
	\arrow["\simeq"{description}, Rightarrow, draw=none, from=0, to=2-2]
\end{tikzcd}\]
where $ a_*$ is the pushout functor. In particular the transition morphisms of $ F$ are obtained as pushouts of the corresponding transition morphisms of $\overline{F}$. 
\end{lemma}

\begin{proof}
We saw that one can lift morphisms. Here we prove that one can lift finite discrete diagrams and parallel pairs. Let be a discrete set $ (n_i : B \rightarrow C_i)_{i \in I}$ with $I$ finite, with $ n_i = a_{i*}k_i$ and $ k_i : K_i \rightarrow K_i'$. Then, for $ \mathcal{B}_{\omega}\downarrow B$ is filtered, there exists some $a : K \rightarrow B$ and for each $ i \in I$ an arrow $ b_i : K_i \rightarrow K$ such that $ a_i = a b_i $; then in $ \mathcal{G}_K$ one gets the following diagram over $ (n_i)_{i \in I}$%one can compute the colimit $ \colim_{i \in I}b_{i*}K$, which is also the colimit in $ \mathcal{V}$ the discrete finite diagram  
% https://q.uiver.app/?q=WzAsMyxbMCwxLCJLIl0sWzEsMCwiYl97aSp9S19pJyJdLFsyLDEsImJfe2oqfUtfaiciXSxbMCwxLCJiX3tpKn1rX2kiXSxbMCwyLCJiX3tqKn1rX2oiLDJdLFsxLDIsIlxcZGRvdHMiLDEseyJzdHlsZSI6eyJib2R5Ijp7Im5hbWUiOiJub25lIn0sImhlYWQiOnsibmFtZSI6Im5vbmUifX19XV0=
\[\begin{tikzcd}
	& {b_{i*}K_i'} \\
	K && {b_{j*}K_j'}
	\arrow["{b_{i*}k_i}", from=2-1, to=1-2]
	\arrow["{b_{j*}k_j}"', from=2-1, to=2-3]
	\arrow["\ddots"{description}, draw=none, from=1-2, to=2-3]
\end{tikzcd}\]
%and by commutation of colimits and pushouts, we have 
%\[ a_*(\underset{i \in I}{\colim}\; b_{i*}k_i) \simeq \underset{i \in I}{\colim}\; a_*(b_{i*}k_i)  \]
Now consider a parallel pair 
% https://q.uiver.app/?q=WzAsMyxbMCwxLCJCIl0sWzEsMCwiQ18xIl0sWzIsMSwiQ18yIl0sWzAsMSwibl8xIl0sWzAsMiwibl8yIiwyXSxbMSwyLCJnIiwyLHsib2Zmc2V0IjoxfV0sWzEsMiwiZiIsMCx7Im9mZnNldCI6LTF9XV0=
\[\begin{tikzcd}
	& {C_1} \\
	B && {C_2}
	\arrow["{n_1}", from=2-1, to=1-2]
	\arrow["{n_2}"', from=2-1, to=2-3]
	\arrow["f'"', shift right=1, from=1-2, to=2-3]
	\arrow["f", shift left=1, from=1-2, to=2-3]
\end{tikzcd}\]
Then from \cref{lifting of diagrams} there are respectively two lifts 
% https://q.uiver.app/?q=WzAsNixbMCwzLCJCIl0sWzEsMiwiQ18xIl0sWzIsMywiQ18yIl0sWzAsMSwiSyJdLFsyLDEsIktfMiJdLFsxLDAsIktfMSJdLFswLDEsIm5fMSIsMV0sWzAsMiwibl8yIiwyXSxbMSwyLCJmIiwxXSxbMywwLCJhIiwyXSxbNCwyXSxbMyw1LCJrXzEiXSxbNSw0LCJsIl0sWzUsMV0sWzEsMywiIiwyLHsic3R5bGUiOnsibmFtZSI6ImNvcm5lciJ9fV0sWzMsNCwia18yIiwxLHsibGFiZWxfcG9zaXRpb24iOjMwfV0sWzIsMTMsIiIsMSx7ImxldmVsIjoxLCJzdHlsZSI6eyJuYW1lIjoiY29ybmVyIn19XV0=
\[\begin{tikzcd}[row sep=small]
	& {K_1} \\
	K && {K_2} \\
	& {C_1} \\
	B && {C_2}
	\arrow["{n_1}"{description}, from=4-1, to=3-2]
	\arrow["{n_2}"', from=4-1, to=4-3]
	\arrow["f"{description}, from=3-2, to=4-3]
	\arrow["a"', from=2-1, to=4-1]
	\arrow[from=2-3, to=4-3]
	\arrow["{k_1}", from=2-1, to=1-2]
	\arrow["l", from=1-2, to=2-3]
	\arrow[""{name=0, anchor=center, inner sep=0}, from=1-2, to=3-2]
	\arrow["\lrcorner"{anchor=center, pos=0.125, rotate=180}, draw=none, from=3-2, to=2-1]
	\arrow["{k_2}"{description, pos=0.3}, from=2-1, to=2-3, crossing over]
	\arrow["\lrcorner"{anchor=center, pos=0.125, rotate=180}, draw=none, from=4-3, to=0]
\end{tikzcd} \hskip 1cm \begin{tikzcd}[row sep=small]
	& {K_1'} \\
	{K'} && {K_2'} \\
	& {C_1} \\
	B && {C_2}
	\arrow["{n_1}"{description}, from=4-1, to=3-2]
	\arrow["{n_2}"', from=4-1, to=4-3]
	\arrow["{f'}"{description}, from=3-2, to=4-3]
	\arrow["{a'}"', from=2-1, to=4-1]
	\arrow[from=2-3, to=4-3]
	\arrow["{k_1'}", from=2-1, to=1-2]
	\arrow["{l'}", from=1-2, to=2-3]
	\arrow[""{name=0, anchor=center, inner sep=0}, from=1-2, to=3-2]
	\arrow["\lrcorner"{anchor=center, pos=0.125, rotate=180}, draw=none, from=3-2, to=2-1]
	\arrow["{k_2'}"{description, pos=0.3}, from=2-1, to=2-3, crossing over]
	\arrow["\lrcorner"{anchor=center, pos=0.125, rotate=180}, draw=none, from=4-3, to=0]
\end{tikzcd}\]
Now one can find a common refinement $ a'' : K'' \rightarrow B$, $b: K \rightarrow K''$ and $ b' : K' \rightarrow K''$ of $ a$, $a'$. Moreover, by applying upstream \cref{Anel lemma} we can chose this common refinement to be such that there exists a factorization 
% https://q.uiver.app/?q=WzAsNixbMCwyLCJCIl0sWzIsMiwiQ18xIl0sWzEsMCwiSyciXSxbMywwLCJLXzEnIl0sWzAsMSwiSycnIl0sWzIsMSwiYl8qSyciXSxbMCwxLCJuXzEiLDFdLFsyLDMsImtfMSciXSxbMywxXSxbMSwyLCIiLDIseyJzdHlsZSI6eyJuYW1lIjoiY29ybmVyIn19XSxbNCwwLCJhIiwyXSxbNCw1XSxbNSwxXSxbMiwwLCJhJyIsMSx7ImxhYmVsX3Bvc2l0aW9uIjozMH1dLFsxLDQsIiIsMix7InN0eWxlIjp7Im5hbWUiOiJjb3JuZXIifX1dLFsyLDQsImInIiwyXSxbMyw1LCJmIiwyLHsic3R5bGUiOnsiYm9keSI6eyJuYW1lIjoiZGFzaGVkIn19fV1d
\[\begin{tikzcd}
	& {K'} && {K_1'} \\
	{K''} && {b_*K'} \\
	B && {C_1}
	\arrow["{n_1}"{description}, from=3-1, to=3-3]
	\arrow["{k_1'}", from=1-2, to=1-4]
	\arrow[from=1-4, to=3-3]
	\arrow["\lrcorner"{anchor=center, pos=0.125, rotate=180}, draw=none, from=3-3, to=1-2]
	\arrow["a''"', from=2-1, to=3-1]
	\arrow[from=2-3, to=3-3]
	\arrow["{a'}"{description, pos=0.3}, from=1-2, to=3-1]
	\arrow["{b'}"', from=1-2, to=2-1]
	\arrow["c"', dashed, from=1-4, to=2-3]
	\arrow["b_*k_1 "{description}, from=2-1, to=2-3, crossing over]
\end{tikzcd}\]
Now we can push the arrow $ l'$ along $ c$ to get a diagram as below
% https://q.uiver.app/?q=WzAsOSxbMCwyLCJLJyciXSxbMiwxLCJiXypLXzEiXSxbNCwyLCJjXypLXzInIl0sWzMsMiwiYl8qS18yIl0sWzAsNCwiQiJdLFszLDQsIkNfMiJdLFszLDAsIksnXzEiXSxbNSwxLCJLJ18yIl0sWzIsMywiQ18xIl0sWzAsMSwiYl8qa18xIl0sWzEsMiwiY18qbCciXSxbMSwzLCJiXypsIiwxXSxbMCw0LCJhJyciLDJdLFs0LDUsIm5fMiIsMl0sWzMsNSwiKGJfKmtfMilfKmEnJyIsMSx7ImxhYmVsX3Bvc2l0aW9uIjozMH1dLFs2LDEsImMiLDJdLFs2LDcsImwnIl0sWzcsMl0sWzIsNiwiIiwxLHsic3R5bGUiOnsibmFtZSI6ImNvcm5lciJ9fV0sWzEsOCwiKGJfKmtfMSlfKmEnJyIsMSx7ImxhYmVsX3Bvc2l0aW9uIjo4MH1dLFs0LDgsIm5fMSIsMV0sWzgsNSwiZiIsMix7Im9mZnNldCI6MX1dLFsyLDVdLFs4LDUsImYnIiwwLHsib2Zmc2V0IjotMX1dLFs3LDUsIiIsMCx7ImN1cnZlIjotNH1dLFswLDMsImJfKmtfMiIsMV0sWzgsMCwiIiwxLHsic3R5bGUiOnsibmFtZSI6ImNvcm5lciJ9fV0sWzUsMCwiIiwxLHsib2Zmc2V0IjotMSwic3R5bGUiOnsibmFtZSI6ImNvcm5lciJ9fV1d
\[\begin{tikzcd}[column sep=large]
	&&& {K'_1} \\
	&& {b_*K_1} &&& {K'_2} \\
	{K''} &&& {b_*K_2} & {c_*K_2'} \\
	&& {C_1} \\
	B &&& {C_2}
	\arrow["{b_*k_1}", from=3-1, to=2-3]
	\arrow["{c_*l'}", from=2-3, to=3-5]
	\arrow["{b_*l}"{description}, from=2-3, to=3-4]
	\arrow["{a''}"', from=3-1, to=5-1]
	\arrow["{n_2}"', from=5-1, to=5-4]
	\arrow["{(b_*k_2)_*a''}"{description, pos=0.3}, from=3-4, to=5-4]
	\arrow["c"', from=1-4, to=2-3]
	\arrow["{l'}", from=1-4, to=2-6]
	\arrow[from=2-6, to=3-5]
	\arrow["\lrcorner"{anchor=center, pos=0.125, rotate=180}, draw=none, from=3-5, to=1-4]
	\arrow["{(b_*k_1)_*a''}"{description, pos=0.8}, from=2-3, to=4-3]
	\arrow["{n_1}"{description}, from=5-1, to=4-3]
	\arrow["f"', shift right=1, from=4-3, to=5-4]
	\arrow[" \scriptsize{(c_*l')_*((b_*k_1)_*a'')}", from=3-5, to=5-4]
	\arrow["{f'}", shift left=1, from=4-3, to=5-4]
	\arrow["{b_*k_2}"{description}, from=3-1, to=3-4, crossing over]
	\arrow["\lrcorner"{anchor=center, pos=0.125, rotate=180}, draw=none, from=4-3, to=3-1]
	\arrow["\lrcorner"{anchor=center, pos=0.2, rotate=180}, shift left=1, draw=none, from=5-4, to=3-1]
\end{tikzcd}\]
where, by cancellation of pushouts, we have that 
\[ f' = ((b_*k_1)_*a'')_*(c_*l') \]
Finally, %from the fact that $ n_1$ equalizes $ f,f'$ in $ n_2$, we have that
%\[  a''n_2 =  (c_*l')_*((b_*k_1)_*a'') c_*l' b_*k_1 \]
again by \cref{Anel lemma}, we can get a last factorization $ a''' : K''' \rightarrow B$ and $ b'': K'' \rightarrow K'''$ and a factorization as below
% https://q.uiver.app/?q=WzAsMTAsWzAsMSwiSycnIl0sWzIsMCwiYl8qS18xIl0sWzQsMSwiY18qS18yJyJdLFszLDEsImJfKktfMiJdLFswLDUsIkIiXSxbMyw1LCJDXzIiXSxbMiw0LCJDXzEiXSxbMCwzLCJLJycnIl0sWzIsMiwiYicnXypiXypLXzEiXSxbMywzLCJiJydfKmIiXSxbMCwxLCJiXyprXzEiXSxbMSwyLCJjXypsJyJdLFsxLDMsImJfKmwiLDFdLFs0LDUsIm5fMiIsMl0sWzQsNiwibl8xIiwxXSxbNiw1LCJmIiwyLHsib2Zmc2V0IjoxfV0sWzYsNSwiZiciLDAseyJvZmZzZXQiOi0xfV0sWzYsMCwiIiwxLHsic3R5bGUiOnsibmFtZSI6ImNvcm5lciJ9fV0sWzUsMCwiIiwxLHsib2Zmc2V0IjotMSwic3R5bGUiOnsibmFtZSI6ImNvcm5lciJ9fV0sWzAsNywiYicnIiwyXSxbNyw0LCJhJycnIiwyXSxbMSw4LCIoYl8qa18xKV8qYicnIiwxLHsibGFiZWxfcG9zaXRpb24iOjcwfV0sWzcsOCwiYicnXypiXyprXzEiLDFdLFsyLDksImQiLDEseyJzdHlsZSI6eyJib2R5Ijp7Im5hbWUiOiJkYXNoZWQifX19XSxbMyw5LCIoYl8qa18yKV8qYScnIiwxLHsibGFiZWxfcG9zaXRpb24iOjMwfV0sWzgsOSwiXFxzY3JpcHRzaXplKChiXyprXzEpXypiJycpXypiXypsIiwxXSxbNyw5LCJiJydfKmJfKmtfMiIsMV0sWzksNV0sWzgsNl0sWzAsMywiYl8qa18yIiwxXSxbMiw1LCIoY18qbCcpXyooKGJfKmtfMSlfKmEnJykiLDAseyJjdXJ2ZSI6LTJ9XV0=
\[\begin{tikzcd}[column sep=large]
	&& {b_*K_1} \\
	{K''} &&& {b_*K_2} & {c_*K_2'} \\
	&& {b''_*b_*K_1} \\
	{K'''} &&& {b''_*b_*K_2} \\
	&& {C_1} \\
	B &&& {C_2}
	\arrow["{b_*k_1}", from=2-1, to=1-3]
	\arrow["{c_*l'}", from=1-3, to=2-5]
	\arrow["{b_*l}"{description}, from=1-3, to=2-4]
	\arrow["{n_2}"', from=6-1, to=6-4]
	\arrow["{n_1}"{description}, from=6-1, to=5-3]
	\arrow["f"', shift right=1, from=5-3, to=6-4]
	\arrow["{f'}", shift left=1, from=5-3, to=6-4]
	\arrow["\lrcorner"{anchor=center, pos=0.125, rotate=180}, draw=none, from=5-3, to=2-1]
	\arrow["{b''}"', from=2-1, to=4-1]
	\arrow["{a'''}"', from=4-1, to=6-1]
	\arrow["{(b_*k_1)_*b''}"{description, pos=0.7}, from=1-3, to=3-3]
	\arrow["{b''_*b_*k_1}"{description}, from=4-1, to=3-3]
	\arrow["d"{description}, dashed, from=2-5, to=4-4]
	\arrow["{(b_*k_2)_*a''}"{description, pos=0.3}, from=2-4, to=4-4]
	\arrow["{\scriptsize((b_*k_1)_*b'')_*b_*l}"{description}, from=3-3, to=4-4]
	\arrow["{b''_*b_*k_2}"{description, pos=0.4}, from=4-1, to=4-4, crossing over]
	\arrow[from=4-4, to=6-4]
	\arrow[from=3-3, to=5-3]
	\arrow["{b_*k_2}"{description, pos=0.4}, from=2-1, to=2-4, crossing over]
	\arrow["{(c_*l')_*((b_*k_1)_*a'')}", curve={height=-12pt}, from=2-5, to=6-4]
\end{tikzcd}\]
But now, the pair $(dc_*l',b''_*b_*k_2) $ induces a unique arrow 
% https://q.uiver.app/?q=WzAsMixbMCwwLCJiJydfKmJfKktfMSJdLFsyLDAsImInJ18qYktfMiJdLFswLDEsIlxcbGFuZ2xlIGRjXypsJyxiJydfKmJfKmtfMiBcXHJhbmdsZSJdXQ==
\[\begin{tikzcd}
	{b''_*b_*K_1} && {b''_*bK_2}
	\arrow["{\langle dc_*l',b''_*b_*k_2 \rangle}", from=1-1, to=1-3]
\end{tikzcd}\]
 which moreover satisfies $ \langle dc_*l',b''_*b_*k_2 \rangle b''_*b_*k_1  = b''_*b_*k_2$, so that by cancellation of pushouts together with the pushout expression of $n_1$, $n_2$, we have 
\[ f' = ((b''_*b_*k_1)_*a''')_* \langle dc_*l',b''_*b_*k_2 \rangle \]
Hence the parallel pair in $ \mathcal{G}_{k'''}$
% https://q.uiver.app/?q=WzAsMixbMCwwLCJiJydfKmJfKktfMSJdLFsyLDAsImInJ18qYiJdLFswLDEsIlxcc2NyaXB0c2l6ZSgoYl8qa18xKV8qYicnKV8qYl8qbCIsMix7ImxhYmVsX3Bvc2l0aW9uIjo0MCwib2Zmc2V0IjoxfV0sWzAsMSwiXFxsYW5nbGUgZGNfKmwnLGInJ18qYl8qa18yIFxccmFuZ2xlIiwwLHsib2Zmc2V0IjotMX1dXQ==
\[\begin{tikzcd}
	{b''_*b_*k_1} && {b''_*b_*k_2}
	\arrow["{\scriptsize((b_*k_1)_*b'')_*b_*l}"'{pos=0.4}, shift right=1, from=1-1, to=1-3]
	\arrow["{\langle dc_*l',b''_*b_*k_2 \rangle}", shift left=1, from=1-1, to=1-3]
\end{tikzcd}\]
is a lift of the parallel pair $f,f'$ as desired.
\end{proof}

\begin{corollary}
The coslice generator $ \mathcal{G}_B$ is closed under finite colimits in the cocomma $B\downarrow\mathcal{B}$. 
\end{corollary}

\begin{proof}
This is a consequence of the previous results, as we saw we can lift both finite discrete diagrams and parallel pairs to some $ \mathcal{G}_K$, where we can compute the corresponding finite colimit as in $\mathcal{B}^2_\omega$. Formally, for any finite diagram $  F :  I \rightarrow \mathcal{G}_B$, choose a lift $ \overline{ F} : I \rightarrow \mathcal{G}_K$ as provided by \cref{lift of finite diagrams}. As $K$ is finitely presented, $ \mathcal{G}_K$ is a full subcategory of $\mathcal{B}^2_\omega$ closed under finite colimits in $\overrightarrow{\mathcal{B}_\omega}$, so that we have in each $i$ in $I$ a diagram in $\mathcal{B}^2_\omega$
% https://q.uiver.app/?q=WzAsMyxbMCwxLCJLIl0sWzEsMCwiXFxjb2QgKFxcb3ZlcmxpbmV7Rn0oaSkpIl0sWzIsMSwiXFxjb2QoXFxjb2xpbVxcLCBcXG92ZXJsaW5le0Z9KSJdLFswLDIsIlxcY29saW1cXCwgXFxvdmVybGluZXtGfSIsMl0sWzAsMSwiXFxvdmVybGluZXtGfShpKSJdLFsxLDIsInFfaSJdXQ==
\[\begin{tikzcd}
	& {\cod (\overline{F}(i))} \\
	K && {\cod(\colim\, \overline{F})}
	\arrow["{\colim\, \overline{F}}"', from=2-1, to=2-3]
	\arrow["{\overline{F}(i)}", from=2-1, to=1-2]
	\arrow["{\overline{q}_i}", from=1-2, to=2-3]
\end{tikzcd}\]
so that the colimit inclusions $ \overline{q}_i$ of $ \overline{F}$ are in $\mathcal{B}^2_\omega$. Then by commutation of pushouts with colimits, we have 
\[ a_*\colim\, \overline{F} \simeq \colim \, a_*\overline{F} \simeq \colim\, F \]
Hence for each $ i$ in $I$ we have a diagram as below
% https://q.uiver.app/?q=WzAsNixbMCwxLCJLIl0sWzEsMCwiXFxjb2QgKFxcb3ZlcmxpbmV7Rn0oaSkpIl0sWzIsMSwiXFxjb2QoXFxjb2xpbVxcLCBcXG92ZXJsaW5le0Z9KSJdLFswLDMsIkIiXSxbMSwyLCJGKGkpIl0sWzIsMywiXFxjb2QoXFxjb2xpbVxcLCBGKSJdLFswLDIsIlxcY29saW1cXCwgXFxvdmVybGluZXtGfSIsMSx7ImxhYmVsX3Bvc2l0aW9uIjozMH1dLFswLDEsIlxcb3ZlcmxpbmV7Rn0oaSkiXSxbMSwyLCJcXG92ZXJsaW5le3F9X2kiXSxbMCwzLCJhIiwyXSxbMSw0LCJcXG92ZXJsaW5le0Z9KGkpXyphIiwxLHsibGFiZWxfcG9zaXRpb24iOjgwfV0sWzIsNV0sWzMsNCwiRihpKSIsMV0sWzQsMCwiIiwxLHsic3R5bGUiOnsibmFtZSI6ImNvcm5lciJ9fV0sWzQsNSwicV9pIiwxXSxbMyw1LCJcXGNvbGltXFwsIEYiLDJdLFs1LDAsIiIsMSx7Im9mZnNldCI6Miwic3R5bGUiOnsibmFtZSI6ImNvcm5lciJ9fV1d
\[\begin{tikzcd}
	& {\cod (\overline{F}(i))} \\
	K && {\cod(\colim\, \overline{F})} \\
	& {F(i)} \\
	B && {\cod(\colim\, F)}
	\arrow["{\overline{F}(i)}", from=2-1, to=1-2]
	\arrow["{\overline{q}_i}", from=1-2, to=2-3]
	\arrow["a"', from=2-1, to=4-1]
	\arrow["{\overline{F}(i)_*a}"{description, pos=0.8}, from=1-2, to=3-2]
	\arrow[from=2-3, to=4-3]
	\arrow["{F(i)}"{description}, from=4-1, to=3-2]
	\arrow["\lrcorner"{anchor=center, pos=0.125, rotate=180}, draw=none, from=3-2, to=2-1]
	\arrow["{q_i}"{description}, from=3-2, to=4-3]
	\arrow["{\colim\, F}"', from=4-1, to=4-3]
	\arrow["\lrcorner"{anchor=center, pos=0.125, rotate=180}, shift right=2, draw=none, from=4-3, to=2-1]	\arrow["{\colim\, \overline{F}}"{description, pos=0.3}, from=2-1, to=2-3, crossing over]
\end{tikzcd}\]
where the front square is a pushout, exhibiting $ \colim\, F$ as an object of $ \mathcal{G}_B$. Moreover, the colimit inclusions $ q_i$ are obtained as the pushouts 
\[ q_i = (\overline{F}(i)_*a)_*(\overline{q}_i) \]
All of this suffice to proves that $ \mathcal{G}_B$ is finitely cocomplete and closed under finite colimits in $B \downarrow\mathcal{B}$. 
\end{proof}

\begin{lemma}\label{closed under retract}
$\mathcal{G}_B$ is closed under the formation of retracts
\end{lemma}

\begin{proof}
 The crucial argument of this proof will be a lift endomorphisms to endomorphisms. Let be $ l : B \rightarrow C$  be a retract of a pushout $ a_*k$ as below
 % https://q.uiver.app/?q=WzAsNixbMCwxLCJLIl0sWzIsMCwiSyciXSxbMCwzLCJCIl0sWzIsMiwiYV8qSyciXSxbMSwyLCJDIl0sWzIsMywiQyJdLFsyLDQsImwiLDFdLFswLDIsImEiLDJdLFswLDEsImsiXSxbMSwzXSxbMywwLCIiLDEseyJzdHlsZSI6eyJuYW1lIjoiY29ybmVyIn19XSxbNCw1LCIiLDIseyJsZXZlbCI6Miwic3R5bGUiOnsiaGVhZCI6eyJuYW1lIjoibm9uZSJ9fX1dLFs0LDMsInMiXSxbMyw1LCJyIiwyXSxbMiw1LCJsJyIsMV0sWzIsMywiYV8qayIsMV1d
\[\begin{tikzcd}[column sep=large]
	&& {K'} \\
	K \\
	& C & {a_*K'} \\
	B && C
	\arrow["l"{description}, from=4-1, to=3-2]
	\arrow["a"', from=2-1, to=4-1]
	\arrow["k", from=2-1, to=1-3]
	\arrow[from=1-3, to=3-3]
	\arrow["\lrcorner"{anchor=center, pos=0.125, rotate=180}, draw=none, from=3-3, to=2-1]
	\arrow[Rightarrow, no head, from=3-2, to=4-3]
	\arrow["s", from=3-2, to=3-3]
	\arrow["r"', from=3-3, to=4-3]
	\arrow["{l'}"{description}, from=4-1, to=4-3]
	\arrow["{a_*k}"{description}, from=4-1, to=3-3, crossing over]
\end{tikzcd}\]
%Then in particular as retract are absolute, $ C$ is exhibited as a retract of $ a_*K'$, and 
Then consider the corresponding idempotent $e$ in $B \downarrow \mathcal{B} $
% https://q.uiver.app/?q=WzAsNCxbMCwxLCJCIl0sWzEsMCwiYV8qSyciXSxbMiwxLCJDIl0sWzMsMCwiXFxidWxsZXQiXSxbMSwyLCJyIiwxLHsibGFiZWxfcG9zaXRpb24iOjYwfV0sWzAsMiwibCciLDFdLFswLDEsImFfKmsiLDFdLFswLDMsImFfKmsiLDEseyJsYWJlbF9wb3NpdGlvbiI6MzB9XSxbMiwzLCJzIiwyXSxbMSwzLCJlIiwxXV0=
\[\begin{tikzcd}[row sep=large]
	& {a_*K'} && {a_*K'} \\
	B && C
	\arrow["r"{description, pos=0.2}, from=1-2, to=2-3]
	\arrow["{l'}"{description}, from=2-1, to=2-3]
	\arrow["{a_*k}"{description}, from=2-1, to=1-2]
	\arrow["{a_*k}"{description, pos=0.3}, from=2-1, to=1-4, crossing over]
	\arrow["s"', from=2-3, to=1-4]
	\arrow["e"{description}, from=1-2, to=1-4]
\end{tikzcd}\]
Then by \cref{Anel lemma} we know there exists a factorization of pushout squares 
% https://q.uiver.app/?q=WzAsOCxbMCwyLCJCIl0sWzIsMiwiYV8qSyciXSxbMCwxLCJLXzAiXSxbMiwxLCJiX3swKn1LJyJdLFswLDAsIksiXSxbMiwwLCJLJyJdLFszLDEsIlxcYnVsbGV0Il0sWzMsMCwiSyciXSxbMCwxLCJhXyprIiwyXSxbMiwwLCJhXzAiLDJdLFsyLDMsImJfezAqfWsiLDFdLFszLDFdLFs0LDIsImJfMCIsMl0sWzQsNSwiayJdLFs1LDNdLFszLDQsIiIsMix7InN0eWxlIjp7Im5hbWUiOiJjb3JuZXIifX1dLFsxLDIsIiIsMix7InN0eWxlIjp7Im5hbWUiOiJjb3JuZXIifX1dLFs2LDFdLFs3LDYsIm5fKmEiXSxbNywzLCJsIiwxLHsic3R5bGUiOnsiYm9keSI6eyJuYW1lIjoiZGFzaGVkIn19fV1d
\[\begin{tikzcd}
	K && {K'} & {K'} \\
	{K_0} && {b_{0*}K'} & a_*K' \\
	B && {a_*K'}
	\arrow["{a_*k}"', from=3-1, to=3-3]
	\arrow["{a_0}"', from=2-1, to=3-1]
	\arrow["{b_{0*}k}"{description}, from=2-1, to=2-3]
	\arrow[from=2-3, to=3-3]
	\arrow["{b_0}"', from=1-1, to=2-1]
	\arrow["k", from=1-1, to=1-3]
	\arrow[from=1-3, to=2-3]
	\arrow["\lrcorner"{anchor=center, pos=0.125, rotate=180}, draw=none, from=2-3, to=1-1]
	\arrow["\lrcorner"{anchor=center, pos=0.125, rotate=180}, draw=none, from=3-3, to=2-1]
	\arrow["e",from=2-4, to=3-3]
	\arrow["{n_*a}", from=1-4, to=2-4]
	\arrow["l"{description}, dashed, from=1-4, to=2-3]
\end{tikzcd}\]
Then the universal property of the pushout induces a unique map $l$ as below
% https://q.uiver.app/?q=WzAsOSxbMCw0LCJCIl0sWzQsNCwiYV8qSyciXSxbMCwzLCJLXzAiXSxbNSwwXSxbNCwzLCJiX3swKn1LJyJdLFsyLDIsImFfKksnIl0sWzIsMSwiYl97MCp9SyciXSxbMCwyLCJLIl0sWzIsMCwiSyciXSxbMCwxLCJhXyprIiwyXSxbMiwwLCJhXzAiLDJdLFsyLDQsImJfezAqfWsiLDFdLFs0LDFdLFs1LDFdLFs2LDUsIm5fKmEiXSxbNiw0LCJsIiwxLHsic3R5bGUiOnsiYm9keSI6eyJuYW1lIjoiZGFzaGVkIn19fV0sWzAsNSwiYV8qayIsMSx7ImxhYmVsX3Bvc2l0aW9uIjozMH1dLFsyLDYsImJfezAqfUsnIiwxXSxbNywyLCJiXzAiLDJdLFs4LDZdLFs3LDgsImsiLDFdLFs4LDQsImxfMCIsMCx7ImxhYmVsX3Bvc2l0aW9uIjo2MCwiY3VydmUiOi0yfV0sWzEsMTEsIiIsMix7ImxldmVsIjoxLCJzdHlsZSI6eyJuYW1lIjoiY29ybmVyIn19XV0=
\[\begin{tikzcd}
	&& {K'} &&& {} \\
	&& {b_{0*}K'} \\
	K && {a_*K'} \\
	{K_0} &&&& {b_{0*}K'} \\
	B &&&& {a_*K'}
	\arrow["{a_*k}"', from=5-1, to=5-5]
	\arrow["{a_0}"', from=4-1, to=5-1]
	\arrow[from=4-5, to=5-5]
	\arrow[from=3-3, to=5-5]
	\arrow["{n_*a}", from=2-3, to=3-3]
	\arrow["l"{description}, dashed, from=2-3, to=4-5]
	\arrow["{a_*k}"{description, pos=0.3}, from=5-1, to=3-3]
	\arrow["{b_0}"', from=3-1, to=4-1]
	\arrow[from=1-3, to=2-3]
	\arrow["k"{description}, from=3-1, to=1-3]
	\arrow["{l_0}"{pos=0.6}, curve={height=-12pt}, from=1-3, to=4-5]
	\arrow["\lrcorner"{anchor=center, pos=0.125, rotate=180}, draw=none, from=5-5, to=0]	
	\arrow["{b_{0*}k}"{description}, from=4-1, to=2-3]
	\arrow[""{name=0, anchor=center, inner sep=0}, "{b_{0*}k}"{description}, from=4-1, to=4-5, crossing over]
\end{tikzcd}\]
Though this endomorphism $ l : b_{0*}K' \rightarrow b_{0*}K'$ is not necessarily an idempotent, we can still compute the coequalizer 
% https://q.uiver.app/?q=WzAsNCxbMCwxLCJLXzAiXSxbMiwxLCJiX3swKn1LJyJdLFsyLDAsImJfezAqfUsnIl0sWzIsMiwiXFxjb2VxKGwsMV97Yl97MCp9Syd9KSJdLFsyLDEsImwiLDIseyJvZmZzZXQiOjF9XSxbMCwyLCJiX3swKn1LJyIsMV0sWzAsMSwiYl97MCp9ayIsMV0sWzIsMSwiMV97Yl97MCp9Syd9IiwwLHsib2Zmc2V0IjotMiwibGV2ZWwiOjIsInN0eWxlIjp7ImhlYWQiOnsibmFtZSI6Im5vbmUifX19XSxbMCwzLCJxX3tsLDFfe2JfezAqfUsnfX1iX3swKn1rIiwyXSxbMSwzLCJxX3tsLDFfe2JfezAqfUsnfX0iXV0=
\[\begin{tikzcd}
	&& {b_{0*}K'} \\
	{K_0} && {b_{0*}K'} \\
	&& {\coeq(l,1_{b_{0*}K'})}
	\arrow["l"', shift right=1, from=1-3, to=2-3]
	\arrow["{b_{0*}k}", from=2-1, to=1-3]
	\arrow["{b_{0*}k}"{description}, from=2-1, to=2-3]
	\arrow["{1_{b_{0*}K'}}", shift left=2, Rightarrow, no head, from=1-3, to=2-3]
	\arrow["{q_{l,1_{b_{0*}K'}}b_{0*}k}"', from=2-1, to=3-3]
	\arrow["{q_{l,1_{b_{0*}K'}}}", from=2-3, to=3-3]
\end{tikzcd}\]
which where $q_{l,1_{b_{0*}K'}} $ still is in $ \mathcal{B}^2_\omega $ by closure of $ \mathcal{B}_\omega$ under finite colimits. Then by commutation of coequalizers with pushouts, the right square in the following diagram is a pushout 
% https://q.uiver.app/?q=WzAsNixbMSwwLCJiX3swKn1LJyJdLFswLDAsImJfezAqfUsnIl0sWzIsMCwiXFxjb2VxKGwsMV97Yl97MCp9Syd9KSJdLFswLDEsImFfKksnIl0sWzEsMSwiYV8qSyciXSxbMiwxLCJcXGNvZXEoZSwxX3thXypLJ30pIl0sWzEsMCwibCIsMix7Im9mZnNldCI6MX1dLFsxLDAsIjFfe2JfezAqfUsnfSIsMCx7Im9mZnNldCI6LTIsImxldmVsIjoyLCJzdHlsZSI6eyJoZWFkIjp7Im5hbWUiOiJub25lIn19fV0sWzAsMiwicV97bCwxX3tiX3swKn1LJ319Il0sWzEsMywiKGJfezAqfWspXyphXzAiLDJdLFswLDRdLFsyLDVdLFszLDQsIjFfe2FfKksnfSIsMCx7Im9mZnNldCI6LTIsImxldmVsIjoyLCJzdHlsZSI6eyJoZWFkIjp7Im5hbWUiOiJub25lIn19fV0sWzMsNCwiZSIsMix7Im9mZnNldCI6MX1dLFs0LDUsInFfe2UsMV97YV8qSyd9fSIsMl0sWzUsMCwiIiwyLHsic3R5bGUiOnsibmFtZSI6ImNvcm5lciJ9fV1d
\[\begin{tikzcd}[sep=large]
	{b_{0*}K'} & {b_{0*}K'} & {\coeq(l,1_{b_{0*}K'})} \\
	{a_*K'} & {a_*K'} & {\coeq(e,1_{a_*K'})}
	\arrow["l"', shift right=1, from=1-1, to=1-2]
	\arrow["{1_{b_{0*}K'}}", shift left=2, Rightarrow, no head, from=1-1, to=1-2]
	\arrow["{q_{l,1_{b_{0*}K'}}}", from=1-2, to=1-3]
	\arrow["{(b_{0*}k)_*a_0}"', from=1-1, to=2-1]
	\arrow[from=1-2, to=2-2]
	\arrow[from=1-3, to=2-3]
	\arrow["{1_{a_*K'}}", shift left=2, Rightarrow, no head, from=2-1, to=2-2]
	\arrow["e"', shift right=1, from=2-1, to=2-2]
	\arrow["{q_{e,1_{a_*K'}}}"', from=2-2, to=2-3]
	\arrow["\lrcorner"{anchor=center, pos=0.125, rotate=180}, draw=none, from=2-3, to=1-2]
\end{tikzcd}\]
But as $e $ was induced from the pair $ (r,s)$, $(r,s)$ is in turn a splitting of the idempotent $ e$, so that $ r$ happens to coincide with the coequalizer above, that is 
\[  C \simeq \coeq(e, 1_{a_*K'}) \hskip 1cm \textrm{and} \hskip 1cm r = q_{e,1_{a_*K'}} \]
Then by composition of pushouts, $ n$ itself is exhibited as the following pushout of a map in $\mathcal{B}^2_\omega$:
% https://q.uiver.app/?q=WzAsNixbMiwwLCJiX3swKn1LJyJdLFszLDEsIlxcY29lcShsLDFfe2JfezAqfUsnfSkiXSxbMCwzLCJCIl0sWzIsMiwiYV8qSyciXSxbMywzLCJDIl0sWzAsMSwiS18wIl0sWzAsMSwicV97bCwxX3tiX3swKn1LJ319Il0sWzAsMywiKGJfezAqfWspXyphXzAiLDEseyJsYWJlbF9wb3NpdGlvbiI6MzB9XSxbMSw0XSxbMyw0LCJxX3tlLDFfe2FfKksnfX0iLDFdLFs0LDAsIiIsMix7ImxhYmVsX3Bvc2l0aW9uIjo3MCwib2Zmc2V0IjotMSwic3R5bGUiOnsibmFtZSI6ImNvcm5lciJ9fV0sWzUsMCwiYl97MCp9ayJdLFs1LDIsImFfMCIsMl0sWzIsMywiYV97MCp9ayIsMV0sWzMsNSwiIiwwLHsic3R5bGUiOnsibmFtZSI6ImNvcm5lciJ9fV0sWzIsNCwibiIsMl0sWzUsMSwicV97bCwxX3tiX3swKn1LJ319Yl97MCp9ayIsMSx7ImxhYmVsX3Bvc2l0aW9uIjo0MH1dXQ==
\[\begin{tikzcd}
	&& {b_{0*}K'} \\
	{K_0} &&& {\coeq(l,1_{b_{0*}K'})} \\
	&& {a_*K'} \\
	B &&& C
	\arrow["{q_{l,1_{b_{0*}K'}}}", from=1-3, to=2-4]
	\arrow["{(b_{0*}k)_*a_0}"{description, pos=0.2}, from=1-3, to=3-3]
	\arrow[from=2-4, to=4-4]
	\arrow["{q_{e,1_{a_*K'}}}"{description}, from=3-3, to=4-4]
	\arrow["\lrcorner"{anchor=center, pos=0.125, rotate=180}, shift left=1, draw=none, from=4-4, to=1-3]
	\arrow["{b_{0*}k}", from=2-1, to=1-3]
	\arrow["{a_0}"', from=2-1, to=4-1]
	\arrow["{a_{0*}k}"{description}, from=4-1, to=3-3]
	\arrow["\lrcorner"{anchor=center, pos=0.125, rotate=180}, draw=none, from=3-3, to=2-1]
	\arrow["n"', from=4-1, to=4-4]
	\arrow["{q_{l,1_{b_{0*}K'}}b_{0*}k}"{description, pos=0.4}, from=2-1, to=2-4, crossing over]
\end{tikzcd}\]
This proves that the category $ \mathcal{G}_B$ is closed under retracts. 
\end{proof}

\begin{remark}
In fact we could have already guessed that $ \mathcal{G}_B$ had to be closed under retracts, for it was proven before to be closed under finite colimits. However we think it is worth emphasizing why retracts of morphisms in $\mathcal{G}_B$ are still in $\mathcal{G}_B$, as someone proving first that $ \mathcal{G}_B$ is a dense generator would let think at first sight that one should also take retracts of objects of $ \mathcal{G}_B$ to have all the finitely presented objects - which would make everything far more complicated when applying the results for concrete situations: it is somewhat reassuring to see precisely why this hindrance is illusory. %In practice, one generally starts from a left generated factorization system $(\mathcal{L},\mathcal{R})$ in a locally finitely presentable category $\mathcal{B}$, and wants to consider the full subcategory $B\downarrow \mathcal{L}$ of $ B \downarrow \mathcal{B}$ whose objects are $ \mathcal{L}$-maps with $B$ as domain: then what we proved before says that, if one defines $ \mathcal{B}^2_\omega$ as $ \mathcal{L} \cap \overline{\mathcal{B}^2_\omega}$, then $ \mathcal{G}_B$ is the generator of finitely presented objects of $ B \downarrow \mathcal{L}$ which is locally finitely presented. First one can see that $ \mathcal{G}_B$ is a generator, by assumption of left generation; but then one may at first sight only be able to characterize finitely presented objects as retracts of maps in $\mathcal{G}_B$.  \\
\end{remark}

\begin{remark}
In \cite{lurie2009derived}[Warning 2.2.5], we are warned that in the context of $(\infty,1)$-categories, the statement above ceases to be true: the analog of $ \mathcal{G}_B$ are not anymore closed under retract. But we think this is due to the fact that, unlike in 1-categories, splitting of idempotents in $(\infty,1)$-categories ceases to be constructible by mean of finite (co)limits, so that in this context there is no analog to our argument involving expression of the splitting as a coequalizer. 
\end{remark}

\begin{lemma}\label{canonical cone in coslice}
$ \mathcal{G}_B$ is a dense generator of $ B \downarrow \mathcal{B}$. 
\end{lemma}

\begin{proof}
From what was said before on the arrow category, we know that any $ f : B \rightarrow C$ decomposes as a filtered colimit $ f \simeq \colim \, \mathcal{B}^2_\omega \downarrow f$ in the arrow category $ \mathcal{B}^2$. Moreover, from computation of colimits in $ \mathcal{B}^2$, we have both that 
\[ B \simeq \underset{(k, b, b') \in \mathcal{B}^2_\omega \downarrow f}{\colim} \dom (k) \hskip 1cm C \simeq \underset{(k, b, b') \in \mathcal{B}^2_\omega \downarrow f}{\colim} \cod (k) \]
where the colimits range over the squares 
% https://q.uiver.app/?q=WzAsNCxbMCwxLCJCIl0sWzEsMSwiQyJdLFswLDAsIksiXSxbMSwwLCJLJyJdLFsyLDAsImIiLDJdLFsyLDMsImsiXSxbMywxLCJiJyJdLFswLDEsImYiLDJdXQ==
\[\begin{tikzcd}
	K & {K'} \\
	B & C
	\arrow["b"', from=1-1, to=2-1]
	\arrow["k", from=1-1, to=1-2]
	\arrow["{b'}", from=1-2, to=2-2]
	\arrow["f"', from=2-1, to=2-2]
\end{tikzcd}\]
But each of those squares induces uniquely an arrow from the pushout
% https://q.uiver.app/?q=WzAsNSxbMCwxLCJCIl0sWzIsMiwiQyJdLFswLDAsIksiXSxbMSwwLCJLJyJdLFsxLDEsImJfKksnIl0sWzIsMCwiYiIsMl0sWzIsMywiayJdLFszLDEsImInIiwwLHsiY3VydmUiOi0yfV0sWzAsMSwiZiIsMix7ImN1cnZlIjoyfV0sWzAsNCwiYl8qayIsMV0sWzMsNF0sWzQsMiwiIiwwLHsic3R5bGUiOnsibmFtZSI6ImNvcm5lciJ9fV0sWzQsMSwiXFxsYW5nbGUgZixiJyBcXHJhbmdsZSIsMSx7InN0eWxlIjp7ImJvZHkiOnsibmFtZSI6ImRhc2hlZCJ9fX1dXQ==
\[\begin{tikzcd}
	K & {K'} \\
	B & {b_*K'} \\
	&& C
	\arrow["b"', from=1-1, to=2-1]
	\arrow["k", from=1-1, to=1-2]
	\arrow["{b'}", curve={height=-12pt}, from=1-2, to=3-3]
	\arrow["f"', curve={height=12pt}, from=2-1, to=3-3]
	\arrow["{b_*k}"{description}, from=2-1, to=2-2]
	\arrow[from=1-2, to=2-2]
	\arrow["\lrcorner"{anchor=center, pos=0.125, rotate=180}, draw=none, from=2-2, to=1-1]
	\arrow["{\langle f,b' \rangle}"{description}, dashed, from=2-2, to=3-3]
\end{tikzcd}\]
while a morphism in $ \mathcal{B}^2_\omega$ % https://q.uiver.app/?q=WzAsNixbMCwxLCJLXzEiXSxbMiwwLCJLXzEnIl0sWzIsMSwiS18yIl0sWzQsMCwiS18yJyJdLFsxLDIsIkIiXSxbMywxLCJDIl0sWzAsNCwiYl8xIiwyXSxbMiw0LCJiXzIiLDFdLFswLDEsImtfMSJdLFs0LDUsImYiLDFdLFsxLDUsImJfMSciLDEseyJsYWJlbF9wb3NpdGlvbiI6NDB9XSxbMyw1LCJiXzInIl0sWzIsMywia18yIiwxXSxbMCwyLCJsIiwxXSxbMSwzLCJsJyJdXQ==
\[\begin{tikzcd}[sep=large]
	&& {K_1'} && {K_2'} \\
	{K_1} && {K_2} & C \\
	& B
	\arrow["{b_1}"', from=2-1, to=3-2]
	\arrow["{b_2}"{description}, from=2-3, to=3-2]
	\arrow["{k_1}", from=2-1, to=1-3]
	\arrow["f"{description}, from=3-2, to=2-4]
	\arrow["{b_1'}"{description, pos=0.3}, from=1-3, to=2-4]
	\arrow["{b_2'}", from=1-5, to=2-4]
	\arrow["{k_2}"{description}, from=2-3, to=1-5, crossing over, crossing over]
	\arrow["l"{description}, from=2-1, to=2-3]
	\arrow["{l'}", from=1-3, to=1-5]
\end{tikzcd}\]
induces a morphism between the pushouts in the coslice 
% https://q.uiver.app/?q=WzAsNCxbMSwwLCJiX3sxKn1LJ18xIl0sWzMsMCwiYl97Mip9SydfMiJdLFswLDEsIkIiXSxbMiwxLCJDIl0sWzIsMywiZiIsMV0sWzAsMywiXFxsYW5nbGUgZiwgYl8xJyBcXHJhbmdsZSIsMSx7ImxhYmVsX3Bvc2l0aW9uIjoyMH1dLFsxLDMsIlxcbGFuZ2xlIGYsIGJfMicgXFxyYW5nbGUiXSxbMCwxLCJcXGxhbmdsZSBmLCBrX3syKn1iXzJsJ1xccmFuZ2xlIl0sWzIsMCwiYl97MSp9a18xIl0sWzIsMSwiYl97Mip9a18yIiwxLHsibGFiZWxfcG9zaXRpb24iOjcwfV1d
\[\begin{tikzcd}[sep=large]
	& {b_{1*}K'_1} && {b_{2*}K'_2} \\
	B && C
	\arrow["f"{description}, from=2-1, to=2-3]
	\arrow["{\langle f, b_1' \rangle}"'{pos=0.1}, from=1-2, to=2-3]
	\arrow["{\langle f, b_2' \rangle}", from=1-4, to=2-3]
	\arrow["{\langle f, k_{2*}b_2l'\rangle}", from=1-2, to=1-4]
	\arrow["{b_{1*}k_1}", from=2-1, to=1-2]
	\arrow["{b_{2*}k_2}"{description, pos=0.7}, from=2-1, to=1-4, crossing over]
\end{tikzcd}\]
This defines a ``pushout" functor 
% https://q.uiver.app/?q=WzAsMixbMCwwLCJcXG1hdGhjYWx7Qn1eMl9cXG9tZWdhIFxcZG93bmFycm93IGYiXSxbMiwwLCJcXG1hdGhjYWx7R31fQlxcZG93bmFycm93IGYiXSxbMCwxLCJcXHBpIl1d
\[\begin{tikzcd}
	{\mathcal{B}^2_\omega \downarrow f} && {\mathcal{G}_B\downarrow f}
	\arrow["{\langle f, - \rangle}", from=1-1, to=1-3]
\end{tikzcd}\]
sending $ (k,b,b')$ on the induced $ \langle f,b' \rangle : b_*k \rightarrow f $ in the coslice $ B \downarrow \mathcal{B}$. From \cref{lifting of diagrams}, we know that the pushout functor is full; moreover it is essentially surjective by the very definition of $\mathcal{G}_B$, and as $ \mathcal{B}^2_\omega \downarrow f$ is filtered, \cref{esofull cofinal} ensures us that $ \langle f, - \rangle $ is cofinal and that $ \mathcal{G}_B \downarrow f$ is moreover filtered - though the later item can also easily be checked concretely. \\

But now we have a pseudocommutative square of functors 
% https://q.uiver.app/?q=WzAsNCxbMCwwLCJcXG1hdGhjYWx7Qn1eMl9cXG9tZWdhIFxcZG93bmFycm93IGYiXSxbMSwwLCJcXG1hdGhjYWx7R31fQlxcZG93bmFycm93IGYiXSxbMCwxLCJcXG1hdGhjYWx7Qn1eMiJdLFsxLDEsIkIgXFxkb3duYXJyb3cgXFxtYXRoY2Fse0J9Il0sWzAsMSwiXFxsYW5nbGUgZiwtIFxccmFuZ2xlIl0sWzAsMiwiXFxkb20iLDJdLFsxLDMsIlxccGkiXSxbMywyLCIiLDAseyJzdHlsZSI6eyJ0YWlsIjp7Im5hbWUiOiJob29rIiwic2lkZSI6ImJvdHRvbSJ9fX1dLFswLDMsIlxcc2ltZXEiLDEseyJzdHlsZSI6eyJib2R5Ijp7Im5hbWUiOiJub25lIn0sImhlYWQiOnsibmFtZSI6Im5vbmUifX19XV0=
\[\begin{tikzcd}
	{\mathcal{B}^2_\omega \downarrow f} & {\mathcal{G}_B\downarrow f} \\
	{\mathcal{B}^2} & {B \downarrow \mathcal{B}}
	\arrow["{\langle f,- \rangle}", from=1-1, to=1-2]
	\arrow["\dom"', from=1-1, to=2-1]
	\arrow["\pi", from=1-2, to=2-2]
	\arrow["\iota_B", hook', from=2-2, to=2-1]
	\arrow["\simeq"{description}, draw=none, from=1-1, to=2-2]
\end{tikzcd}\]
where $ \pi : \mathcal{G}_B \downarrow f \rightarrow B \downarrow \mathcal{B}$ is the projection sending the triangle 
% https://q.uiver.app/?q=WzAsNSxbMCwwLCJLIl0sWzEsMCwiSyciXSxbMCwxLCJCIl0sWzEsMSwiYl8qSyciXSxbMSwyLCJDIl0sWzAsMiwiYiIsMl0sWzIsMywiYl8qayIsMV0sWzAsMSwiayJdLFsxLDMsImInIl0sWzMsMCwiIiwxLHsic3R5bGUiOnsibmFtZSI6ImNvcm5lciJ9fV0sWzIsNCwiZiIsMl0sWzMsNCwiYSJdXQ==
\[\begin{tikzcd}
	K & {K'} \\
	B & {b_*K'} \\
	& C
	\arrow["b"', from=1-1, to=2-1]
	\arrow["{b_*k}"{description}, from=2-1, to=2-2]
	\arrow["k", from=1-1, to=1-2]
	\arrow["{b'}", from=1-2, to=2-2]
	\arrow["\lrcorner"{anchor=center, pos=0.125, rotate=180}, draw=none, from=2-2, to=1-1]
	\arrow["f"', from=2-1, to=3-2]
	\arrow["a", from=2-2, to=3-2]
\end{tikzcd}\]
on $ b_*k : B \rightarrow b_*K'$ in $\mathcal{B}$. From $ \langle f, - \rangle  $ is cofinal, we have an isomorphism of the corresponding filtered colimit in $\mathcal{B}^2$
\begin{align*}
    f &\simeq \underset{\mathcal{B}^2_\omega \downarrow f}{\colim} \; \cod \\ &\simeq \underset{\mathcal{B}^2_\omega \downarrow f}{\colim} \; \iota_B \pi \langle f, - \rangle\\
    &\simeq  \underset{\mathcal{G}_B \downarrow f}{\colim} \; \iota_B \pi
\end{align*}
But we saw that $ \iota_B : B \downarrow \mathcal{B} \hookrightarrow \mathcal{B}$ creates filtered colimits: hence we already have a filtered colimit in $ B \downarrow \mathcal{B}$
\[ f \simeq \colim \; \mathcal{G}_B \downarrow f  \]
This proves that $\mathcal{G}_B$ is a dense generator in $B \downarrow \mathcal{B}$ consisting of finitely presented objects.
\end{proof}

Putting altogether the previous lemma, we have proven the main result of this part:

\begin{theorem}
Let be $\mathcal{B}$ a locally finitely presentable category and $ B$ an object of $\mathcal{B}$. Then the coslice $ B \downarrow \mathcal{B}$ is locally finitely presentable and we have an equivalence
\[  ( B \downarrow \mathcal{B})_\omega \simeq \mathcal{G}_B \]
\end{theorem}

It is also worth precising that the codomain functor is actually part of this locally finitely presentable structure:

\begin{proposition}
The codomain functor $ \cod : B \downarrow \mathcal{B} \rightarrow \mathcal{B}$ is the right part of a morphism of locally finitely presentable categories.  
\end{proposition}

\begin{proof}
From what was said in part 1, we know that $ \cod$ is both finitary and continuous. Its left adjoint $ \cod^*$ can be defined as sending $C$ in $\mathcal{B}$ to the coproduct inclusion $ i_{B,C}^B : B \rightarrow B + C$ and a morphism $f : C_1 \rightarrow C_2$ on the morphism induced by universal property of the coproduct. It is immediate that this functor preserves finitely presented objects as for a finitely presented $K$ the following square 
% https://q.uiver.app/?q=WzAsNCxbMCwxLCJCIl0sWzAsMCwiMCJdLFsxLDAsIksiXSxbMSwxLCJCK0siXSxbMCwzLCJcXGNvZF4qKEspIiwyXSxbMSwwLCIhX0IiLDJdLFsxLDIsIiFfSyJdLFsyLDNdLFszLDEsIiIsMSx7InN0eWxlIjp7Im5hbWUiOiJjb3JuZXIifX1dXQ==
\[\begin{tikzcd}
	0 & K \\
	B & {B+K}
	\arrow["{\cod^*(K)}"', from=2-1, to=2-2]
	\arrow["{!_B}"', from=1-1, to=2-1]
	\arrow["{!_K}", from=1-1, to=1-2]
	\arrow[from=1-2, to=2-2]
	\arrow["\lrcorner"{anchor=center, pos=0.125, rotate=180}, draw=none, from=2-2, to=1-1]
\end{tikzcd}\]
is a pushout and $ 0$ is always finitely presented. 
\end{proof}

Now observe that in the particular case of the coslice at finitely presented objects, the descriptions above simplifies and we have 

\begin{corollary}
Let be $K$ a finitely presented object. Then we have
\[ ( K \downarrow \mathcal{B})_\omega \simeq K \downarrow \mathcal{B}_\omega \]
\end{corollary}

\begin{proof}
This is because in this case $ \mathcal{G}_K$ is a subcategory of $ \mathcal{B}^2_\omega$ as $ \mathcal{B}_\omega$ is closed under finite colimits, so that in any pushout square 
% https://q.uiver.app/?q=WzAsNCxbMCwwLCJLXzAiXSxbMCwxLCJLIl0sWzEsMCwiS18wJyJdLFsxLDEsImtfKktfMCciXSxbMCwxLCJrIiwyXSxbMCwyLCJrXzAiXSxbMSwzLCJrXyprXzAiLDJdLFsyLDNdLFszLDAsIiIsMix7InN0eWxlIjp7Im5hbWUiOiJjb3JuZXIifX1dXQ==
\[\begin{tikzcd}
	{K_0} & {K_0'} \\
	K & {k_*K_0'}
	\arrow["k"', from=1-1, to=2-1]
	\arrow["{k_0}", from=1-1, to=1-2]
	\arrow["{k_*k_0}"', from=2-1, to=2-2]
	\arrow[from=1-2, to=2-2]
	\arrow["\lrcorner"{anchor=center, pos=0.125, rotate=180}, draw=none, from=2-2, to=1-1]
\end{tikzcd}\]
the pushout $ k_*K_0'$ is finitely presented. Conversely any arrow $ k : K \rightarrow K'$ in $\mathcal{B}_\omega^2$ is its own pushout along the identity map of $K$. 
\end{proof}

To finish, let us examine functoriality of the construction.

\begin{proposition}
For $ f: B_1 \rightarrow B_2$ in $\mathcal{B}$, the adjoint pair 
% https://q.uiver.app/?q=WzAsMixbMCwwLCJCXzIgXFxkb3duYXJyb3cgXFxtYXRoY2Fse0J9Il0sWzIsMCwiQl8xIFxcZG93bmFycm93IFxcbWF0aGNhbHtCfSJdLFswLDEsImZeISIsMix7ImN1cnZlIjoyfV0sWzEsMCwiZl8qIiwyLHsiY3VydmUiOjJ9XSxbMywyLCIiLDIseyJsZXZlbCI6MSwic3R5bGUiOnsibmFtZSI6ImFkanVuY3Rpb24ifX1dXQ==
\[\begin{tikzcd}
	{B_2 \downarrow \mathcal{B}} && {B_1 \downarrow \mathcal{B}}
	\arrow[""{name=0, anchor=center, inner sep=0}, "{f^!}"', curve={height=12pt}, from=1-1, to=1-3]
	\arrow[""{name=1, anchor=center, inner sep=0}, "{f_*}"', curve={height=12pt}, from=1-3, to=1-1]
	\arrow["\dashv"{anchor=center, rotate=-90}, draw=none, from=1, to=0]
\end{tikzcd}\]
defines a morphisms $f\downarrow \mathcal{B}$ of locally finitely presentable categories.\end{proposition}

\begin{proof}
Let us see why the right adjoint $ f^!$, which is precomposition with $f$, is continuous and finitary. Consider any limit in $B_2 \downarrow \mathcal{B}$ of a diagram $ F : I \rightarrow B_2\downarrow \mathcal{B}$, knowing that this limit is the induced map $ \lim F = (F(i))_{i \in I}$. Then by naturality of the universal property of the limit at $f$
% https://q.uiver.app/?q=WzAsNCxbMCwwLCJcXG1hdGhjYWx7Qn1bQl8yLCAgXFxsaW0gXFwsIFxcY29kIFxcLCBcXEZdIl0sWzAsMSwiXFxtYXRoY2Fse0J9W0JfLCAgXFxsaW0gXFwsIFxcY29kIFxcLCBcXEZdIl0sWzEsMCwiXFxtYXRoY2Fse0J9XklbXFxEZWx0YV97Ql8yfSwgIEZdIl0sWzEsMSwiXFxtYXRoY2Fse0J9XklbXFxEZWx0YV97Ql8xfSwgIEZdIl0sWzAsMiwiXFxzaW1lcSIsMSx7InN0eWxlIjp7ImJvZHkiOnsibmFtZSI6Im5vbmUifSwiaGVhZCI6eyJuYW1lIjoibm9uZSJ9fX1dLFsyLDMsIlxcbWF0aGNhbHtCfV5JW1xcRGVsdGFfe2Z9LCAgRl0iXSxbMCwxLCJmXiEiLDJdLFsxLDMsIlxcc2ltZXEiLDEseyJzdHlsZSI6eyJib2R5Ijp7Im5hbWUiOiJub25lIn0sImhlYWQiOnsibmFtZSI6Im5vbmUifX19XV0=
\[\begin{tikzcd}
	{\mathcal{B}[B_2,  \lim \, \cod \, F]} & {\mathcal{B}^I[\Delta_{B_2},  F]} \\
	{\mathcal{B}[B_,  \lim \, \cod \, F]} & {\mathcal{B}^I[\Delta_{B_1},  F]}
		\arrow["\simeq"{description}, draw=none, from=1-1, to=1-2]
	\arrow["\simeq"{description}, draw=none, from=1-1, to=1-2]
	\arrow["{\mathcal{B}^I[\Delta_{f},  F]}", from=1-2, to=2-2]
	\arrow["{f^!}"', from=1-1, to=2-1]
	\arrow["\simeq"{description}, draw=none, from=2-1, to=2-2]
	\end{tikzcd}\]
where $ \Delta$ returns the constant diagram at an object and $\mathcal{B}^I[\Delta_{f},  F] $ is precomposition of a cone over $F $ with tip $ B_2$ with $f$, we know that $ (F(i))_{i \in I}f$ is the universal map induced from the composite cone $ (F(i)f)_{i \in I}$, so that
\begin{align*}
    f^!(\lim \, F)  &= (F(i))_{i \in I} f \\
    &= (F(i) f)_{i \in I} \\
    &= (f^!F(i))_{i \in I} \\
    &= \lim f^! F 
    \end{align*}

Concerning filtered colimits, recall that again for $ F : I \rightarrow \mathcal{B}$ filtered, $\cod \,\colim\, F = \colim \, \cod \, F$ and $ \colim \, F$ is equal to any of the composite $ q_i F(i) $ with  $q_i : \cod \, F(i) \rightarrow \colim \, \cod \, F(i)$ the colimit inclusion. Hence it is immediate that for any $i $ in $I$ we have
\begin{align*}
    f^!(\colim \, F) &= \colim \, F f \\
    &= q_i f_i f \\
    &= \colim \, f^!F
\end{align*}
Finally, compositions of pushouts makes obvious that the left adjoint $ f_*$ sends finitely presented objects of $ B_1 \downarrow \mathcal{B}$ to finitely presented objects in $B_2 \downarrow \mathcal{B}$. 
\end{proof}

Observe that we have in particular a triangle in $ \LFP$
% https://q.uiver.app/?q=WzAsMyxbMCwwLCJCXzIgXFxkb3duYXJyb3cgXFxtYXRoY2Fse0J9Il0sWzIsMCwiQl8xIFxcZG93bmFycm93IFxcbWF0aGNhbHtCfSJdLFsxLDEsIlxcbWF0aGNhbHtCfSJdLFswLDEsImZcXGRvd25hcnJvd1xcbWF0aGNhbHtCfSJdLFswLDIsIlxcY29kIiwyXSxbMSwyLCJcXGNvZCJdLFszLDIsIj0iLDEseyJzaG9ydGVuIjp7InNvdXJjZSI6MjB9LCJzdHlsZSI6eyJib2R5Ijp7Im5hbWUiOiJub25lIn0sImhlYWQiOnsibmFtZSI6Im5vbmUifX19XV0=
\[\begin{tikzcd}
	{B_2 \downarrow \mathcal{B}} && {B_1 \downarrow \mathcal{B}} \\
	& {\mathcal{B}}
	\arrow[""{name=0, anchor=center, inner sep=0}, "{f\downarrow\mathcal{B}}", from=1-1, to=1-3]
	\arrow["\cod"', from=1-1, to=2-2]
	\arrow["\cod", from=1-3, to=2-2]
	\arrow["{=}"{description}, Rightarrow, draw=none, from=0, to=2-2]
\end{tikzcd}\]
induced from the triangle 
% https://q.uiver.app/?q=WzAsMyxbMCwwLCJCXzIiXSxbMiwwLCJCXzEgIl0sWzEsMSwiMCJdLFsxLDAsImYiLDJdLFsyLDAsIiFfe0JfMn0iXSxbMiwxLCIhX3tCXzF9IiwyXV0=
\[\begin{tikzcd}
	{B_2} && {B_1 } \\
	& 0
	\arrow["f"', from=1-3, to=1-1]
	\arrow["{!_{B_2}}", from=2-2, to=1-1]
	\arrow["{!_{B_1}}"', from=2-2, to=1-3]
\end{tikzcd}\]
This exhibits in particular the codomain functor $ \cod : B\downarrow \mathcal{B} \rightarrow \mathcal{B}$ as the transition morphism $ !_B \downarrow \mathcal{B}$. \\

The construction above defines a functor
% https://q.uiver.app/?q=WzAsMixbMCwwLCJcXG1hdGhjYWx7Qn1eXFx7XFxvcH0iXSxbMiwwLCJcXExGUCJdLFswLDEsIigtKVxcZG93bmFycm93XFxtYXRoY2Fse0J9Il1d
\[\begin{tikzcd}
	{\mathcal{B}^{\op}} && \LFP
	\arrow["{(-)\downarrow\mathcal{B}}", from=1-1, to=1-3]
\end{tikzcd}\]

In a future work, we shall describe how this is related to a notion of \emph{spectral 2-site} of a finite limit theory, in the context of a \emph{2-dimensional geometry} associated to Gabriel-Ulmer duality. 

\section{Comma-objects in $\LFP$}

In this section we analyse a similar problem, concerning this time the comma of locally a finitely presentable category at a morphism of locally finitely presentable categories. In \cite{makkai1989accessible}[Proposition 6.1.1] and \cite{adamek1994locally}[Proposition 2.43] it is proven that for any cospan of accessible functors between accessible categories 
% https://q.uiver.app/?q=WzAsMyxbMSwwLCJcXG1hdGhjYWx7QX1fMSJdLFsxLDEsIlxcbWF0aGNhbHtCfSJdLFswLDEsIlxcbWF0aGNhbHtBfV8yIl0sWzAsMSwiRl8xIl0sWzIsMSwiRl8yIiwyXV0=
\[\begin{tikzcd}
	& {\mathcal{A}_2} \\
	{\mathcal{A}_1} & {\mathcal{B}}
	\arrow["{F_2}", from=1-2, to=2-2]
	\arrow["{F_1}"', from=2-1, to=2-2]
\end{tikzcd}\]
the comma category $ F_1\downarrow F_2$ is accessible; then in particular whenever $F_1$ and $F_2$ are locally presentable, preservation of limits makes the comma complete and then locally presentable itself. However, in both of those sources, the strategy of the proof does not allow to control the rank of accessibility of the comma, even when we know the ranks of the functors and the categories involved, with for instance $ \lambda$ a convenient cardinal for all those ranks. The argument relies indeed on the fact that there exists \emph{another cardinal} $ \mu$ larger than $ \lambda$ such that both $F_1$ and $ F_2$ moreover preserves $ \mu$-finitely presented objects. Then one can prove only $ F_1\downarrow F_2$ to be $\mu$-accessible, with its $ \mu$-presentable objects being of the form $f : F_1(K_1) \rightarrow F_2(K_2)$ with $ K_1$ and $ K_2$ being $ \mu$-presentable in $\mathcal{A}_1$ and $\mathcal{A}_2$ respectively. However this is not satisfactory for our purpose. While we do not investigate the general form of the comma involving two functors, this section is devoted to improving this result concerning comma of the form $F \downarrow \mathcal{B}$.\\

In this section, we prove that for a morphism of locally finitely presentable categories $ F : \mathcal{A} \rightarrow \mathcal{B}$ consisting of an adjoint pair $ F^*\dashv F_*$ with $F_*$ finitary, the comma object $ F_*\downarrow \mathcal{B}$ is locally \emph{finitely} presentable. The intuition behind this claim is the following: from Gabriel-Ulmer duality, we know that the 2-category $ \LFP$ of locally finitely presentable categories is equivalent to the opposite 2-category $ \Lex^{\op}$ of small lex categories with lex functors between them. But ongoing investigation on 2-categorical model theory, as well as the work of \cite{bourke2020accessible}, tell us that $ \Lex $ is a \emph{locally finitely bipresentable 2-category} - or at least a \emph{locally presentable 2-category} in \cite{bourke2020accessible} - which indicate its closure under both small pseudolimits and pseudocolimits. Hence if $\Lex$ is closed under pseudocolimit, $ \LFP$ must be closed under pseudolimits, in particular under comma objects.\\

At first sight, such comma objects might be different from those computed in $\Cat$. If $ \Lex$, being KZ-monadic on $\Cat$, is known to be equipped with a 2-functor $ \Lex \rightarrow \Cat$ preserving pseudolimits, controlling its pseudocolimits is more difficult: by duality, while controlling pseudocolimits of $\LFP$ should be easy, controlling its pseudolimits is not. \\

However, we are going to see in this section that actually, not only the comma computed in $\Cat$ is locally finitely presentable as well as all the canonical functors involved, but moreover, that this comma object, as computed in $\Cat$, also has the universal property of the comma object in $ \LFP$. \\

In $\Cat$, the is a weak algebraic factorization system 
\[ (coreflection,\, fibration) \]
which is obtained as follows. Let be $ F : \mathcal{A} \rightarrow \mathcal{B}$ in $\Cat$. Then form the following comma objects 
% https://q.uiver.app/?q=WzAsNCxbMCwwLCJGXFxkb3duYXJyb3cgXFxtYXRoY2Fse0J9Il0sWzEsMCwiXFxtYXRoY2Fse0J9Il0sWzAsMSwiXFxtYXRoY2Fse0F9Il0sWzEsMSwiXFxtYXRoY2Fse0J9Il0sWzIsMywiRiIsMl0sWzEsMywiMV9cXG1hdGhjYWx7Qn0iXSxbMCwyLCJcXHBpX0YiLDJdLFswLDEsIlxccGlfXFxtYXRoY2Fse0J9Il0sWzQsNSwiXFxsYW1iZGFfe0Z9IiwwLHsiY3VydmUiOi0xLCJzaG9ydGVuIjp7InNvdXJjZSI6MjAsInRhcmdldCI6MjB9fV1d
\[  \begin{tikzcd}
	{F\downarrow \mathcal{B}} & {\mathcal{B}} \\
	{\mathcal{A}} & {\mathcal{B}}
	\arrow[""{name=0, anchor=center, inner sep=0}, "F"', from=2-1, to=2-2]
	\arrow[""{name=1, anchor=center, inner sep=0}, "{1_\mathcal{B}}", Rightarrow, no head, from=1-2, to=2-2]
	\arrow["{\pi_F}"', from=1-1, to=2-1]
	\arrow["{\pi_\mathcal{B}}", from=1-1, to=1-2]
	\arrow["{\lambda_{F}}", curve={height=-6pt}, shorten <=4pt, shorten >=4pt, Rightarrow, from=0, to=1]
\end{tikzcd}\]
and observe that the pair $(1_\mathcal{A}, f)$ induces in both cases a canonical equality 2-cell 
% https://q.uiver.app/?q=WzAsMyxbMiwwLCJcXG1hdGhjYWx7Qn0iXSxbMCwwLCJcXG1hdGhjYWx7QX0iXSxbMSwxLCJGIFxcZG93bmFycm93IFxcbWF0aGNhbHtCfSJdLFsxLDIsIjFfe0YoLSl9IiwyXSxbMiwwLCJcXGNvZCIsMl0sWzEsMCwiRiJdLFs1LDIsIj0iLDEseyJzaG9ydGVuIjp7InNvdXJjZSI6MjB9LCJzdHlsZSI6eyJib2R5Ijp7Im5hbWUiOiJub25lIn0sImhlYWQiOnsibmFtZSI6Im5vbmUifX19XV0=
\[\begin{tikzcd}
	{\mathcal{A}} && {\mathcal{B}} \\
	& {F \downarrow \mathcal{B}}
	\arrow["{1_{F}}"', from=1-1, to=2-2]
	\arrow["\cod_F"', from=2-2, to=1-3]
	\arrow[""{name=0, anchor=center, inner sep=0}, "F", from=1-1, to=1-3]
	\arrow["{=}"{description}, Rightarrow, draw=none, from=0, to=2-2]
\end{tikzcd}\]
where $ 1_{F}$ sends $ A$ in $ \mathcal{A}$ to the identity 1-cell $ F(1_A) = 1_{F(A)}$. 

\begin{remark}
Beware that the pair (coreflection,  fibration) is not a strong orthogonality structure, only a weak one. However the factorization above is canonical up to equality, hence the qualificative of \emph{algebraic}. 
\end{remark}

Now we apply this factorization to the direct image of a morphism of locally finitely presentable categories. Let be $ F : \mathcal{A} \rightarrow \mathcal{B}$ a $\LFP$ morphism, consisting of a finitary continuous functor $ F_* : \mathcal{A} \rightarrow \mathcal{B}$, with its corresponding left adjoint $ F^*$, which is known from \cref{Ladj of finitary functor} to restrict to a functor $ F^*_\omega : \mathcal{B}_\omega \rightarrow \mathcal{A}_\omega$. Then one can consider the (coreflection, fibration) factorization of $ F_*$ in $\Cat$
\[ \begin{tikzcd}
	{\mathcal{A}} && {\mathcal{B}} \\
	& {F_* \downarrow \mathcal{B}}
	\arrow["{1_{F_*}}"', from=1-1, to=2-2]
	\arrow["\cod{F_*}"', from=2-2, to=1-3]
	\arrow[""{name=0, anchor=center, inner sep=0}, "F_*", from=1-1, to=1-3]
	\arrow["{=}"{description}, Rightarrow, draw=none, from=0, to=2-2]
\end{tikzcd} \]

Since $ F_*$ is itself finitary and continuous, we know that limits and filtered colimits in $ F_*\downarrow \mathcal{B}$ are computed in the arrow category $ \mathcal{B}$ as respectively 
% https://q.uiver.app/?q=WzAsNCxbMCwxLCJGKEFfaSkiXSxbMSwxLCJCX2kiXSxbMCwwLCJGKFxcdW5kZXJzZXR7aSBcXGluIEl9e1xcbGltfVxcLEFfaSkiXSxbMSwwLCJcXHVuZGVyc2V0e2kgXFxpbiBJfXtcXGxpbX1cXCxCX2kiXSxbMCwxLCJmX2kiLDJdLFsyLDAsIkYocV9pKSIsMl0sWzMsMSwicSdfaSJdLFsyLDMsIlxcdW5kZXJzZXR7aSBcXGluIEl9e1xcbGltfVxcLGZfaSJdXQ==
\[\begin{tikzcd}
	{F(\underset{i \in I}{\lim}\,A_i)} & {\underset{i \in I}{\lim}\,B_i} \\
	{F(A_i)} & {B_i}
	\arrow["{f_i}"', from=2-1, to=2-2]
	\arrow["{F(q_i)}"', from=1-1, to=2-1]
	\arrow["{q'_i}", from=1-2, to=2-2]
	\arrow["{\underset{i \in I}{\lim}\,f_i}", from=1-1, to=1-2]
\end{tikzcd} \hskip 1cm \begin{tikzcd}
	{F(A_i)} & {B_i} \\
	{F(\underset{i \in J}{\colim}\,A_i)} & {\underset{i \in J}{\colim}\,B_i}
	\arrow["{f_i}", from=1-1, to=1-2]
	\arrow["{F(q_i)}"', from=1-1, to=2-1]
	\arrow["{q'_i}", from=1-2, to=2-2]
	\arrow["{\underset{i \in J}{\colim}\,f_i}"', from=2-1, to=2-2]
\end{tikzcd}\]
(whenever $ J$ is filtered) with $ F(\lim_{i \in I} A_i) \simeq\lim_{i \in I} F(A_i)  $ and $F(\colim_{i \in J} A_i) \simeq \colim_{i \in J} F(A_i)  $.\\

As a consequence, $\cod_{F_*}$ is finitary and continuous, as well as $ 1_{F_*(-)} $ and its pseudoretract $ \pi_{F_*} $ sending $ f: F_*(A) \rightarrow B$ on $A$. From what follows we shall see they are actually part of LFP morphisms.  \\
% and hence is the right part of a LFP morphism $ F_*\downarrow \mathcal{B} \rightarrow \mathcal{B}$. 

Most of this section will be devoted to prove that $ F_*\downarrow \mathcal{B}$ is finitely presented. As in section 2, we first guess a candidate for the generator for finitely presented objects, and prove it to be effectively such. Then again proving it to be closed under retract also proves it to contain all finitely presented objects. Finally, the previous discussion above ensuring existence of limits in the comma, we will know it to be locally finitely presented.

\begin{definition}
Define the \emph{generator of the comma} as the full subcategory $ \mathcal{G}_{F_*}$ having as objects all the arrows $ f : F_*(M)\rightarrow B$ with $M$ a finitely presented object of $\mathcal{A}_\omega$ and such that there exists a pushout square 
% https://q.uiver.app/?q=WzAsNCxbMCwwLCJLIl0sWzEsMCwiSyciXSxbMCwxLCJGXyooTSkiXSxbMSwxLCJCIl0sWzAsMiwiYSIsMl0sWzAsMSwiayJdLFsxLDMsImEnIl0sWzIsMywiZiIsMl0sWzMsMCwiIiwxLHsic3R5bGUiOnsibmFtZSI6ImNvcm5lciJ9fV1d
\[\begin{tikzcd}
	K & {K'} \\
	{F_*(M)} & B
	\arrow["a"', from=1-1, to=2-1]
	\arrow["k", from=1-1, to=1-2]
	\arrow["{a'}", from=1-2, to=2-2]
	\arrow["f"', from=2-1, to=2-2]
	\arrow["\lrcorner"{anchor=center, pos=0.125, rotate=180}, draw=none, from=2-2, to=1-1]
\end{tikzcd}\]
\end{definition}

\begin{remark}
From both $ \mathcal{A}_\omega$ and $ \mathcal{B}_\omega$ are essentially small, and $\mathcal{B}$ being locally small, $\mathcal{G}_{F_*}$ is also essentially small.
\end{remark}

\begin{lemma}
$ \mathcal{G}_{F_*}$ consists of finitely presented objects in $F_*\downarrow\mathcal{B}$.
\end{lemma}

\begin{proof}
 Let be $ f_{(-)} : I \rightarrow F_*\downarrow \mathcal{B}$ a filtered diagram, with $ f_i : F_*(A_i) \rightarrow B_i$. Now let be a diagram as below:
% https://q.uiver.app/?q=WzAsNixbMCwwLCJLIl0sWzEsMCwiSyciXSxbMCwxLCJGXyooTSkiXSxbMSwxLCJCIl0sWzAsMiwiRl8qKFxcdW5kZXJzZXR7aSBcXGluIEl9e1xcY29saW19IFxcLCBBX2kpIl0sWzEsMiwiXFx1bmRlcnNldHtpIFxcaW4gSX17XFxjb2xpbX0gXFwsQl9pIl0sWzAsMiwiYSIsMl0sWzAsMSwiayJdLFsxLDMsImEnIl0sWzIsMywiZiIsMV0sWzMsMCwiIiwxLHsic3R5bGUiOnsibmFtZSI6ImNvcm5lciJ9fV0sWzIsNCwicyIsMl0sWzMsNSwicyciXSxbNCw1LCJcXHVuZGVyc2V0e2kgXFxpbiBJfXtcXGNvbGltfSBcXCxmX2kiLDJdXQ==
\[\begin{tikzcd}
	K & {K'} \\
	{F_*(M)} & s_*K' \\
	{F_*(\underset{i \in I}{\colim} \, A_i)} & {\underset{i \in I}{\colim} \,B_i}
	\arrow["s"', from=1-1, to=2-1]
	\arrow["k", from=1-1, to=1-2]
	\arrow["{k_*s}", from=1-2, to=2-2]
	\arrow["s_*k"{description}, from=2-1, to=2-2]
	\arrow["\lrcorner"{anchor=center, pos=0.125, rotate=180}, draw=none, from=2-2, to=1-1]
	\arrow["F_*(a)"', from=2-1, to=3-1]
	\arrow["{b}", from=2-2, to=3-2]
	\arrow["{\underset{i \in I}{\colim} \,f_i}"', from=3-1, to=3-2]
\end{tikzcd}\]
Then for $ M$ is finitely presented in $\mathcal{A}$ we have a lift for some $i \in I$
% https://q.uiver.app/?q=WzAsMyxbMSwwLCJNIl0sWzEsMSwiXFx1bmRlcnNldHtpIFxcaW4gSX17XFxjb2xpbX1cXCwgQV9pIl0sWzAsMSwiQV9pIl0sWzAsMiwiXFxvdmVybGluZXthfSIsMl0sWzIsMSwicV9pIiwyXSxbMCwxLCJhIl1d
\[\begin{tikzcd}
	& M \\
	{A_i} & {\underset{i \in I}{\colim}\, A_i}
	\arrow["{\overline{a}}"', from=1-2, to=2-1]
	\arrow["{q_i}"', from=2-1, to=2-2]
	\arrow["a", from=1-2, to=2-2]
\end{tikzcd}\]
Moreover, this factorization also provides by precomposition with $a$ a lift of $ F_*(a)s$ witnessing finitely presentedness of $K$ relatively to the filtered colimit $F_*(\colim_{i \in I} A_i) \simeq \colim_{i \in I}F_*(A_i) $:
% https://q.uiver.app/?q=WzAsNCxbMSwxLCJGXyooTSkiXSxbMSwyLCJGXyooXFx1bmRlcnNldHtpIFxcaW4gSX17XFxjb2xpbX1cXCwgQV9pKSJdLFswLDIsIkZfKihBX2kpIl0sWzEsMCwiSyJdLFswLDIsIkZfKihcXG92ZXJsaW5le2F9KSIsMl0sWzIsMSwiRl8qKHFfaSkiLDJdLFswLDEsIkZfKihhKSJdLFszLDAsInMiLDJdXQ==
\[\begin{tikzcd}
	& K \\
	& {F_*(M)} \\
	{F_*(A_i)} & {F_*(\underset{i \in I}{\colim}\, A_i)}
	\arrow["{F_*(\overline{a})}"', from=2-2, to=3-1]
	\arrow["{F_*(q_i)}"', from=3-1, to=3-2]
	\arrow["{F_*(a)}", from=2-2, to=3-2]
	\arrow["s"', from=1-2, to=2-2]
\end{tikzcd}\]
On the other side, the arrow $ k : K \rightarrow K'$ is finitely presented in $\mathcal{B}^2$, and hence also lifts as follows for some $i' \in I$:
% https://q.uiver.app/?q=WzAsNixbMCwwLCJLIl0sWzEsMCwiSyciXSxbMCwxLCJGXyooQV97aSd9KSJdLFsxLDEsIkJfe2knfSJdLFsxLDIsIkZfKihcXHVuZGVyc2V0e2kgXFxpbiBJfXtcXGNvbGltfSBcXCwgQV9pKSJdLFsyLDIsIlxcdW5kZXJzZXR7aSBcXGluIEl9e1xcY29saW19IFxcLEJfaSJdLFswLDIsInQiLDJdLFswLDEsImsiXSxbMSwzLCJ0JyJdLFsyLDMsImZfe2knfSIsMV0sWzMsMCwiIiwxLHsic3R5bGUiOnsibmFtZSI6ImNvcm5lciJ9fV0sWzIsNCwiRl8qKHFfe2knfSkiLDJdLFszLDUsInEnX3tpJ30iXSxbNCw1LCJcXHVuZGVyc2V0e2kgXFxpbiBJfXtcXGNvbGltfSBcXCxmX2kiLDJdLFswLDQsIkZfKihhKXMiLDEseyJsYWJlbF9wb3NpdGlvbiI6MzAsImN1cnZlIjotMn1dLFsxLDUsImJrXyphIiwwLHsiY3VydmUiOi0yfV1d
\[\begin{tikzcd}[column sep=large]
	K & {K'} \\
	{F_*(A_{i'})} & {B_{i'}} \\
	& {F_*(\underset{i \in I}{\colim} \, A_i)} & {\underset{i \in I}{\colim} \,B_i}
	\arrow["t"', from=1-1, to=2-1]
	\arrow["k", from=1-1, to=1-2]
	\arrow["{t'}", from=1-2, to=2-2]
	\arrow["{f_{i'}}"{description}, from=2-1, to=2-2]
	\arrow["{F_*(q_{i'})}"', from=2-1, to=3-2]
	\arrow["{q'_{i'}}", from=2-2, to=3-3]
	\arrow["{\underset{i \in I}{\colim} \,f_i}"', from=3-2, to=3-3]
	\arrow["{F_*(a)s}"{description, pos=0.3}, curve={height=-12pt}, from=1-1, to=3-2, crossing over]
	\arrow["{bk_*a}", curve={height=-12pt}, from=1-2, to=3-3, crossing over]
\end{tikzcd}\]
Hence we have two parallel lifts $(F_*(\overline{a})s, i)$ and $ (t,i')$ of $ F_*(a)s$, so by finite presentedness of $K$ we know there exists some $ j \in I$ and a common refinement $ d: i \rightarrow j$, $ d': i' \rightarrow j$ of those parallel lifts
% https://q.uiver.app/?q=WzAsNSxbMSwyLCJGXyooQV9qKSJdLFswLDEsIkZfKihBX2kpIl0sWzEsMCwiSyJdLFsyLDEsIkZfKihBX3tpJ30pIl0sWzEsMywiRl8qKFxcdW5kZXJzZXR7aSBcXGluIEl9e1xcY29saW19XFwsIEFfaSkiXSxbMSwwLCJGXyooZl9kKSIsMV0sWzIsMywidCJdLFszLDAsIkZfKihmX3tkJ30pIiwxXSxbMiwxLCJGXyooXFxvdmVybGluZXthfSlzIiwyXSxbMCw0XSxbMSw0LCJGXyoocV9pKSIsMix7ImN1cnZlIjozfV0sWzMsNCwiRl8qKHFfe2knfSkiLDAseyJjdXJ2ZSI6LTN9XV0=
\[\begin{tikzcd}
	& K \\
	{F_*(A_i)} && {F_*(A_{i'})} \\
	& {F_*(A_j)} \\
	& {F_*(\underset{i \in I}{\colim}\, A_i)}
	\arrow["{F_*(f_d)}"{description}, from=2-1, to=3-2]
	\arrow["t", from=1-2, to=2-3]
	\arrow["{F_*(f_{d'})}"{description}, from=2-3, to=3-2]
	\arrow["{F_*(\overline{a})s}"', from=1-2, to=2-1]
	\arrow["{F_*(q_j)}",from=3-2, to=4-2]
	\arrow["{F_*(q_i)}"', curve={height=18pt}, from=2-1, to=4-2]
	\arrow["{F_*(q_{i'})}", curve={height=-18pt}, from=2-3, to=4-2]
\end{tikzcd}\]
Inserting this common refinement in the following diagram
% https://q.uiver.app/?q=WzAsNyxbMSwyLCJGXyooQV9qKSJdLFswLDEsIkZfKihNKSJdLFsxLDAsIksiXSxbMSwxLCJGXyooQV97aSd9KSJdLFsyLDAsIksnIl0sWzIsMSwiQl9pIl0sWzIsMiwiQl9qIl0sWzIsMywidCIsMV0sWzMsMCwiRl8qKGZfe2QnfSkiLDFdLFsyLDEsInMiLDJdLFsyLDQsImsiXSxbNCw1LCJ0JyJdLFs1LDYsImdfe2QnfSJdLFswLDYsImZfaiIsMV0sWzMsNSwiZl9pIiwxXSxbMSwwLCJGXyooZl9kKUZfKihcXG92ZXJsaW5le2F9KSIsMl1d
\[\begin{tikzcd}
	& K & {K'} \\
	{F_*(M)} & {F_*(A_{i'})} & {B_i} \\
	& {F_*(A_j)} & {B_j}
	\arrow["t"{description}, from=1-2, to=2-2]
	\arrow["{F_*(f_{d'})}"{description}, from=2-2, to=3-2]
	\arrow["s"', from=1-2, to=2-1]
	\arrow["k", from=1-2, to=1-3]
	\arrow["{t'}", from=1-3, to=2-3]
	\arrow["{g_{d'}}", from=2-3, to=3-3]
	\arrow["{f_j}"{description}, from=3-2, to=3-3]
	\arrow["{f_i}"{description}, from=2-2, to=2-3]
	\arrow["{F_*(f_d)F_*(\overline{a})}"', from=2-1, to=3-2]
\end{tikzcd}\]
provides us with a commuting square equalizing $ (a, k)$, inducing a factorization through the pushout
% https://q.uiver.app/?q=WzAsNixbMSwyLCJGXyooQV9qKSJdLFswLDEsIkZfKihNKSJdLFsxLDAsIksiXSxbMiwwLCJLJyJdLFsyLDIsIkJfaiJdLFsxLDEsInNfKksnIl0sWzIsMSwicyIsMl0sWzIsMywiayJdLFswLDQsImZfaiIsMV0sWzEsMCwiRl8qKGZfZClGXyooXFxvdmVybGluZXthfSkiLDJdLFsxLDUsInNfKmsiLDFdLFszLDUsImtfKnMiLDFdLFszLDQsImdfe2QnfXQnIl0sWzUsNCwiXFxvdmVybGluZXtifSIsMSx7InN0eWxlIjp7ImJvZHkiOnsibmFtZSI6ImRhc2hlZCJ9fX1dLFs1LDYsIiIsMSx7ImxldmVsIjoxLCJzdHlsZSI6eyJuYW1lIjoiY29ybmVyIn19XV0=
\[\begin{tikzcd}
	& K & {K'} \\
	{F_*(M)} & {s_*K'} \\
	& {F_*(A_j)} & {B_j}
	\arrow[""{name=0, anchor=center, inner sep=0}, "s"', from=1-2, to=2-1]
	\arrow["k", from=1-2, to=1-3]
	\arrow["{f_j}"{description}, from=3-2, to=3-3]
	\arrow["{F_*(f_d)F_*(\overline{a})}"', from=2-1, to=3-2]
	\arrow["{s_*k}"{description}, from=2-1, to=2-2]
	\arrow["{k_*s}"{description}, from=1-3, to=2-2]
	\arrow["{g_{d'}t'}", from=1-3, to=3-3]
	\arrow["{\overline{b}}"{description}, dashed, from=2-2, to=3-3]
	\arrow["\lrcorner"{anchor=center, pos=0.125, rotate=180}, draw=none, from=2-2, to=0]
\end{tikzcd}\]
This provides a lift $((f_d\overline{a}, \overline{ b}), j)$ of $ (a,b)$ in the comma $ F_*\downarrow \mathcal{B}$ as desired. \\

Now we must check that any two parallel lifts have a common refinement. Suppose we have a situation as below 
% https://q.uiver.app/?q=WzAsMTAsWzEsMSwiRl8qKE0pIl0sWzAsMiwiRl8qKEFfe2l9KSJdLFsxLDQsIkZfKihcXHVuZGVyc2V0e2kgXFxpbiBJfXtcXGNvbGltfSBcXCwgQV9pKSJdLFsyLDQsIlxcdW5kZXJzZXR7aSBcXGluIEl9e1xcY29saW19IFxcLEJfaSJdLFsyLDEsInNfKksnIl0sWzEsMiwiQl97aX0iXSxbMiwzLCJGXyooQV97aSd9KSJdLFszLDMsIkJfe2knfSJdLFsxLDAsIksiXSxbMiwwLCJLJyJdLFsxLDIsIkZfKihxX3tpfSkiLDJdLFsyLDMsIlxcdW5kZXJzZXR7aSBcXGluIEl9e1xcY29saW19IFxcLGZfaSIsMl0sWzAsMSwiYSIsMl0sWzAsNCwic18qayIsMV0sWzQsNSwiYiIsMV0sWzEsNSwiZl97aX0iLDFdLFs1LDMsInEnX3tpfSIsMV0sWzYsNywiZl97aSd9IiwxXSxbNywzLCJxJ197aSd9Il0sWzQsNywiYiciXSxbOCw5LCJrIl0sWzgsMCwicyIsMl0sWzksNF0sWzQsOCwiIiwwLHsic3R5bGUiOnsibmFtZSI6ImNvcm5lciJ9fV0sWzAsNiwiYSciLDFdLFs2LDIsIkZfKihxX3tpJ30pIiwxXV0=
\[\begin{tikzcd}
	& K & {K'} \\
	& {F_*(M)} & {s_*K'} \\
	{F_*(A_{i})} & {\;B_{i}} \\
	&& {F_*(A_{i'})} & {B_{i'}} \\
	& {F_*(\underset{i \in I}{\colim} \, A_i)} & {\underset{i \in I}{\colim} \,B_i}
	\arrow["{F_*(q_{i})}"', from=3-1, to=5-2]
	\arrow["{\underset{i \in I}{\colim} \,f_i}"', from=5-2, to=5-3]
	\arrow["F_*(a)"', from=2-2, to=3-1]
	\arrow["{s_*k}"{description}, from=2-2, to=2-3]
	\arrow["b"{description}, from=2-3, to=3-2]
	\arrow["{f_{i}}"{description}, from=3-1, to=3-2]
	\arrow["{q'_{i}}"{description}, from=3-2, to=5-3]
	\arrow["{f_{i'}}"{description}, from=4-3, to=4-4]
	\arrow["{q'_{i'}}", from=4-4, to=5-3]
	\arrow["{b'}", from=2-3, to=4-4]
	\arrow["k", from=1-2, to=1-3]
	\arrow["s"', from=1-2, to=2-2]
	\arrow[from=1-3, to=2-3]
	\arrow["\lrcorner"{anchor=center, pos=0.125, rotate=180}, draw=none, from=2-3, to=1-2]
	\arrow["{F_*(a')}"{description}, from=2-2, to=4-3, crossing over]
	\arrow["{F_*(q_{i'})}"{description}, from=4-3, to=5-2, crossing over]
\end{tikzcd}\]
where the two lifts $ ((a,b), i)$ and $ ((a',b'),i')$ define the same arrow $ s_*k \rightarrow \colim_{i \in I}f_i$ in $ F_*\downarrow \mathcal{B}$. Then from $M$ is finitely presentable, there is a common refinement $ d : i \rightarrow j$ and $ d':i' \rightarrow j $ for the parallel lifts $ a,a'$ in $\mathcal{A}$. On the other hand we also have a common refinement $ e : i \rightarrow j'$ and $ e' : i' \rightarrow j'$ for the parallel lifts $ ((F_*(a)s, b k*s),i)$ and $ ((F_*(a')s, b' k*s),i')$ from $k$ to $ \colim_{i \in I}f_i$ in $\mathcal{B}^2$. Moreover, because the left part of those parallel lifts factorize through $ s$, we are provided with two further parallel lifts $ (da = d'a',j) $ and $ (ea=ea', j')$ in $ \mathcal{A}$, which have hence themselves a further refinement $ h : j \rightarrow l$, $ h': j' \rightarrow l$ in $I$. Hence by a similar argument as above, this common refinement equalizes $s $ and $ k$, hence factorizes through the pushout. This achieves to prove that $ s_*k$ is finitely presented in $ F_*\downarrow \mathcal{B}$. (Beware that we need to consider a refinement relatively to $M$ and a refinement relatively to $k$ separately before refining them jointly, as the sole common refinement $(e,e')$ relatively to $k$ may fail to equalize $ F_*(a)$ and $ F_*(a')$.) \\
\end{proof}

\begin{lemma}
$ \mathcal{G}_{F_*}$ form a dense generator.
\end{lemma}

\begin{proof}
Let be $ f : F_*(A) \rightarrow B$. Before anything, remark that the canonical cone $ A \simeq \mathcal{B}_\omega \downarrow A$ of $A $ is sent by $F_*$ to a colimiting cone
\[ F_*(A) \simeq  \underset{ a: M \rightarrow A\atop M \in \mathcal{A}_\omega}{\colim} F_*(M) \] 
This fact will be used twice in the following. \\

First, we must prove that the category consisting of all arrows of the form $ s_*k \rightarrow f$ is filtered in $ F_*\downarrow \mathcal{B}$. Let be a situation as below: 
% https://q.uiver.app/?q=WzAsMTAsWzAsMCwiS18xIl0sWzIsMCwiS18xJyJdLFsxLDEsIktfMiJdLFszLDEsIktfMiciXSxbMCwyLCJGXyooTV8xKSJdLFsxLDMsIkZfKihNXzIpIl0sWzIsMiwic197MSp9S18xJyJdLFszLDMsInNfezIqfUtfMiciXSxbMSw0LCJGXyooQSkiXSxbMyw0LCJCIl0sWzAsNCwic18xIiwyXSxbMCwxLCJrXzEiXSxbMSw2XSxbNCw2LCJzX3sxKn1rXzEiLDEseyJsYWJlbF9wb3NpdGlvbiI6MzB9XSxbNCw4LCJGXyooYV8xKSIsMl0sWzUsOCwiRl8qKGFfMikiLDFdLFs2LDksImJfMSIsMSx7ImxhYmVsX3Bvc2l0aW9uIjozMH1dLFs3LDksImJfMiJdLFszLDddLFsyLDMsImtfMiIsMSx7ImxhYmVsX3Bvc2l0aW9uIjo0MH1dLFsyLDUsInNfMiIsMSx7ImxhYmVsX3Bvc2l0aW9uIjozMH1dLFs2LDAsIiIsMix7InN0eWxlIjp7Im5hbWUiOiJjb3JuZXIifX1dLFs3LDIsIiIsMCx7InN0eWxlIjp7Im5hbWUiOiJjb3JuZXIifX1dLFs4LDksImYiLDJdLFs1LDcsInNfezIqfWtfMiIsMV1d
\[\begin{tikzcd}[row sep=small]
	{K_1} && {K_1'} \\
	& {K_2} && {K_2'} \\
	{F_*(M_1)} && {s_{1*}K_1'} \\
	& {F_*(M_2)} && {s_{2*}K_2'} \\
	& {F_*(A)} && B
	\arrow["{s_1}"', from=1-1, to=3-1]
	\arrow["{k_1}", from=1-1, to=1-3]
	\arrow[from=1-3, to=3-3]
	\arrow["{s_{1*}k_1}"{description, pos=0.2}, from=3-1, to=3-3]
	\arrow["{F_*(a_1)}"', from=3-1, to=5-2]
	\arrow["{F_*(a_2)}", from=4-2, to=5-2]
	\arrow["{b_1}"{description, pos=0.3}, from=3-3, to=5-4]
	\arrow["{b_2}", from=4-4, to=5-4]
	\arrow[from=2-4, to=4-4]
	\arrow["{k_2}"{description, pos=0.4}, from=2-2, to=2-4, crossing over]
	\arrow["{s_2}"{description, pos=0.3}, from=2-2, to=4-2, crossing over]
	\arrow["\lrcorner"{anchor=center, pos=0.125, rotate=180}, draw=none, from=3-3, to=1-1]
	\arrow["\lrcorner"{anchor=center, pos=0.125, rotate=180}, draw=none, from=4-4, to=2-2]
	\arrow["f"', from=5-2, to=5-4]
	\arrow["{s_{2*}k_2}"{description}, from=4-2, to=4-4, crossing over]
\end{tikzcd}\]
Then $ (F_*(a_1)s_1, b_1k_{1*}s_1)$, $(F_*(a_2)s_2, b_2k_{2*}s_2)$, are part of the canonical cone of $ f$ in $ \mathcal{B}^2$, which is filtered, so that there exists some common refinement as below
% https://q.uiver.app/?q=WzAsOCxbMCwwLCJLXzEiXSxbMywwLCJLJ18xIl0sWzIsMSwiS18yIl0sWzUsMSwiS18yJyJdLFsxLDIsIksiXSxbNCwyLCJLJyJdLFsxLDMsIkZfKihBKSJdLFs0LDMsIkIiXSxbMCwxLCJrXzEiXSxbMCw0LCJsXzEiLDJdLFsyLDQsImxfMiIsMV0sWzEsNSwibF8xJyIsMCx7ImxhYmVsX3Bvc2l0aW9uIjozMH1dLFszLDUsImxfMiciXSxbMiwzXSxbNCw2LCJ0IiwyXSxbNSw3LCJ0JyJdLFs0LDUsImsiLDFdLFs2LDcsImYiLDJdXQ==
\[\begin{tikzcd}[column sep=small, row sep=small]
	{K_1} &&& {K'_1} \\
	&& {K_2} &&& {K_2'} \\
	& K &&& {K'} \\
	& {F_*(A)} &&& B
	\arrow["{k_1}", from=1-1, to=1-4]
	\arrow["{l_1}"', from=1-1, to=3-2]
	\arrow["{l_2}"{description}, from=2-3, to=3-2]
	\arrow["{l_1'}"{pos=0.3}, from=1-4, to=3-5]
	\arrow["{l_2'}", from=2-6, to=3-5]
	\arrow["k_2"{description, pos=0.3}, from=2-3, to=2-6, crossing over]
	\arrow["t"', from=3-2, to=4-2]
	\arrow["{t'}", from=3-5, to=4-5]
	\arrow["k"{description}, from=3-2, to=3-5]
	\arrow["f"', from=4-2, to=4-5]
\end{tikzcd}\]
But from the filtered colimit decomposition of $ F_*(A) $ above, we can find a lift $s$ of $ t$ through some $ F_*(a) : F_*(M) \rightarrow F_*(A)$ with $ M$ finitely presented in $\mathcal{A}$, and moreover, this $M$ can be chosen as equipped with a common refinement $ u_1 : M_1 \rightarrow M$ of $ a_1$ and $ u_2 : M_2 \rightarrow M$ of $ a_2$ as below
% https://q.uiver.app/?q=WzAsNyxbMCwwLCJLXzEiXSxbMiwwLCJLXzIiXSxbMSwxLCJLIl0sWzEsMiwiRl8qKE0pIl0sWzAsMSwiRl8qKE1fMSkiXSxbMiwxLCJGXyooTV8yKSJdLFsxLDMsIkZfKihBKSJdLFswLDQsInNfMSIsMl0sWzAsMiwibF8xIl0sWzEsMiwibF8yIiwyXSxbMSw1LCJzXzIiXSxbMiwzLCJzIiwxXSxbNCwzLCJGXyoodV8xKSIsMV0sWzUsMywiRl8qKHVfMikiLDFdLFszLDZdLFs0LDYsIkZfKihhXzEpIiwyLHsiY3VydmUiOjN9XSxbNSw2LCJGXyooYV8yKSIsMCx7ImN1cnZlIjotM31dXQ==
\[\begin{tikzcd}
	{K_1} && {K_2} \\
	{F_*(M_1)} & K & {F_*(M_2)} \\
	& {F_*(M)} \\
	& {F_*(A)}
	\arrow["{s_1}"', from=1-1, to=2-1]
	\arrow["{l_1}", from=1-1, to=2-2]
	\arrow["{l_2}"', from=1-3, to=2-2]
	\arrow["{s_2}", from=1-3, to=2-3]
	\arrow["s"{description}, from=2-2, to=3-2]
	\arrow["{F_*(u_1)}"{description}, from=2-1, to=3-2]
	\arrow["{F_*(u_2)}"{description}, from=2-3, to=3-2]
	\arrow[from=3-2, to=4-2]
	\arrow["{F_*(a_1)}"', curve={height=18pt}, from=2-1, to=4-2]
	\arrow["{F_*(a_2)}", curve={height=-18pt}, from=2-3, to=4-2]
\end{tikzcd}\]
But as $ l1'$ and $ l_2'$ also are a common refinement of $ b_1$ and $b_2$, we end with a pair of morphisms $ v_1, v_2$ between the induced pushouts as seen below
% https://q.uiver.app/?q=WzAsMTQsWzAsMSwiS18xIl0sWzIsMSwiS18yIl0sWzEsMiwiSyJdLFsxLDQsIkZfKihNKSJdLFswLDMsIkZfKihNXzEpIl0sWzIsMywiRl8qKE1fMikiXSxbMSw1LCJGXyooQSkiXSxbMywwLCJLXzEnIl0sWzUsMCwiS18yJyJdLFs0LDEsIksnIl0sWzQsMywic18qSyciXSxbMywyLCJzX3sxKn1LXzEnIl0sWzUsMiwic197Mip9S18yJyJdLFs0LDQsIkIiXSxbMCw0LCJzXzEiLDJdLFsxLDUsInNfMiIsMSx7ImxhYmVsX3Bvc2l0aW9uIjozMH1dLFsyLDMsInMiLDFdLFs0LDMsIkZfKih1XzEpIiwxXSxbNSwzLCJGXyoodV8yKSIsMV0sWzMsNl0sWzQsNiwiRl8qKGFfMSkiLDIseyJjdXJ2ZSI6M31dLFsyLDksImsiLDEseyJvZmZzZXQiOi0xfV0sWzcsOSwibF8xJyIsMV0sWzgsOSwibF8yJyIsMV0sWzMsMTAsInNfKmsiLDFdLFs0LDExLCJzX3sxKn1LXzEnIiwxLHsibGFiZWxfcG9zaXRpb24iOjYwLCJvZmZzZXQiOi0yfV0sWzExLDEwLCJ2XzEiLDEseyJsYWJlbF9wb3NpdGlvbiI6MzAsInN0eWxlIjp7ImJvZHkiOnsibmFtZSI6ImRhc2hlZCJ9fX1dLFs3LDExXSxbNSwxMiwic197Mip9S18yJyIsMSx7ImxhYmVsX3Bvc2l0aW9uIjozMCwib2Zmc2V0IjoxfV0sWzgsMTJdLFsxMiwxMCwidl8yIiwxLHsic3R5bGUiOnsiYm9keSI6eyJuYW1lIjoiZGFzaGVkIn19fV0sWzYsMTMsImYiLDJdLFsxMCwxM10sWzAsNywia18xIiwxXSxbOSwxMF0sWzEsOCwia18yIiwxLHsibGFiZWxfcG9zaXRpb24iOjcwfV0sWzUsNiwiRl8qKGFfMikiLDEseyJjdXJ2ZSI6LTN9XSxbMTAsMTEsIiIsMCx7Im9mZnNldCI6LTMsInN0eWxlIjp7Im5hbWUiOiJjb3JuZXIifX1dLFswLDIsImxfMSIsMV0sWzEsMiwibF8yIiwxXSxbMTIsMTMsImJfMiIsMCx7ImN1cnZlIjotMn1dLFsxMSwxMywiYl8xIiwxLHsibGFiZWxfcG9zaXRpb24iOjgwLCJjdXJ2ZSI6Mn1dLFsxMSwzMywiIiwxLHsibGV2ZWwiOjEsInN0eWxlIjp7Im5hbWUiOiJjb3JuZXIifX1dLFsxMiwzNSwiIiwxLHsibGV2ZWwiOjEsInN0eWxlIjp7Im5hbWUiOiJjb3JuZXIifX1dXQ==
\[\begin{tikzcd}[row sep=large]
	&&& {K_1'} && {K_2'} \\
	{K_1} && {K_2} && {K'} \\
	& K && {s_{1*}K_1'} && {s_{2*}K_2'} \\
	{F_*(M_1)} && {F_*(M_2)} && {s_*K'} \\
	& {F_*(M)} &&& B \\
	& {F_*(A)}
	\arrow["{s_1}"', from=2-1, to=4-1]
	\arrow["{b_1}"{description, pos=0.8}, curve={height=12pt}, from=3-4, to=5-5]
	\arrow["s"{description}, from=3-2, to=5-2]
	\arrow["{F_*(u_1)}"{description}, from=4-1, to=5-2]
	\arrow["{F_*(u_2)}"{description}, from=4-3, to=5-2]
	\arrow[from=5-2, to=6-2]
	\arrow["{F_*(a_1)}"', curve={height=18pt}, from=4-1, to=6-2]
	\arrow["{l_1'}"{description, pos=0.2}, from=1-4, to=2-5]
	\arrow["{l_2'}"{description}, from=1-6, to=2-5]
	\arrow["{s_*k}"{description}, from=5-2, to=4-5, crossing over]
	\arrow["{s_{1*}K_1'}"{description, pos=0.5}, shift left=2, from=4-1, to=3-4]
	\arrow["{v_1}"{description, pos=0.3}, dashed, from=3-4, to=4-5]
	\arrow[from=1-4, to=3-4]
	\arrow[from=1-6, to=3-6]
	\arrow["{v_2}"{description}, dashed, from=3-6, to=4-5]
	\arrow["f"', from=6-2, to=5-5]
	\arrow[from=4-5, to=5-5]
	\arrow[""{name=0, anchor=center, inner sep=0}, "{k_1}"{description}, from=2-1, to=1-4]
	\arrow[from=2-5, to=4-5]
	\arrow[""{name=1, anchor=center, inner sep=0}, "{k_2}"{description, pos=0.7}, from=2-3, to=1-6, crossing over]
	\arrow["\lrcorner"{anchor=center, pos=0.125, rotate=180}, shift left=3, draw=none, from=4-5, to=3-4]
	\arrow["{s_{2*}K_2'}"{description, pos=0.2}, shift right=1, from=4-3, to=3-6, crossing over]
	\arrow["k"{description}, shift left=1, from=3-2, to=2-5, crossing over]
	\arrow["{s_2}"{description, pos=0.4}, from=2-3, to=4-3, crossing over]
	\arrow["{l_1}"{description}, from=2-1, to=3-2, crossing over]
	\arrow["{l_2}"{description}, from=2-3, to=3-2, crossing over]
	\arrow["{b_2}", curve={height=-12pt}, from=3-6, to=5-5]
	\arrow["{F_*(a_2)}"{description}, curve={height=-18pt}, from=4-3, to=6-2, crossing over]
	\arrow["\lrcorner"{anchor=center, pos=0.125, rotate=180}, draw=none, from=3-4, to=0]
	\arrow["\lrcorner"{anchor=center, pos=0.125, rotate=180}, draw=none, from=3-6, to=1]
\end{tikzcd}\]
which can be seen as a exhibiting $ s_*k $ as equipped with a common refinement $ (u_1, v_1)$ and $ (u_2, v_2)$ for $s_{1*}K_1 $ and $s_{2*}K_2 $ respectively. Proving that a parallel pair between two such arrows above $ f$ also are equalized by a common refinement involves similar arguments and can be left as an exercise. \\

%and, though they are not necessarily finitely presented, each of the $ F_*(M)$ itself decomposes as a filtered colimit $ $
Now we want to prove that $ f$ is a filtered colimit of all the $ s_*k$ above it. First, observe that $f $ decomposes as a filtered colimit in $\mathcal{B}^2$ $ f \simeq  \colim\, \mathcal{B}^2_\omega \downarrow f$, which in particular induces that $ B \simeq \colim_{\mathcal{B}^2_\omega \downarrow f} \cod(k)$ for $\cod$ is finitary. But remember that $F_*$ also transported the canonical cone of $A$ into a filtering colimit cocone, and as consequence for each $(b,b') :  k \rightarrow f$ in $\mathcal{B}^2$ we have a factorization as below 
% https://q.uiver.app/?q=WzAsNixbMSwwLCJLIl0sWzMsMCwiSyciXSxbMywyLCJCIl0sWzEsMiwiRl8qKEEpIl0sWzAsMSwiRl8qKE0pIl0sWzIsMSwic18qSyciXSxbMywyLCJmIiwyXSxbMCwxLCJrJyJdLFsxLDIsImInIiwwLHsibGFiZWxfcG9zaXRpb24iOjcwfV0sWzAsNCwicyIsMl0sWzEsNV0sWzUsMl0sWzQsMywiRl8qKGEpIiwyXSxbNCw1LCJzXyprIiwxLHsibGFiZWxfcG9zaXRpb24iOjcwfV0sWzAsMywiYiIsMSx7ImxhYmVsX3Bvc2l0aW9uIjo3MH1dLFs1LDAsIiIsMSx7InN0eWxlIjp7Im5hbWUiOiJjb3JuZXIifX1dXQ==
\[\begin{tikzcd}
	& K && {K'} \\
	{F_*(M)} && {s_*K'} \\
	& {F_*(A)} && B
	\arrow["f"', from=3-2, to=3-4]
	\arrow["{k'}", from=1-2, to=1-4]
	\arrow["{b'}"{pos=0.7}, from=1-4, to=3-4]
	\arrow["s"', from=1-2, to=2-1]
	\arrow[from=1-4, to=2-3]
	\arrow["{\langle F_*(a),b \rangle}"{description}, dashed, from=2-3, to=3-4]
	\arrow["{F_*(a)}"', from=2-1, to=3-2]
	\arrow["{s_*k}"{description, pos=0.7}, from=2-1, to=2-3]
	\arrow["b"{description, pos=0.7}, from=1-2, to=3-2, crossing over]
	\arrow["\lrcorner"{anchor=center, pos=0.125, rotate=180}, draw=none, from=2-3, to=1-2]
\end{tikzcd}\]
Moreover, if we denote by $I_f$ the category of all those factorizations $ (k, s,a,b,b')$, which is obviously filtered from the previous items, we have that $ F_*(A) \simeq \colim_{(k, s,a,b,b') \in I_f} F_*(\dom(a)) $. From the previous observation, the following projection is essentially surjective
\[ I_f \twoheadrightarrow \mathcal{B}_\omega^f\downarrow f \]
Moreover, in each $ a : M \rightarrow A$ we can define the category $ I_{f,a}$ consisting of all the quintuplets of the form $ (k,s,a,b,b')$ with $a$ fixed; we have an inclusion $ I_{f,a} \hookrightarrow I_f$, and a filtered colimit 
\[ F_*(M) = \underset{(k,s,a,b,b') \in I_{f,a}}{\colim} \; K \]
But now from \cref{canonical cone in coslice} we know that actually computing the colimit of all $ k$ over $I_{f,a}$ is the same as computing the colimit of the corresponding pushouts $ s_*k$, that is we have an equality of filtered colimits in $ F_*(M)\downarrow \mathcal{B}$
\[  \underset{(k,s,a,b,b') \in I_{f,a}}{\colim} \; k \simeq \underset{(k,s,a,b,b') \in I_{f,a}}{\colim} \; s_*k \]
and also, for the codomain functor preserves filtered colimits, we have a filtered colimit in $\mathcal{B}$
\[  \underset{(k,s,a,b,b') \in I_{f,a}}{\colim} \; K' \simeq \underset{(k,s,a,b,b') \in I_{f,a}}{\colim} \; s_*K' \]
But from the sequence of colimit decomposition
\[ F_*(A) \simeq \underset{\mathcal{A}_\omega \downarrow A}{\colim} \, F_*(M) \simeq  \underset{\mathcal{A}_\omega \downarrow A}{\colim} \underset{(k,s,a,b,b') \in I_{f,a}}{\colim} \; k  \]
we can use the isomorphism in each $ a$ to get an isomorphism between the following colimits:
\[ f \simeq  \underset{\mathcal{A}_\omega \downarrow A}{\colim} \underset{(k,s,a,b,b') \in I_{f,a}}{\colim} \; k \simeq  \underset{\mathcal{A}_\omega \downarrow A}{\colim} \underset{(k,s,a,b,b') \in I_{f,a}}{\colim} \; s_*k \]
This achieves to prove that the arrows of the form $ s_*k$ are a generator of finitely presented object in $F_*\downarrow \mathcal{B}$. For $\mathcal{A}_\omega$ and $ \mathcal{B}_\omega^2$ both are essentially small while $ \mathcal{B}$ is locally small, this category itself is essentially small. This achieves to prove that $ F_*\downarrow \mathcal{B}$ is finitely accessible. And from what we said about existence of small limits, it is moreover finitely accessible.
\end{proof}

From
Finally, we would like to control explicitly the generator of finitely presented objects of $ F_*\downarrow \mathcal{B}$: we claim that not only pushouts maps of $\mathcal{G}_{F_*}$ are a generator of etale objects, but that any finitely presented object is actually of this form: 

\begin{lemma}
$\mathcal{G}_{F_*}$ is closed under retracts. Hence any finitely presented object of $F_*\downarrow \mathcal{B}$ is in $\mathcal{G}_{F_*}$. 
\end{lemma}

\begin{proof}
Suppose that $ f :F_*(A) \rightarrow B$ is finitely presented. Then from the expression of $ f $ as the filtered colimit $ f \simeq \colim \;{ \mathcal{G}_{F_*}\downarrow f} $ established in the previous item, we can exhibit $f$ as a retract of some arrow in $\mathcal{G}_{F_*}$
% https://q.uiver.app/?q=WzAsOCxbMCwyLCJGXyooQSkiXSxbMiwzLCJGXyooQSkiXSxbMSwxLCJGXyooTSkiXSxbMSwwLCJLIl0sWzIsMiwiQiJdLFs0LDMsIkIiXSxbMywxLCJzXypLIl0sWzMsMCwiSyciXSxbMCwyLCJGXyoocykiXSxbMCwxLCIiLDIseyJsZXZlbCI6Miwic3R5bGUiOnsiaGVhZCI6eyJuYW1lIjoibm9uZSJ9fX1dLFszLDIsInMiLDJdLFswLDQsImYiLDFdLFs0LDUsIiIsMix7ImxldmVsIjoyLCJzdHlsZSI6eyJoZWFkIjp7Im5hbWUiOiJub25lIn19fV0sWzQsNiwicyciLDFdLFs2LDUsInInIiwxXSxbMiw2LCJzXyprIiwxXSxbMyw3LCJrIl0sWzcsNl0sWzEsNSwiZiIsMV0sWzIsMSwiRl8qKHIpIiwxLHsibGFiZWxfcG9zaXRpb24iOjcwfV0sWzYsMywiIiwxLHsic3R5bGUiOnsibmFtZSI6ImNvcm5lciJ9fV1d
\[\begin{tikzcd}
	& K && {K'} \\
	& {F_*(M)} && {s_*K'} \\
	{F_*(A)} && B \\
	&& {F_*(A)} && B
	\arrow["{F_*(s)}", from=3-1, to=2-2]
	\arrow[Rightarrow, no head, from=3-1, to=4-3]
	\arrow["s"', from=1-2, to=2-2]
	\arrow["f"{description}, from=3-1, to=3-3]
	\arrow[Rightarrow, no head, from=3-3, to=4-5]
	\arrow["{s'}"{description}, from=3-3, to=2-4]
	\arrow["{r'}"{description}, from=2-4, to=4-5]
	\arrow["{s_*k}"{description}, from=2-2, to=2-4]
	\arrow["k", from=1-2, to=1-4]
	\arrow[from=1-4, to=2-4]
	\arrow["f"{description}, from=4-3, to=4-5]
	\arrow["{F_*(r)}"{description, pos=0.7}, from=2-2, to=4-3, crossing over]
	\arrow["\lrcorner"{anchor=center, pos=0.125, rotate=180}, draw=none, from=2-4, to=1-2]
\end{tikzcd}\]
But then from $\mathcal{A}_\omega$ is closed under retracts, $A$ must be in $\mathcal{A}_\omega$. Moreover, this retraction can be transferred into a retract of a finitely presented object in the coslice $ F_*(A)\downarrow \mathcal{B}$
% https://q.uiver.app/?q=WzAsOSxbMCwzXSxbMSwyLCJGXyooQSkiXSxbMSwxLCJGXyooTSkiXSxbMiwzLCJCIl0sWzQsNCwiQiJdLFszLDIsIihGXyoocilzKV8qSyciXSxbMywxLCJzXypLJyJdLFsxLDAsIksiXSxbMywwLCJLJyJdLFszLDQsIiIsMix7ImxldmVsIjoyLCJzdHlsZSI6eyJoZWFkIjp7Im5hbWUiOiJub25lIn19fV0sWzMsNSwic18qa18qRl8qKHIpcyciLDFdLFs1LDQsIlxcbGFuZ2xlIGYsIHInIFxccmFuZ2xlIiwxXSxbMSw1XSxbNiw1LCJzXyprXypGXyoocikiXSxbNyw4LCJrIl0sWzIsNiwic18qayIsMV0sWzcsMiwicyIsMl0sWzgsNl0sWzIsMSwiRl8qKHIpIiwyXSxbMSwzLCJmIiwxXSxbMSw0LCJmIiwxXSxbNiwxNCwiIiwyLHsibGV2ZWwiOjEsInN0eWxlIjp7Im5hbWUiOiJjb3JuZXIifX1dLFs1LDE1LCIiLDEseyJsZXZlbCI6MSwic3R5bGUiOnsibmFtZSI6ImNvcm5lciJ9fV1d
\[\begin{tikzcd}
	& K && {K'} \\
	& {F_*(M)} && {s_*K'} \\
	& {F_*(A)} && {(F_*(r)s)_*K'} \\
	{} && B \\
	&&&& B
	\arrow[Rightarrow, no head, from=4-3, to=5-5]
	\arrow["{s_*k_*F_*(r)s'}"{description, pos=0.6}, from=4-3, to=3-4]
	\arrow["{\langle f, r' \rangle}"{description}, from=3-4, to=5-5]
	\arrow[from=3-2, to=3-4]
	\arrow["{s_*k_*F_*(r)}", from=2-4, to=3-4]
	\arrow[""{name=0, anchor=center, inner sep=0}, "k", from=1-2, to=1-4]
	\arrow[""{name=1, anchor=center, inner sep=0}, "{s_*k}"{description}, from=2-2, to=2-4]
	\arrow["s"', from=1-2, to=2-2]
	\arrow[from=1-4, to=2-4]
	\arrow["{F_*(r)}"', from=2-2, to=3-2]
	\arrow["f"{description}, from=3-2, to=4-3]
	\arrow["f"{description}, from=3-2, to=5-5, crossing over]
	\arrow["\lrcorner"{anchor=center, pos=0.125, rotate=180}, draw=none, from=2-4, to=0]
	\arrow["\lrcorner"{anchor=center, pos=0.125, rotate=180}, draw=none, from=3-4, to=1]
\end{tikzcd}\]
But from \cref{closed under retract}, we know this forces $ f =  \langle f, r'\rangle (F_*(r)s)_*k$ to be obtained as some pushout of a finitely presented map: hence $ f$ is actually in $\mathcal{G}_{F_*}$. 
\end{proof}

Putting the previous lemma altogether we get the main result of this section:

\begin{theorem}
For a morphism of locally finitely presentable categories $F : \mathcal{A} \rightarrow \mathcal{B}$, the comma category $ F_*\downarrow \mathcal{B} $ is locally finitely presentable as well, and moreover we have an equivalence
\[  (F_* \downarrow \mathcal{B})_\omega \simeq \mathcal{G}_{F_*}  \]
\end{theorem}

\begin{proof}
From what precedes we know that $ F_*\downarrow \mathcal{B}$ is finitely accessible and $ \mathcal{G}_{F_*}$ is exactly its generator of finitely presented objects. Moreover, from what we saw at the begining of the section, it has filtered colimits and small limits both computed in $ \mathcal{B}^2$. Therefore it is locally finitely presentable.\\
\end{proof}

\begin{remark}
As a locally finitely presentable category, $ F \downarrow \mathcal{B}$ also has arbitrary colimits. However they are not computed in the arrow categories, contrarily to filtered ones. 
\end{remark}

\begin{proposition}
The codomain functor $ \cod : F_*\downarrow \mathcal{B} \rightarrow \mathcal{B}$ has a left adjoint, which moreover sends finitely presented objects on finitely presented objects, and the pair $ \cod^*\dashv \cod$ defines a morphism of locally finitely presentable categories. 
\end{proposition}

\begin{proof}
Observe first that $\cod$ factorizes through the inclusion into the arrow category
% https://q.uiver.app/?q=WzAsMyxbMCwwLCJGXypcXGRvd25hcnJvd1xcbWF0aGNhbHtCfSJdLFsyLDAsIlxcbWF0aGNhbHtCfSJdLFsxLDEsIlxcbWF0aGNhbHtCfV4yIl0sWzAsMSwiXFxjb2QiXSxbMCwyLCJcXGlvdGFfRiIsMl0sWzIsMSwiXFxjb2QiLDJdLFszLDIsIlxcc2ltZXEiLDEseyJzaG9ydGVuIjp7InNvdXJjZSI6MjB9LCJzdHlsZSI6eyJib2R5Ijp7Im5hbWUiOiJub25lIn0sImhlYWQiOnsibmFtZSI6Im5vbmUifX19XV0=
\[\begin{tikzcd}
	{F_*\downarrow\mathcal{B}} && {\mathcal{B}} \\
	& {\mathcal{B}^2}
	\arrow[""{name=0, anchor=center, inner sep=0}, "\cod", from=1-1, to=1-3]
	\arrow["{\iota_F}"', from=1-1, to=2-2]
	\arrow["\cod"', from=2-2, to=1-3]
	\arrow["\simeq"{description}, Rightarrow, draw=none, from=0, to=2-2]
\end{tikzcd}\]
Then $ \cod : \mathcal{B}^2 \rightarrow \mathcal{B}$ has as left adjoint the functor $ !_{(-)}$ sending $ B$ on the initial map $ !_B : 0 \rightarrow B$, and any $ f: B_1 \rightarrow B_2$ on the induced square
% https://q.uiver.app/?q=WzAsNCxbMCwwLCIwIl0sWzAsMSwiMCJdLFsxLDAsIkJfMSJdLFsxLDEsIkJfMiJdLFswLDEsIiIsMix7ImxldmVsIjoyLCJzdHlsZSI6eyJoZWFkIjp7Im5hbWUiOiJub25lIn19fV0sWzAsMiwiIV97Ql8xfSJdLFsyLDMsImYiXSxbMSwzLCIhX3tCXzJ9IiwyXV0=
\[\begin{tikzcd}
	0 & {B_1} \\
	0 & {B_2}
	\arrow[Rightarrow, no head, from=1-1, to=2-1]
	\arrow["{!_{B_1}}", from=1-1, to=1-2]
	\arrow["f", from=1-2, to=2-2]
	\arrow["{!_{B_2}}"', from=2-1, to=2-2]
\end{tikzcd}\]
Indeed, any square $ !_B \rightarrow f$ for $ f : \dom (f) \rightarrow \cod(f)$ as below
% https://q.uiver.app/?q=WzAsNCxbMCwwLCIwIl0sWzAsMSwiQiJdLFsxLDAsIkJfMSJdLFsxLDEsIkJfMiJdLFswLDEsIiFfQiIsMl0sWzAsMiwiIV97Ql8xfSJdLFsyLDMsImYiXSxbMSwzLCJnIiwyXV0=
\[\begin{tikzcd}
	0 & {\dom (f)} \\
	B & {\cod(f)}
	\arrow["{!_B}"', from=1-1, to=2-1]
	\arrow["{!_{B_1}}", from=1-1, to=1-2]
	\arrow["f", from=1-2, to=2-2]
	\arrow["g"', from=2-1, to=2-2]
\end{tikzcd}\]
is uniquely defined by a choice of $ g : B \rightarrow \cod(f)$, which proves that $ !_{(-)} \dashv \cod $. \\ 

Now, we construct a left adjoint $ \iota_F^*$ of $ \iota_F: F_*\downarrow \mathcal{B} \hookrightarrow \mathcal{B}^2$ as follows: for any arrow $ f : B_1 \rightarrow B_2$ takes the pushout along the unit of $ B_1$
% https://q.uiver.app/?q=WzAsNCxbMCwwLCJCXzEiXSxbMCwxLCJGXypGXiooQl8xKSJdLFsxLDAsIkJfMiJdLFsxLDEsIlxcZXRhX3tCXzEqfUJfMiJdLFswLDEsIlxcZXRhX3tCXzF9IiwyXSxbMCwyLCJmIl0sWzIsM10sWzEsMywiXFxldGFfe0JfMSp9ZiIsMl0sWzMsMCwiIiwxLHsic3R5bGUiOnsibmFtZSI6ImNvcm5lciJ9fV1d
\[\begin{tikzcd}
	{B_1} & {B_2} \\
	{F_*F^*(B_1)} & {\eta_{B_1*}B_2}
	\arrow["{\eta_{B_1}}"', from=1-1, to=2-1]
	\arrow["f", from=1-1, to=1-2]
	\arrow[from=1-2, to=2-2]
	\arrow["{\eta_{B_1*}f}"', from=2-1, to=2-2]
	\arrow["\lrcorner"{anchor=center, pos=0.125, rotate=180}, draw=none, from=2-2, to=1-1]
\end{tikzcd}\]
and for a morphism of arrows $ (u,v) : f \rightarrow f'$, take as $ \iota^*_F(u,v)$ the vertical arrow of the front square in the diagram below
% https://q.uiver.app/?q=WzAsOCxbMCwwLCJCXzEiXSxbMSwxLCJGXypGXiooQl8xKSJdLFsyLDAsIkJfMiJdLFszLDEsIlxcZXRhX3tCXzEqfUJfMiJdLFswLDIsIkJfMSciXSxbMiwyLCJCXzInIl0sWzEsMywiRl8qRl4qKEInXzEpIl0sWzMsMywiXFxldGFfe0InXzEqfUInXzIiXSxbMCwxLCJcXGV0YV97Ql8xfSIsMV0sWzAsMiwiZiJdLFsyLDNdLFsxLDMsIlxcZXRhX3tCXzEqfWYiLDEseyJsYWJlbF9wb3NpdGlvbiI6MzB9XSxbMCw0LCJ1IiwyXSxbMiw1LCJ2IiwxLHsibGFiZWxfcG9zaXRpb24iOjcwfV0sWzQsNSwiZiciLDIseyJsYWJlbF9wb3NpdGlvbiI6NzB9XSxbNiw3LCJcXGV0YV97QidfMSp9ZiciLDJdLFs0LDYsIlxcZXRhX3tCJ18xfSIsMl0sWzUsN10sWzEsNiwiRl8qRl4qKHUpIiwxLHsibGFiZWxfcG9zaXRpb24iOjMwfV0sWzMsNywiXFxsYW5nbGUgIFxcZXRhX3tCJ18xKn1mJ0ZfKkZeKih1KSwgZidfKlxcZXRhX3tCJ18xfXYgXFxyYW5nbGUiLDAseyJzdHlsZSI6eyJib2R5Ijp7Im5hbWUiOiJkYXNoZWQifX19XSxbMyw5LCIiLDEseyJsZXZlbCI6MSwic3R5bGUiOnsibmFtZSI6ImNvcm5lciJ9fV1d
\[\begin{tikzcd}
	{B_1} && {B_2} \\
	& {F_*F^*(B_1)} && {\eta_{B_1*}B_2} \\
	{B_1'} && {B_2'} \\
	& {F_*F^*(B'_1)} && {\eta_{B'_1*}B'_2}
	\arrow["{\eta_{B_1}}"{description}, from=1-1, to=2-2]
	\arrow[""{name=0, anchor=center, inner sep=0}, "f", from=1-1, to=1-3]
	\arrow[from=1-3, to=2-4]
	\arrow["u"', from=1-1, to=3-1]
	\arrow["v"{description, pos=0.7}, from=1-3, to=3-3]
	\arrow["{\eta_{B_1*}f}"{description, pos=0.3}, from=2-2, to=2-4, crossing over]
	\arrow["{f'}"'{pos=0.7}, from=3-1, to=3-3]
	\arrow["{\eta_{B'_1*}f'}"', from=4-2, to=4-4]
	\arrow["{\eta_{B'_1}}"', from=3-1, to=4-2, crossing over]
	\arrow[from=3-3, to=4-4]
	\arrow["{F_*F^*(u)}"{description, pos=0.3}, from=2-2, to=4-2, crossing over]
	\arrow["{\langle  \eta_{B'_1*}f'F_*F^*(u), f'_*\eta_{B'_1}v \rangle}", dashed, from=2-4, to=4-4]
	\arrow["\lrcorner"{anchor=center, pos=0.125, rotate=180}, draw=none, from=2-4, to=0]
\end{tikzcd}\]
Then for any morphism of arrow in $\mathcal{B}^2$ as below
% https://q.uiver.app/?q=WzAsNCxbMCwwLCJCXzEiXSxbMSwwLCJCXzIiXSxbMCwxLCJGXyooQSkiXSxbMSwxLCJCIl0sWzAsMSwiZiJdLFsyLDMsImciLDJdLFswLDIsInUiLDJdLFsxLDMsInYiXV0=
\[\begin{tikzcd}
	{B_1} & {B_2} \\
	{F_*(A)} & B
	\arrow["f", from=1-1, to=1-2]
	\arrow["g"', from=2-1, to=2-2]
	\arrow["u"', from=1-1, to=2-1]
	\arrow["v", from=1-2, to=2-2]
\end{tikzcd}\]
the adjunction $ F^*\dashv F_*$ returns a unique morphism $ u : F^*(B_1) \rightarrow A$ in $\mathcal{A}$ such that $u$ factorizes uniquely as $ F_*(\overline{ u}) \eta_{B_1}$, and then this induces uniquely a morphism of arrows as seen below in the front square
% https://q.uiver.app/?q=WzAsNixbMCwwLCJCXzEiXSxbMSwxLCJGXypGXiooQl8xKSJdLFsyLDAsIkJfMiJdLFszLDEsIlxcZXRhX3tCXzEqfUJfMiJdLFsxLDMsIkZfKihBKSJdLFszLDMsIkIiXSxbMCwxLCJcXGV0YV97Ql8xfSIsMV0sWzAsMiwiZiJdLFsyLDNdLFsxLDMsIlxcZXRhX3tCXzEqfWYiLDEseyJsYWJlbF9wb3NpdGlvbiI6MzB9XSxbNCw1LCJnIiwyXSxbMSw0LCJcXG92ZXJsaW5le3V9IiwxXSxbMyw1LCJcXGxhbmdsZSAgZ1xcb3ZlcmxpbmV7dX0gLCB2IFxccmFuZ2xlIiwwLHsic3R5bGUiOnsiYm9keSI6eyJuYW1lIjoiZGFzaGVkIn19fV0sWzAsNCwidSIsMix7ImN1cnZlIjozfV0sWzIsNSwidiIsMSx7ImN1cnZlIjoyfV0sWzMsNywiIiwxLHsibGV2ZWwiOjEsInN0eWxlIjp7Im5hbWUiOiJjb3JuZXIifX1dXQ==
\[\begin{tikzcd}
	{B_1} && {B_2} \\
	& {F_*F^*(B_1)} && {\eta_{B_1*}B_2} \\
	\\
	& {F_*(A)} && B
	\arrow["{\eta_{B_1}}"{description}, from=1-1, to=2-2]
	\arrow[""{name=0, anchor=center, inner sep=0}, "f", from=1-1, to=1-3]
	\arrow[from=1-3, to=2-4]
	\arrow["g"', from=4-2, to=4-4]
	\arrow["{\overline{u}}"{description}, from=2-2, to=4-2]
	\arrow["{\langle  g\overline{u} , v \rangle}", dashed, from=2-4, to=4-4]
	\arrow["u"', curve={height=18pt}, from=1-1, to=4-2]
	\arrow["v"{description}, curve={height=12pt}, from=1-3, to=4-4]
	\arrow["\lrcorner"{anchor=center, pos=0.125, rotate=180}, draw=none, from=2-4, to=0]	
	\arrow["{\eta_{B_1*}f}"{description, pos=0.3}, from=2-2, to=2-4, crossing over]
\end{tikzcd}\]
Now we want to compose those functors to get a left adjoint $ \cod^*: \mathcal{B} \rightarrow \mathcal{B}^2$. First observe that $ F^*(0) \simeq 0$ as $ F^*$ preserves colimits; then the unit $ \eta_0 : 0 \rightarrow F_*F^*(0)$ coincides with the initial map $ !_{F_*(0)} : 0 \rightarrow F_*(0)$. Now take a $B$ in $\mathcal{B}$ to the lower arrow $\eta_0*!_B$ in the following pushout
% https://q.uiver.app/?q=WzAsNCxbMCwwLCIwIl0sWzEsMCwiQiJdLFswLDEsIkZfKigwKSJdLFsxLDEsIkZfKigwKStCIl0sWzAsMSwiIV9CIl0sWzAsMiwiXFxldGFfMCIsMl0sWzIsMywiXFxldGFfMCohX0IiLDJdLFsxLDNdLFszLDAsIiIsMSx7InN0eWxlIjp7Im5hbWUiOiJjb3JuZXIifX1dXQ==
\[\begin{tikzcd}
	0 & B \\
	{F_*(0)} & {F_*(0)+B}
	\arrow["{!_B}", from=1-1, to=1-2]
	\arrow["{\eta_0}"', from=1-1, to=2-1]
	\arrow["{\eta_0*!_B}"', from=2-1, to=2-2]
	\arrow[from=1-2, to=2-2]
	\arrow["\lrcorner"{anchor=center, pos=0.125, rotate=180}, draw=none, from=2-2, to=1-1]
\end{tikzcd}\]
This defines a left adjoint to $\cod$: indeed, though $ F_*$ does not preserve initialness, initialness is remembered in some sense in the comma any morphsim $ \eta_{B*}f \rightarrow g$ with $ g : F_*(A) \rightarrow B'$ has is left component forced to be $ F*(!_A)$, so that in the diagram below
% https://q.uiver.app/?q=WzAsNixbMCwxLCJGXyooMCkiXSxbMiwxLCJGXyooMCkrQiJdLFsyLDAsIkIiXSxbMCwwLCIwIl0sWzEsMiwiRl8qKEEpIl0sWzMsMiwiQiciXSxbMCwxLCJcXGV0YV8wKiFfQiIsMV0sWzIsMV0sWzMsMiwiIV9CIl0sWzMsMCwiXFxldGFfMCIsMl0sWzAsNCwiRl4qKCFfQSkiLDJdLFsxLDUsInYiXSxbNCw1LCJnIiwyXSxbMSw4LCIiLDEseyJsZXZlbCI6MSwic3R5bGUiOnsibmFtZSI6ImNvcm5lciJ9fV1d
\[\begin{tikzcd}
	0 && B \\
	{F_*(0)} && {F_*(0)+B} \\
	& {F_*(A)} && {B'}
	\arrow["{\eta_0*!_B}"{description}, from=2-1, to=2-3]
	\arrow[from=1-3, to=2-3]
	\arrow[""{name=0, anchor=center, inner sep=0}, "{!_B}", from=1-1, to=1-3]
	\arrow["{\eta_0}"', from=1-1, to=2-1]
	\arrow["{F^*(!_A)}"', from=2-1, to=3-2]
	\arrow["v", from=2-3, to=3-4]
	\arrow["g"', from=3-2, to=3-4]
	\arrow["\lrcorner"{anchor=center, pos=0.125, rotate=180}, draw=none, from=2-3, to=0]
\end{tikzcd}\]
we see that such an arrow is uniquely determined by a map $ v!_{B*}\eta_0 $, which itself was in return uniquely determined by $v$. Hence the adjunction. \\

Now, though it is also a consequence of $\cod $ being finitary, it is easy to see directly that this left adjoint preserves finite presentedness with our notion of finitely presented objects in $F_*\downarrow \mathcal{B}$. Indeed if $ K$ is in $\mathcal{B}_\omega$, then not only is the initial map $ !_K : 0 \rightarrow K$ in $\mathcal{B}_\omega$ as 0 always is finitely presented, but as $ F^*(0) = 0 $ is in $\mathcal{A}_\omega$, then $ \cod^*(B) =  \eta_{0*}!_B$ is finitely presented in $F_*\downarrow \mathcal{B}$ from what was proved above. 
\end{proof}

Now we turn to the other part of the factorization:

\begin{proposition}
The functor $ 1_{F_*} : \mathcal{A} \rightarrow F_*\downarrow \mathcal{B}$ has a left adjoint $ 1^*_{F_*}$, which moreover sends finitely presented objects on finitely presented objects, and the pair $ 
1^*_{F_*} \dashv 1_{F_*}$ defines a morphism of locally finitely presentable categories. 
\end{proposition}

\begin{proof}
We saw that $  1_{F_*}$ is continuous and accessible by the very computation of limits in the comma. Applying the adjoint functor theorem would ensure the existence of a left adjoint: however we want an explicit description of this left adjoint. From adjoint functor theorem, we can compute its value at a given object as a limit as follows. \\

%For $ f : F_*(A) \rightarrow B$, the comma $ f \downarrow 1_{F_*}$ consisting of all the square of the form 
% https://q.uiver.app/?q=WzAsNCxbMCwwLCJGXyooQSkiXSxbMSwwLCJCIl0sWzAsMSwiRl8qKE0pIl0sWzEsMSwiRl8qKE0pIl0sWzAsMiwiRl8qKGEpIiwyXSxbMCwxLCJmIl0sWzEsMywiYiJdLFsyLDMsIiIsMCx7ImxldmVsIjoyLCJzdHlsZSI6eyJoZWFkIjp7Im5hbWUiOiJub25lIn19fV1d
%with $ M$ in $\mathcal{A}_\omega$ is small, so we can compute in $\mathcal{A}$ the limit 
%while on a square $ (u,v) : f_1 \rightarrow f_2$, precomposition with $(u,v)$ defines a functor $ J_{f_2} \rightarrow J_{f_1}$ so that $ \lim_{f_1 \downarrow 1_{F_*}} M$ is in particular equiped with a cone over $ J_2$ which induces a factorization by the universal property of the limit, ensuring functoriality. In particular the unit  

Being finitary, $ 1_{F_*}$ is accessible, hence satisfies the solution set condition: thus for each $f : F_*(A) \rightarrow B$, the comma category $ f \downarrow 1_{F_*}$ has a small weakly initial family $ J_f$, whose elements will be denoted as % https://q.uiver.app/?q=WzAsNCxbMCwwLCJGXyooQSkiXSxbMSwwLCJCIl0sWzAsMSwiRl8qKEFfaSkiXSxbMSwxLCJGXyooQV9pKSJdLFsyLDMsIiIsMCx7ImxldmVsIjoyLCJzdHlsZSI6eyJoZWFkIjp7Im5hbWUiOiJub25lIn19fV0sWzAsMiwiRl8qKGFfaSkiLDJdLFswLDEsImYiXSxbMSwzLCJiX2kiXV0=
\[\begin{tikzcd}
	{F_*(A)} & B \\
	{F_*(A_i)} & {F_*(A_i)}
	\arrow[Rightarrow, no head, from=2-1, to=2-2]
	\arrow["{F_*(a_i)}"', from=1-1, to=2-1]
	\arrow["f", from=1-1, to=1-2]
	\arrow["{b_i}", from=1-2, to=2-2]
\end{tikzcd}\]
for each $ i \in J_f$. Then define the value of $ 1^*_{F_*}$ as the limit
\[  1^*_{F_*}(f) = \underset{j \in J_f}{\lim}\; A_i \]
Then it is standard calculation to see that this defines a left adjoint to $ 1_{F_*}$. Moreover, the fact it restricts to finitely presented objects is a consequence of $ 1_{F_*}$ being finitary.   \\
\end{proof}

Uniqueness of adjoints ensures that the composite of those left adjoints coincides up to invertible 2-cell with the left adjoint $ F^*$ as depicted in the factorization below:
% https://q.uiver.app/?q=WzAsMyxbMCwwLCJcXG1hdGhjYWx7QX1fXFxvbWVnYSJdLFsyLDAsIlxcbWF0aGNhbHtCfV9cXG9tZWdhIl0sWzEsMSwiKEZfKlxcZG93bmFycm93IFxcbWF0aGNhbHtCfSlfXFxvbWVnYSJdLFsxLDAsIkZeKl9cXG9tZWdhIiwyXSxbMSwyLCJcXGNvZF4qX1xcb21lZ2EiXSxbMiwwLCIoMV97Rl8qfSlfXFxvbWVnYV4qIl0sWzMsMiwiXFxzaW1lcSIsMSx7InNob3J0ZW4iOnsic291cmNlIjoyMH0sInN0eWxlIjp7ImJvZHkiOnsibmFtZSI6Im5vbmUifSwiaGVhZCI6eyJuYW1lIjoibm9uZSJ9fX1dXQ==
\[\begin{tikzcd}
	{\mathcal{A}_\omega} && {\mathcal{B}_\omega} \\
	& {(F_*\downarrow \mathcal{B})_\omega}
	\arrow[""{name=0, anchor=center, inner sep=0}, "{F^*_\omega}"', from=1-3, to=1-1]
	\arrow["{\cod^*_\omega}", from=1-3, to=2-2]
	\arrow["{(1_{F_*})_\omega^*}", from=2-2, to=1-1]
	\arrow["\simeq"{description}, Rightarrow, draw=none, from=0, to=2-2]
\end{tikzcd}\]

However we would like to understand explicitely what makes $1^*_{F_*}$ to return the finitely presented object $F^*(K)$ when applied to $ \cod^*(K) $ for a finite presented $K$, the precedent proof failing to provide a satisfactory description of the process. \\

Recall that we constructed $ \cod^*(K)$ as $ \eta_{0*}!_K: F_*(0) \rightarrow F_*(0) + K$. Then observe that the following square
% https://q.uiver.app/?q=WzAsNSxbMCwwLCIwIl0sWzAsMSwiRl8qKDApIl0sWzAsMiwiRl8qRl4qKEspIl0sWzEsMCwiSyJdLFsxLDIsIkZfKkZeKihLKSJdLFsxLDIsIkZfKighX0spIiwyXSxbMCwxLCJcXGV0YV8wIiwyXSxbMCwzLCIhX0siXSxbMiw0LCIiLDIseyJsZXZlbCI6Miwic3R5bGUiOnsiaGVhZCI6eyJuYW1lIjoibm9uZSJ9fX1dLFszLDQsIlxcZXRhX0siXV0=
\[\begin{tikzcd}
	0 & K \\
	{F_*(0)} \\
	{F_*F^*(K)} & {F_*F^*(K)}
	\arrow["{F_*(!_K)}"', from=2-1, to=3-1]
	\arrow["{\eta_0}"', from=1-1, to=2-1]
	\arrow["{!_K}", from=1-1, to=1-2]
	\arrow[Rightarrow, no head, from=3-1, to=3-2]
	\arrow["{\eta_K}", from=1-2, to=3-2]
\end{tikzcd}\]
induces a unique factorization 
% https://q.uiver.app/?q=WzAsNixbMCwwLCIwIl0sWzAsMSwiRl8qKDApIl0sWzAsMiwiRl8qRl4qKEspIl0sWzEsMCwiSyJdLFsyLDIsIkZfKkZeKihLKSJdLFsxLDEsIkZfKigwKStLIl0sWzEsMiwiRl8qKCFfSykiLDJdLFswLDEsIlxcZXRhXzAiLDJdLFswLDMsIiFfSyJdLFsyLDQsIiIsMix7ImxldmVsIjoyLCJzdHlsZSI6eyJoZWFkIjp7Im5hbWUiOiJub25lIn19fV0sWzMsNCwiXFxldGFfSyIsMCx7ImN1cnZlIjotMn1dLFszLDVdLFsxLDVdLFs1LDAsIiIsMix7InN0eWxlIjp7Im5hbWUiOiJjb3JuZXIifX1dLFs1LDQsIlxcbGFuZ2xlIEZfKighX0spLCBcXGV0YV9LIFxccmFuZ2xlIiwxLHsic3R5bGUiOnsiYm9keSI6eyJuYW1lIjoiZGFzaGVkIn19fV1d
\[\begin{tikzcd}
	0 & K \\
	{F_*(0)} & {F_*(0)+K} \\
	{F_*F^*(K)} && {F_*F^*(K)}
	\arrow["{F_*(!_K)}"', from=2-1, to=3-1]
	\arrow["{\eta_0}"', from=1-1, to=2-1]
	\arrow["{!_K}", from=1-1, to=1-2]
	\arrow[Rightarrow, no head, from=3-1, to=3-3]
	\arrow["{\eta_K}", curve={height=-12pt}, from=1-2, to=3-3]
	\arrow[from=1-2, to=2-2]
	\arrow[from=2-1, to=2-2]
	\arrow["\lrcorner"{anchor=center, pos=0.125, rotate=180}, draw=none, from=2-2, to=1-1]
	\arrow["{\langle F_*(!_K), \eta_K \rangle}"{description}, dashed, from=2-2, to=3-3]
\end{tikzcd}\]
Then it is clear that the following square 
% https://q.uiver.app/?q=WzAsNCxbMCwwLCJGXyooMCkiXSxbMCwxLCJGXypGXiooSykiXSxbMSwxLCJGXypGXiooSykiXSxbMSwwLCJGXyooMCkrSyJdLFswLDEsIkZfKighX0spIiwyXSxbMSwyLCIiLDIseyJsZXZlbCI6Miwic3R5bGUiOnsiaGVhZCI6eyJuYW1lIjoibm9uZSJ9fX1dLFswLDMsIlxcY29kXiooSykiXSxbMywyLCJcXGxhbmdsZSBGXyooIV9LKSwgXFxldGFfSyBcXHJhbmdsZSIsMCx7InN0eWxlIjp7ImJvZHkiOnsibmFtZSI6ImRhc2hlZCJ9fX1dXQ==
\[\begin{tikzcd}
	{F_*(0)} & {F_*(0)+K} \\
	{F_*F^*(K)} & {F_*F^*(K)}
	\arrow["{F_*(!_K)}"', from=1-1, to=2-1]
	\arrow[Rightarrow, no head, from=2-1, to=2-2]
	\arrow["{\cod^*(K)}", from=1-1, to=1-2]
	\arrow["{\langle F_*(!_K), \eta_K \rangle}", dashed, from=1-2, to=2-2]
\end{tikzcd}\]
is an initial object of the comma $ \cod^*(K)\downarrow 1_{F_*}$, and in particular factorizes all the members of the solution set of $1_{F_*} $ at $ \cod^*(F)$: hence the limit above reduces here on 
\[ 1_{F_*}^*(\cod^*(K)) = \underset{ \cod^*(K)\downarrow 1_{F_*}}{\lim} \; A \simeq F^*(K) \]

Observe also that the pseudosection $ \pi_1$ of $ 1_{F_*}$ is part of a morphism of locally finitely presentable categories: in fact, it is immediate to see that it is \emph{right adjoint} to $ \pi_1$, and that its both finitary and continuous. To sum up, the following square 
% https://q.uiver.app/?q=WzAsNCxbMSwwLCJcXG1hdGhjYWx7Qn0iXSxbMSwxLCJcXG1hdGhjYWx7Qn0iXSxbMCwxLCJcXG1hdGhjYWx7QX0iXSxbMCwwLCJGXypcXGRvd25hcnJvdyBcXG1hdGhjYWx7Qn0iXSxbMCwxLCIiLDAseyJsZXZlbCI6Miwic3R5bGUiOnsiaGVhZCI6eyJuYW1lIjoibm9uZSJ9fX1dLFsyLDEsIkYiLDJdLFszLDIsIlxccGlfMSIsMl0sWzMsMCwiXFxjb2QiXSxbNSw0LCJcXGxhbWJkYV97Rl8qfSIsMCx7ImN1cnZlIjotMSwic2hvcnRlbiI6eyJzb3VyY2UiOjIwLCJ0YXJnZXQiOjIwfX1dXQ==
\[\begin{tikzcd}
	{F_*\downarrow \mathcal{B}} & {\mathcal{B}} \\
	{\mathcal{A}} & {\mathcal{B}}
	\arrow[""{name=0, anchor=center, inner sep=0}, Rightarrow, no head, from=1-2, to=2-2]
	\arrow[""{name=1, anchor=center, inner sep=0}, "F"', from=2-1, to=2-2]
	\arrow["{\pi_1}"', from=1-1, to=2-1]
	\arrow["\cod", from=1-1, to=1-2]
	\arrow["{\lambda_{F_*}}", curve={height=-6pt}, shorten <=4pt, shorten >=4pt, Rightarrow, from=1, to=0]
\end{tikzcd}\]
actually lies inside of $\LFP$. \\

\begin{theorem}
For $ F : \mathcal{A} \rightarrow \mathcal{B}$ in $\LFP$, $F_*\downarrow \mathcal{B}$ together with $ \cod $ and $\pi_1$ is the comma object $ F\downarrow \mathcal{B}$ in $\LFP$. 
\end{theorem}

\begin{proof}
Any other 2-cell in $\LFP$ 
% https://q.uiver.app/?q=WzAsNCxbMSwwLCJcXG1hdGhjYWx7Qn0iXSxbMSwxLCJcXG1hdGhjYWx7Qn0iXSxbMCwxLCJcXG1hdGhjYWx7QX0iXSxbMCwwLCJDIl0sWzAsMSwiIiwwLHsibGV2ZWwiOjIsInN0eWxlIjp7ImhlYWQiOnsibmFtZSI6Im5vbmUifX19XSxbMiwxLCJGIiwyXSxbMywyLCJHIiwyXSxbMywwLCJIIl0sWzUsNCwiXFxsYW1iZGEiLDAseyJjdXJ2ZSI6LTEsInNob3J0ZW4iOnsic291cmNlIjoyMCwidGFyZ2V0IjoyMH19XV0=
\[\begin{tikzcd}
	C & {\mathcal{B}} \\
	{\mathcal{A}} & {\mathcal{B}}
	\arrow[""{name=0, anchor=center, inner sep=0}, Rightarrow, no head, from=1-2, to=2-2]
	\arrow[""{name=1, anchor=center, inner sep=0}, "F"', from=2-1, to=2-2]
	\arrow["G"', from=1-1, to=2-1]
	\arrow["H", from=1-1, to=1-2]
	\arrow["\lambda", curve={height=-6pt}, shorten <=4pt, shorten >=4pt, Rightarrow, from=1, to=0]
\end{tikzcd}\]
defines from its underlying 2-cell in $\Cat$ a unique arrow $ S_\lambda$ as below
% https://q.uiver.app/?q=WzAsNSxbMiwxLCJcXG1hdGhjYWx7Qn0iXSxbMiwyLCJcXG1hdGhjYWx7Qn0iXSxbMSwyLCJcXG1hdGhjYWx7QX0iXSxbMCwwLCJDIl0sWzEsMSwiRl8qXFxkb3duYXJyb3cgXFxtYXRoY2Fse0J9Il0sWzAsMSwiIiwwLHsibGV2ZWwiOjIsInN0eWxlIjp7ImhlYWQiOnsibmFtZSI6Im5vbmUifX19XSxbMiwxLCJGXyoiLDJdLFszLDIsIkdfKiIsMix7ImN1cnZlIjozfV0sWzMsMCwiSF8qIiwwLHsiY3VydmUiOi0zfV0sWzQsMiwiXFxwaV8xIiwxXSxbNCwwLCJcXGNvZCIsMV0sWzMsNCwiU19cXGxhbWJkYSIsMSx7InN0eWxlIjp7ImJvZHkiOnsibmFtZSI6ImRhc2hlZCJ9fX1dLFs0LDcsIlxcc2ltZXEiLDEseyJzaG9ydGVuIjp7InRhcmdldCI6MjB9LCJzdHlsZSI6eyJib2R5Ijp7Im5hbWUiOiJub25lIn0sImhlYWQiOnsibmFtZSI6Im5vbmUifX19XSxbNCw4LCJcXHNpbWVxIiwxLHsic2hvcnRlbiI6eyJ0YXJnZXQiOjIwfSwic3R5bGUiOnsiYm9keSI6eyJuYW1lIjoibm9uZSJ9LCJoZWFkIjp7Im5hbWUiOiJub25lIn19fV0sWzYsNSwiXFxsYW1iZGFfe0ZfKn0iLDAseyJjdXJ2ZSI6LTEsInNob3J0ZW4iOnsic291cmNlIjoyMCwidGFyZ2V0IjoyMH19XV0=
\[\begin{tikzcd}
	C \\
	& {F_*\downarrow \mathcal{B}} & {\mathcal{B}} \\
	& {\mathcal{A}} & {\mathcal{B}}
	\arrow[""{name=0, anchor=center, inner sep=0}, Rightarrow, no head, from=2-3, to=3-3]
	\arrow[""{name=1, anchor=center, inner sep=0}, "{F_*}"', from=3-2, to=3-3]
	\arrow[""{name=2, anchor=center, inner sep=0}, "{G_*}"', curve={height=18pt}, from=1-1, to=3-2]
	\arrow[""{name=3, anchor=center, inner sep=0}, "{H_*}", curve={height=-18pt}, from=1-1, to=2-3]
	\arrow["{\pi_1}"{description}, from=2-2, to=3-2]
	\arrow["\cod"{description}, from=2-2, to=2-3]
	\arrow["{S_\lambda}"{description}, dashed, from=1-1, to=2-2]
	\arrow["\simeq"{description}, Rightarrow, draw=none, from=2-2, to=2]
	\arrow["\simeq"{description}, Rightarrow, draw=none, from=2-2, to=3]
	\arrow["{\lambda_{F_*}}", curve={height=-6pt}, shorten <=4pt, shorten >=4pt, Rightarrow, from=1, to=0]
\end{tikzcd}\]
sending any $ C$ in $\mathcal{C}$ to the corresponding component of the natural transformation $ \lambda$
% https://q.uiver.app/?q=WzAsMixbMCwwLCJGXypHXyooQykiXSxbMSwwLCJIXyooQykiXSxbMCwxLCJcXGxhbWJkYV9DIl1d
\[\begin{tikzcd}
	{F_*G_*(C)} & {H_*(C)}
	\arrow["{\lambda_C}", from=1-1, to=1-2]
\end{tikzcd}\]
But from $ F_*$, $ G_*$ and $H_*$ all are continuous and finitary, and from the computation of limits and filtered colimits in the comma, $ S_\lambda$ is itself continuous and finitary. Moreover, from uniqueness of adjoints, any two 2-cells inducing the same functor in $\LFP$ induce in fact a same factorization in $\Cat$ and must then be equivalent. 

\end{proof}

\printbibliography

\end{document}